\documentclass{amsart}

\usepackage{graphicx}
\usepackage{amssymb}

\usepackage{multido}

\usepackage[T1]{fontenc}
\usepackage{hyperref}

\hypersetup{
 colorlinks=true,       
    linkcolor=cyan,          
    citecolor=green,        
    filecolor=magenta,      
   urlcolor=cyan           
   }

\numberwithin{equation}{section}
\numberwithin{figure}{section}

\DeclareSymbolFont{bbold}{U}{bbold}{m}{n}
\DeclareSymbolFontAlphabet{\mathbbold}{bbold}
\newcommand{\ind}{\mathbbold{1}}

\theoremstyle{plain} \newtheorem{theorem}{Theorem}[section]
\theoremstyle{plain} \newtheorem{proposition}[theorem]{Proposition}
\theoremstyle{plain} \newtheorem{lemma}[theorem]{Lemma}
\theoremstyle{plain} \newtheorem{corollary}[theorem]{Corollary}
\theoremstyle{definition} \newtheorem{definition}[theorem]{Definition}
\theoremstyle{definition} \newtheorem{notation}[theorem]{Notation}
\theoremstyle{remark} \newtheorem{remark}[theorem]{Remark}
\theoremstyle{remark} \newtheorem{example}[theorem]{Example}

\newcommand{\bE}{\mathbb{E}}

\newcommand{\bH}{\mathbb{H}}

\newcommand{\bN}{\mathbb{N}}
\newcommand{\bP}{\mathbb{P}}
\newcommand{\bQ}{\mathbb{Q}}
\newcommand{\bR}{\mathbb{R}}

\newcommand{\dmin}{\partial_{\text{\rm \tiny min}}}

\def\cD{\mathcal{D}}
\def\cE{\mathcal{E}}
\def\cF{\mathcal{F}}
\def\cG{\mathcal{G}}
\def\cH{\mathcal{H}}
\def\cI{\mathcal{I}}

\def\cT{\mathcal{T}}

\def\tt{\mathbf{t}}
\def\bb{\mathbf{b}}
\def\ww{\mathbf{w}}
\renewcommand{\ss}{\mathbf{s}}

\def\DD{\mathbf{D}}

\def\SSS{\mathbf{S}}
\def\TT{\mathbf{T}}
\def\UU{\mathbf{U}}
\def\VV{\mathbf{V}}
\def\ZZ{\mathbf{Z}}

\begin{document}

\allowdisplaybreaks

\title[R\'emy's tree growth chain]{Doob--Martin boundary \\
of R\'emy's tree growth chain}

\author{Steven N. Evans}
\address{Department of Statistics\\
         University  of California\\ 
         367 Evans Hall \#3860\\
         Berkeley, CA 94720-3860 \\
         U.S.A.}

\email{evans@stat.berkeley.edu}
\thanks{SNE supported in part by NSF grant DMS-0907630, NSF grant  DMS-1512933,
and NIH grant 1R01GM109454-01. AW supported in part by DFG priority program 1590.}

\author{Rudolf Gr\"ubel}
\address{Institut f\"ur Mathematische Stochastik\\ 
         Leibniz Universit\"at Hannover\\
         Postfach 6009\\ 
         30060 Hannover\\
         Germany}

\email{rgrubel@stochastik.uni-hannover.de}

\author{Anton Wakolbinger}
\address{Institut f\"ur Mathematik \\
         Goethe-Universit\"at \\
         60054 Frankfurt am Main\\
         Germany}

\email{wakolbin@math.uni-frankfurt.de}

\subjclass[2000]{Primary 60J50, secondary 60J10, 68W40}

\keywords{binary tree,  
tail $\sigma$-field, 
Doob--Martin compactification, 
Poisson boundary,  
bridge,
real tree,
exchangeability,
continuum random tree,
Catalan number, 
graph limit,
graphon,
partial order}

\date{\today}

\begin{abstract}
R\'emy's algorithm is a Markov chain that iteratively generates
a sequence of random  trees 
in such a way that the $n^{\mathrm{th}}$ tree is
uniformly distributed over the set of rooted, planar, binary trees with
$2n+1$ vertices.  We obtain a concrete characterization of the
Doob--Martin boundary of this transient Markov chain and thereby
delineate all the ways in which, loosely speaking, this process
can be conditioned to ``go to infinity'' at large times.
A (deterministic) sequence of finite rooted, planar,
binary trees converges to a point 
in the boundary if for each $m$ the random rooted, planar, binary
tree spanned by $m+1$ leaves
chosen uniformly at random from the $n^{\mathrm{th}}$ tree in the
sequence converges in distribution as $n$ tends to infinity -- a
notion of convergence that is analogous to one that appears in
the recently developed theory of graph limits.

We show that a point in the Doob--Martin boundary may be identified
with the following ensemble of objects:
a complete separable $\bR$-tree that is rooted and binary
in a suitable sense, a diffuse probability measure on the $\bR$-tree
that allows us to make sense of sampling points from it,
and a kernel on the $\bR$-tree
that describes the probability that the first of a given
pair of points is below and to the left of their most recent common
ancestor while the second is below and to the right.  
Two such ensembles represent the same point in the
boundary if for each
$m$ the random, rooted, planar, binary trees spanned by $m+1$ independent
points chosen according to the respective probability measures have
the same distribution.  Also, the Doob--Martin boundary
corresponds bijectively to the set of extreme point of the closed convex set 
of nonnegative harmonic functions that take the value $1$ at the binary tree with
$3$ vertices; in other words, the minimal and full Doob--Martin
boundaries coincide.  These results are in the spirit of
the identification of graphons as limit objects in the theory of graph limits.
\end{abstract}

\maketitle
\tableofcontents

\section{Introduction}
\label{S:intro}

R\'emy's algorithm \cite{MR803997} iteratively generates
a sequence of random binary trees 
$T_1, T_2, \ldots$ in a Markovian manner in such a way that
$T_n$ is uniformly distributed on the set of binary trees with
$2n+1$ vertices (see \cite{MR1331596} for a textbook
discussion of this procedure).
Here (and throughout this paper) a {\em binary tree} is a finite rooted tree 
in which every vertex has zero or two children and the tree
is {\em planar}, so it is possible
to distinguish between the {\em left} and {\em right}
children of a vertex with two children.  

A binary tree has $2n+1$ vertices for some $n \in \bN$:
$n+1$ leaves and $n$ interior vertices. 
The number of binary trees with $2n+1$ vertices is the 
{\em Catalan number} $C_n := \frac{1}{n+1} \binom{2n}{n}$ \cite{MR1442260}.

Writing $\{0,1\}^\star:=\bigsqcup_{k=0}^\infty\{0,1\}^k$ for 
the set of finite words drawn from the alphabet $\{0,1\}$ 
(with the empty word $\emptyset$
allowed), any binary tree can be identified with a unique
finite subset  $\tt \subset \{0,1\}^\star$ 
that has the properties:
\begin{itemize}
\item
$v_1 \ldots v_k \in \tt \Longrightarrow v_1 \ldots v_{k-1} \in \tt$,
\item
$v_1 \ldots v_k 0 \in \tt \Longleftrightarrow v_1 \ldots v_k 1 \in \tt$.
\end{itemize}
The empty word $\emptyset \in \{0,1\}^\star$ is the root of the tree.
See Figure~\ref{fig:binary_tree} for an example of this representation.
\begin{figure}[h]
		\includegraphics[height=4.5cm]{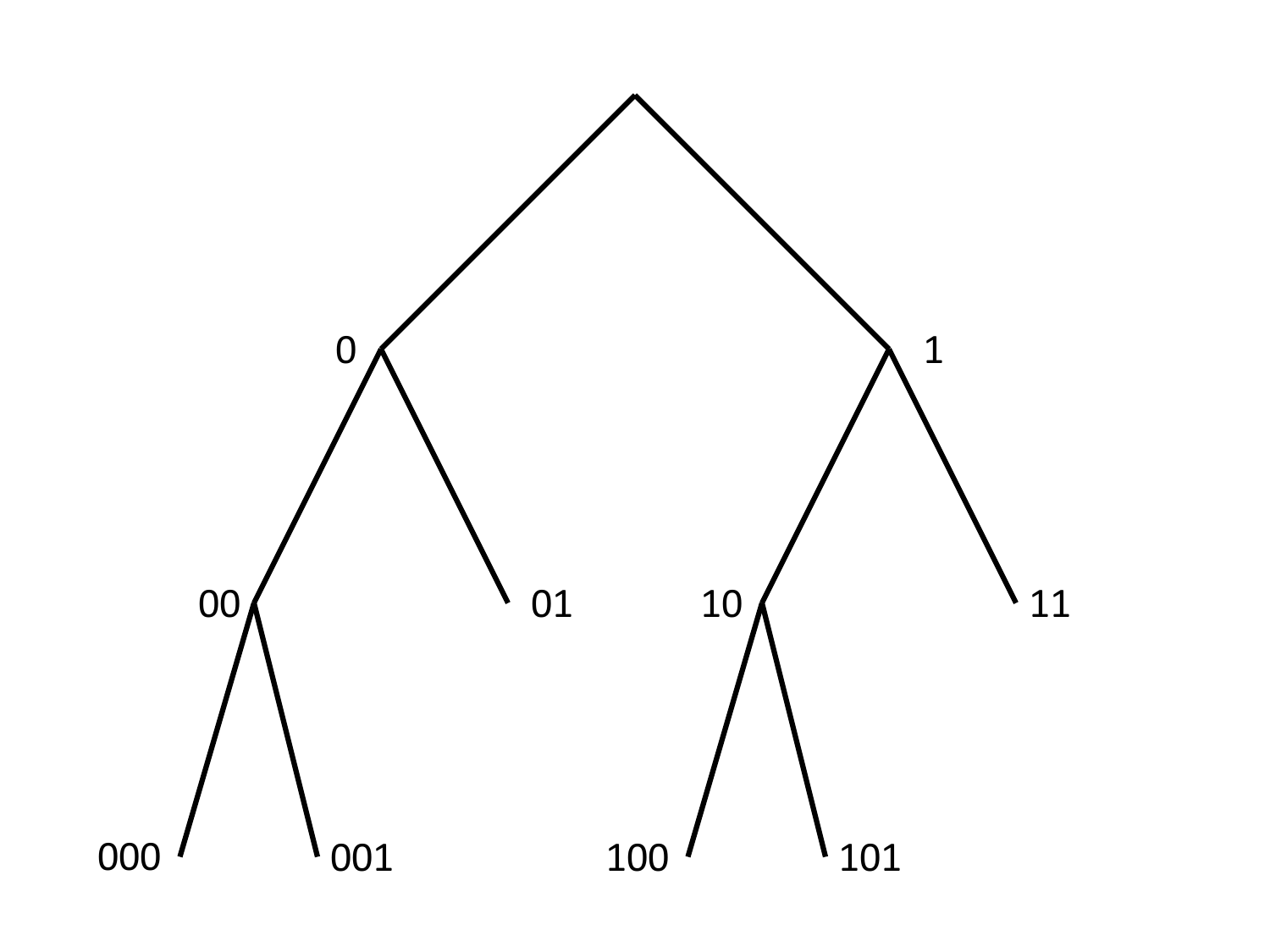}
	\caption{\small An example of a binary tree as a subset of $\{0,1\}^\star$.}
	\label{fig:binary_tree}
\end{figure}
R\'emy's algorithm begins by setting $T_1$
to be the unique binary tree
$\aleph$ with $3$ vertices (a root
and two leaves).  Supposing that $T_1, \ldots, T_n$
have been generated,
the algorithm generates $T_{n+1}$
as follows (see Figure~\ref{fig:Remy_move_1}, Figure~\ref{fig:Remy_move_2},
and Figure~\ref{fig:Remy_move_3} for a depiction of the steps that make up
a single iteration of the algorithm).
\begin{itemize}
\item
Pick a
vertex $v$ of $T_n$ uniformly at random.
\item
Cut off the subtree of $T_n$ rooted at $v$ and set it aside.
\item
Attach a copy of the tree $\aleph$ with $3$ vertices to the end of the edge
in $T_n$ that previously led to $v$.
\item
Re-attach the subtree that was rooted at $v$ in $T_n$ uniformly at random
to one of the two leaves in the copy of $\aleph$.
\end{itemize}
We call the two new vertices that are produced in the above iteration
{\em clones} of $v$. 

\begin{figure}[ht]
		\includegraphics[width=6.5cm]{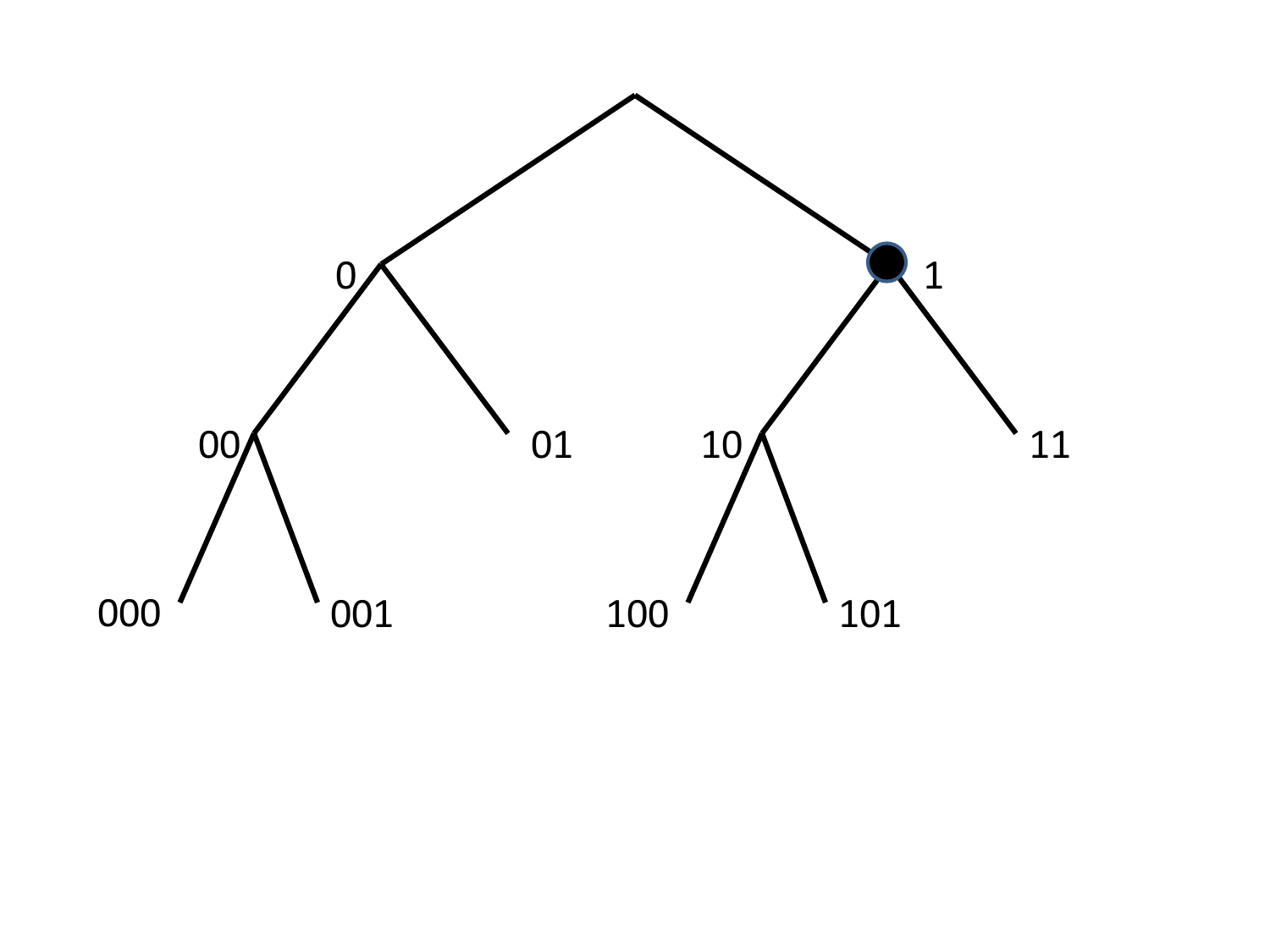}
	\vspace{-1cm}\caption{\small First step in an iteration of R\'emy's algorithm: pick a  vertex $v$ uniformly at random.}
	\label{fig:Remy_move_1}
\end{figure}
\begin{figure}[ht]
	\centering
		\includegraphics[width=6.5cm]{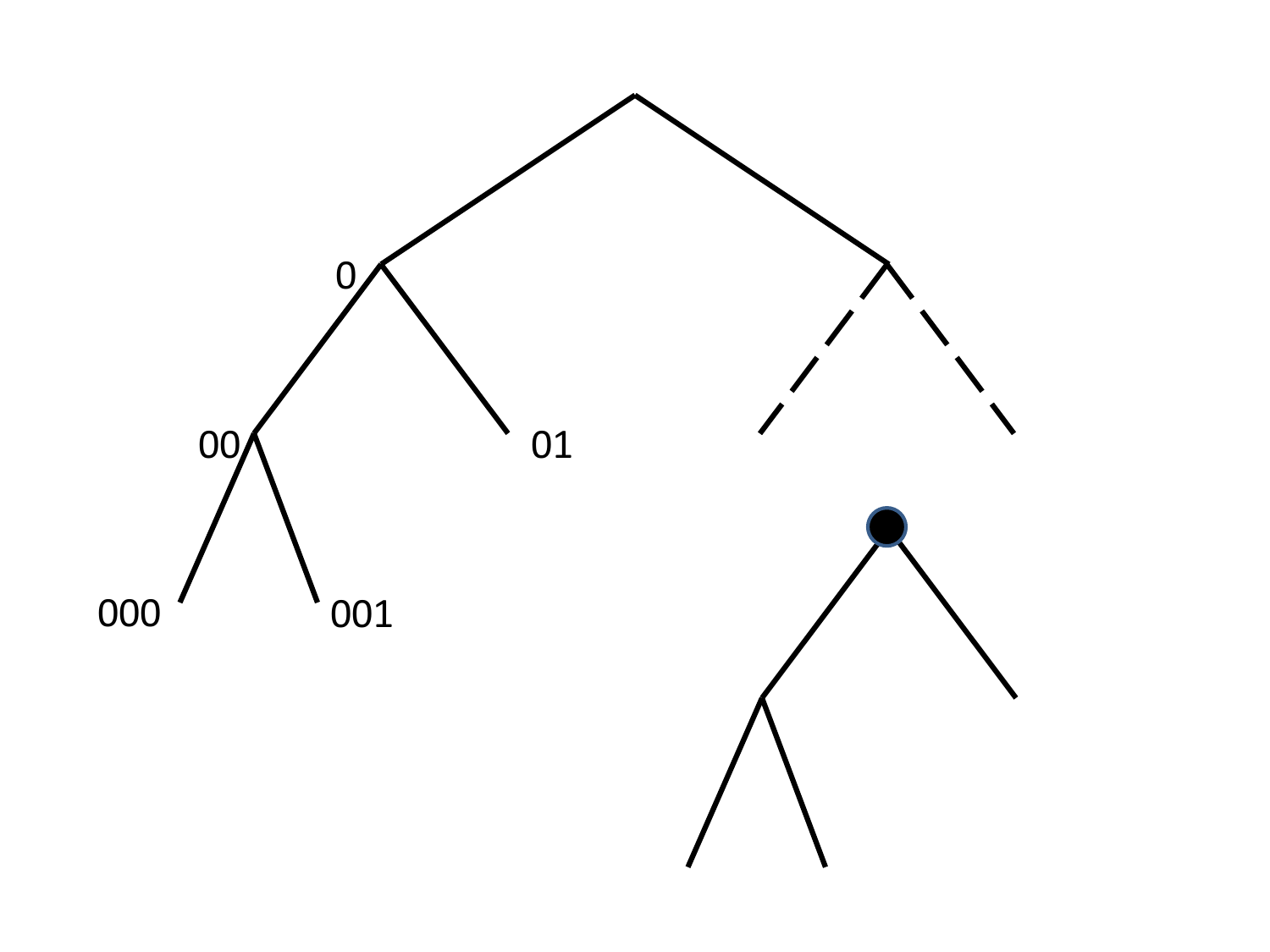}
	\caption{\small Second step in an iteration of R\'emy's algorithm: 
	cut off the subtree rooted at $v$
	and attach a copy of $\aleph$ to the end of the
	edge that previously led to $v$.}
	\label{fig:Remy_move_2}
\end{figure}
\begin{figure}[ht]
	\centering
		\includegraphics[width=6.5cm]{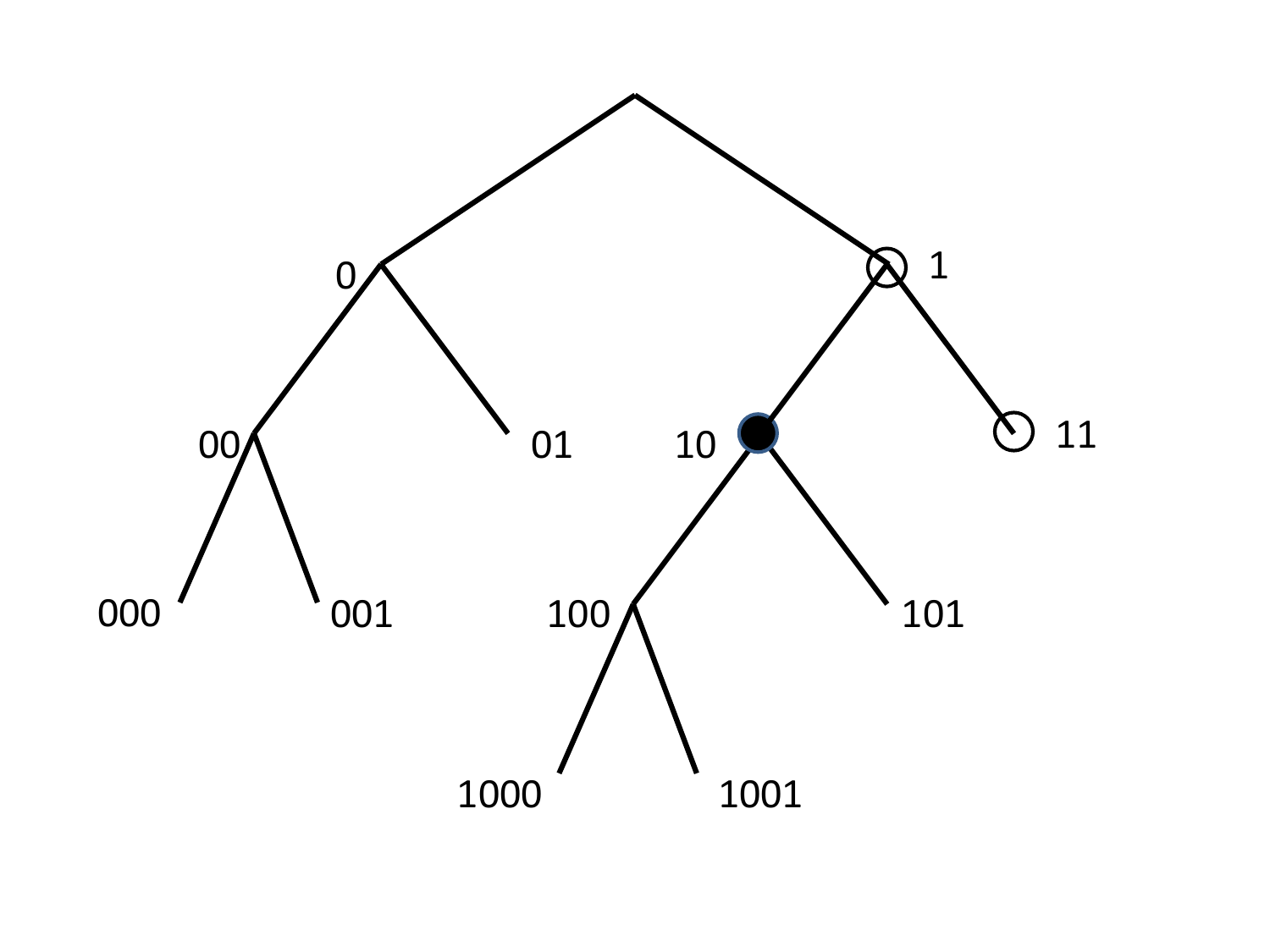}
	\caption{\small Third step in an iteration of R\'emy's algorithm: 
	re-attach the subtree rooted at $v$ 
	to one of the two leaves of the copy of $\aleph$, and re-label the vertices
	appropriately.  The solid circle 
	is the new location of $v$ and
	the open circles are the  clones of $v$.}
	\label{fig:Remy_move_3}
\end{figure}

It is not too difficult to see that the algorithm does 
produce uniformly distributed binary trees.  Indeed,
suppose that the algorithm is modified so that it
starts with the leaves of $\aleph$ labeled with $1$ and $2$,
with each of the two labelings being equally likely,
a random leaf--labeled tree that we denote by $\tilde T_1$.   
Suppose further that we begin 
the $(n+1)^{\mathrm{st}}$ step with a leaf--labeled 
binary tree $\tilde T_n$ that has
$n+1$ leaves labeled with $[n+1] := \{1, \ldots, n+1\}$ in some order and that this step
produces a random leaf--labeled binary tree 
$\tilde T_{n+1}$ labeled with $[n+2]$ as follows.
\begin{itemize}
\item
Use the same randomization as in the algorithm described above to produce
a tree with a single new leaf.
\item
Leave the labels of the old leaves unchanged.
\item
Label the new leaf with $n+2$.  
\end{itemize}
It will certainly
suffice to show that this enhanced algorithm produces a sequence
$\tilde T_1, \tilde T_2, \ldots$ such that for all $n \in \bN$
the random leaf--labeled binary tree $\tilde T_n$ is uniformly
distributed on the set of binary trees with $2n+1$ vertices that have
their $n+1$ leaves labeled by $[n+1]$.  This, however,
is almost immediate from an inductive argument and the observation that
in order for the value of $\tilde T_{n+1}$ 
to be a particular labeled binary tree,
there is a unique possibility for 
the value of $\tilde T_n$, 
the choice of vertex
$v$ to clone, 
and the left-right choice for re-attaching the 
subtree below $v$; see \cite{MR803997, MR1331596} for more details.

Following \cite{MR803997, MR1331596}, 
we note that this argument also shows that if we
let $p_{n}$ be the common value of $\bP\{\tilde T_n = \tilde \tt\}$
as $\tilde \tt$ ranges over the binary trees with $2n+1$ vertices and their  $n+1$ leaves labeled by
$[n+1]$, then $p_{n+1} = \frac{1}{2n+1} \frac{1}{2} p_n$, so that
$p_n = 
\frac{1}{1 \times 3 \times \cdots \times (2n - 1)}\frac{1}{2^n}$.  
It follows that the number of binary trees with $2n+1$ vertices and their $n+1$ leaves labeled by
$[n+1]$ is
\[
(1 \times 3 \times \cdots (2n - 1)) 2^n
=
\frac{(2n)!}{n!},
\]
and so the number of binary trees with $2n+1$ vertices and $n+1$ leaves is
\[
\frac{(2n)!}{(n+1)! n!}
= 
C_n,
\]
as expected.

As well as counting the number of binary trees with 
$2n+1$ vertices for $n \in \bN$,
the Catalan number $C_n$ counts the number of functions 
$f:\{0,1,\ldots,2n\} \to \bN_0$
such that $f(0) = f(2n) = 0$ and $f(k+1) = f(k) \pm 1$ for $0 \le k < 2n$.
It is shown in \cite{MR2042386} that there are particular bijections 
$\phi_n$ (sometimes credited to {\L}ukasiewicz or
Dwass) between the former and latter sets such that if
$f_n := \phi_n(T_n)$, then 
$(n^{-\frac{1}{2}} f_n(\lfloor 2nt\rfloor))_{t \in [0,1]}$ converges almost
surely in the supremum norm to a standard Brownian excursion
A similar result is given in \cite[Exercise 7.4.11]{MR2245368}
where $T_n$ is represented as a function from $\{0, 1, \ldots , 4n\}$
to $\bN_0$ using a coding where one ``walks around the outside'' of the tree
visiting left children before right children 
(so that each edge is traversed twice, leaves are visited once,
and other vertices are visited three times),
and recording the distance from the
root to the vertex visited at each step 
-- this coding, or a minor modification of it, is sometimes called the
Harris path of the tree.

The standard Brownian excursion
induces a random metric space that is, up to a scaling factor, Aldous' Brownian
continuum random tree (CRT) 
\cite{MR1085326, MR1166406, MR1207226}.  
More precisely, if $(E_t)_{t \in [0,1]}$ is the
standard Brownian excursion, then 
$d(s,t) := E_s + E_t - 2 \min_{u \in [s,t]} E_u$, $s,t \in [0,1]$ defines
a pseudo-metric on $[0,1]$ that becomes a metric on the collection of
equivalence classes for the equivalence relation 
$s \equiv t \Leftrightarrow d(s,t) = 0$, and the latter random metric space
is a random $\bR$-tree that is, by definition, a scaled version of
the Brownian CRT (see \cite{MR2351587} for a general treatment
of $\bR$-trees directed at probabilists). This definition carries with it
a natural rooting and hence a natural genealogical structure:
the  most recent common ancestor of the equivalence classes containing
$s$ and $t$ is the equivalence class of the
almost surely unique $v \in [s,t]$ such that 
$E_v = \min_{u \in [s,t]} E_u$.  The Brownian CRT with this
rooting is almost surely binary
in the sense that almost surely for all $r,s,t \in [0,1]$ 
coming from distinct equivalence classes the most recent common ancestors of
the pairs $(r,s), (r,t), (s,t)$ are not all equal.  Moreover, this construction
also endows the Brownian CRT with a natural planar structure: for 
$s, t \in [0,1]$ coming from distinct equivalence classes, the equivalence
class containing $s$ may be consistently declared to be
below and to the left of the most recent common ancestor 
of the two equivalence classes (and the equivalence
class containing $t$ is below and to the right) if 
$\min\{q : \min_{u \in [q,s]} E_u = E_s\} 
< \min\{r : \min_{v \in [r,t]} E_v = E_t\}$
(in other words, the time parameter in the Brownian excursion induces
a traversal of the points of the Brownian CRT that starts and
ends at the root, and we say that one equivalence class
is below and to the left of the most recent common ancestor
it shares with another equivalence class if this traversal
encounters the former equivalence class before the latter).

Conversely, it is observed in \cite{MR1714707}
that if one samples i.i.d.\ uniformly distributed points $U_0, U_1, \ldots$
from $[0,1]$ and lets $\hat T_n$ be the binary tree spanned by the equivalence
classes of $U_0, \ldots, U_n$ for $n \in \bN$ (more fully, one takes the
subtree of the re-scaled Brownian CRT thought of as a random $\bR$-tree
but equipped with the additional rooting and left--right ordering
described above, forgets the metric structure on the subtree, but keeps
the rooting and left--right ordering), then $(\hat T_n)_{n \in \bN}$
has the same distribution as $(T_n)_{n \in \bN}$; that is,
$(\hat T_n)_{n \in \bN}$ is an instance of R\'emy's chain.

As we shall explain soon,
these last two results and several more
are parallels of classical results
about the simplest P\'olya urn scheme in which one starts
with an urn containing one black and one white ball and at each step
one picks a ball uniformly at random and replaces it along
with another of the same color.

If we write $N_n$ for the number
of new black balls that have been added to the urn up to 
and including the $n^{\mathrm{th}}$ step of the P\'olya
urn chain, then $((N_n, n - N_n))_{n \in \bN}$
is a Markov chain with the following properties.  For each $n \in \bN$
the random variable $N_n$ is uniformly distributed on $\{0,1,\ldots,n\}$
and $N_n/n$ converges almost surely as $n \to \infty$ to a random variable
$U$ that is uniformly distributed on the interval $[0,1]$.
If $(X_n)_{n \in \bN}$ is a sequence of $\{0,1\}$-valued random variables
that are conditionally independent given $U$ with 
$\bP\{X_n = 1 \, | \, U\} = U$, then \cite{MR0362614} 
$(N_n)_{n \in \bN}$ has the same
distribution as $(X_1 + \cdots + X_n)_{n \in \bN}$. 
It follows from this observation and the Hewitt--Savage
zero--one law that the tail $\sigma$-field of the 
P\'olya urn chain is generated up to null sets by the random variable $U$.
By the martingale convergence theorem,
the vector space of bounded harmonic functions for the P\'olya urn chain
(that is, the {\em Poisson boundary} of the chain)
can thus be identified with $L^\infty$ of the unit interval 
equipped with Lebesgue measure.
Another consequence is the well-known fact that the colors
of the successive balls form an exchangeable sequence and so
the backward dynamics of the P\'olya urn chain from step $n$
to step $n-1$ can be thought of as removing one of the $n$ added balls
present at step $n$ uniformly at random and discarding it.

We will show that the backward transitions of the R\'emy chain are as follows.
\begin{itemize}
\item
Pick a leaf uniformly at random.
\item
Delete the chosen leaf and its sibling (the sibling may or may not be a leaf).
\item
Close up the gap if there is one (there will be a gap if the
sibling is not a leaf).
\end{itemize} 
Note how these dynamics are  reminiscent of the backward transitions of the
P\'olya urn chain.
It is a consequence of the exchangeability inherent in these dynamics 
and the Hewitt--Savage zero--one law that the tail $\sigma$-field of
the R\'emy chain is generated up to null sets by the limiting
Brownian CRT augmented by the additional
rooting and left--right ordering described above. More precisely, we
may assume that the entire R\'emy chain has been 
built from a Brownian excursion
(equivalently, the augmented Brownian CRT) 
and an independent, identically distributed sequence 
$(U_k)_{k \in \bN_0}$ of random variables that are each uniformly distributed
on $[0,1]$. If the first $n+1$ of these random variables are 
permuted in any way, then the values of the R\'emy chain 
from time $n$ onwards are unchanged, and so the Hewitt--Savage 
zero--one law gives that the tail $\sigma$-field of the 
R\'emy chain is, up to null sets, contained in the 
$\sigma$-field generated by the augmented Brownian CRT.  
Conversely, since one can build the Brownian CRT as an almost sure limit 
(as $n \to \infty$) of the rescaled R\'emy chain, 
the tail $\sigma$-field is equal to the 
$\sigma$-field generated by the augmented Brownian CRT up to null sets. 
Hence the Poisson boundary of the R\'emy chain can
be identified with $L^\infty$ of a space of 
suitably defined ``rooted, planar, binary''
$\bR$-trees equipped with
the distribution of the augmented Brownian CRT or, equivalently,
with $L^\infty$ of the space of continuous excursion paths indexed by
$[0,1]$ equipped with the standard Brownian excursion measure.

The R\'emy chain is not the only Markov chain which at step
$n$ produces uniformly distributed binary trees with
$2n+1$ vertices. Another example is the Markov chain proposed 
in \cite{MR2060629} which, unlike the R\'emy chain, has
the property that the state at time $n$ is a subtree
of the state at time $n+1$ for all $n \in \bN$. 
The Poisson boundary of this chain, 
which was described in \cite{MR2869248}, 
turns out to be quite different from that of the R\'emy chain.

The object of the present paper is to go further 
and investigate the {\em Doob--Martin compactification} 
of the R\'emy chain. Before giving a 
formal definition of the Doob--Martin compactification 
in Section \ref{S:Martin_general}, 
let us illustrate the concept with the archetypal example 
of the P\'olya urn chain.
Given $(\bb,\ww) \in (\bN_0 \times \bN_0) \setminus \{0,0\}$, let
$N_1^{(\bb,\ww)}, \ldots, N_{\bb+\ww}^{(\bb,\ww)}$ 
be the bridge process obtained by conditioning the initial segment
$N_1, \ldots, N_{\bb+\ww}$ of the P\'olya urn chain on the event
$\{N_{\bb+\ww} = \bb\}$.  The backward transitions of such a
bridge are the same as those of the P\'olya urn chain itself
and it is not hard to show that if $((\bb_k,\ww_k))_{k \in \bN}$ is
a sequence such that $\bb_k + \ww_k \to \infty$ as $k \to \infty$, then
the finite-dimensional distributions of the corresponding bridges converge
if and only if $\lim_{k \to \infty} \frac{\bb_k}{\bb_k + \ww_k} \in [0,1]$
exists.  It is a classical result \cite[Chapter 10]{MR0407981} 
that a sequence 
$((\bb_k,\ww_k))_{k \in \bN}$ such that 
$\bb_k + \ww_k \to \infty$ as $k \to \infty$
converges in the Doob--Martin compactification  of the P\'olya urn chain
if and only if 
$\lim_{k \to \infty} \frac{\bb_k}{\bb_k + \ww_k} \in [0,1]$
exists and, as we recall in Section~\ref{S:Martin_general}, 
a general result from 
\cite{MR0426176} establishes the equivalence between convergence
of bridges and convergence in the Doob--Martin compactification
under suitable conditions.  It follows that the Doob--Martin boundary
of the P\'olya urn chain is (homeomorphic to) the unit interval $[0,1]$.
There is thus a nonnegative harmonic function associated with each
point $u \in [0,1]$ and the corresponding Doob $h$-transform process
can be interpreted as $(N_n)_{n \in \bN}$ conditioned on the event
$\{U=u\}$.  As one would expect, the distribution of the 
Doob $h$-transform process is nothing other than that of
the process of partial sums
of independent, identically distributed Bernoulli random variables
with success probability $u$.

We will investigate the bridges of the R\'emy chain and thereby identify 
its Doob--Martin boundary. This boundary
of the space of (finite) binary trees determines the
compact convex set of nonnegative harmonic functions 
normalized to take the value $1$ at the binary tree with two leaves.
We show that the set of extreme points of the the latter compact convex set
corresponds bijectively to the Doob--Martin boundary, and hence
the boundary delineates all the ways that the R\'emy chain
can be conditioned to ``behave at infinity'' in such a way that any randomness 
disappears asymptotically in the sense that the tail $\sigma$-field of the
conditioned chain is trivial.

We will show that a sequence of finite binary trees with 
the number of vertices going to infinity converges in the Doob--Martin
topology if and only if for all $m$ the sequence of random binary trees
spanned by $m+1$ leaves sampled uniformly at random from  those of
the $n^{\mathrm{th}}$ tree in the original sequence
converges in distribution as $n \to \infty$.  Moreover,
two convergent sequences converge to the same limit if and only if
the corresponding limit distributions of these ``sampled subtrees''
are the same for all $m$.  (The analogous fact is also true
for the P\'olya urn: a sequence of states converges
in the Doob--Martin topology if and only if for any $m$ when we sample $m$
balls uniformly at random from the urn composition specified
by the $n^{\mathrm{th}}$ state, 
the distribution of the number of black balls 
in the sample converges as $n \to \infty$.) 
This type of convergence of a sequence of
large combinatorial objects in terms of the convergence in distribution
of randomly sampled sub-objects of a given but arbitrary size is
similar to a notion of convergence of finite graphs investigated
in the theory of graph limits where a sequence
of graphs with increasing numbers of vertices converges if for each $m$
the distributions of the random finite graphs induced
by $m$ vertices sampled uniformly at random converge 
(see \cite{MR3012035} for a recent
monograph and \cite{MR2426176, MR2594615, MR2455626, MR2925382, 
MR2463439, MR2274085, MR2373263} for some examples of papers
in this area).  
A binary tree encodes two partial orders
on its set of vertices (one vertex can be below and to the left
(respectively, right) of another vertex), and so the work in \cite{MR2886098} 
on limits of large partially ordered sets is particularly close
in spirit to our work.  A further connection between our
work and graph limits is the result of \cite{dmtcs:1289}
that the above notion of graph convergence is nothing other than
convergence in the Doob--Martin topology for the graph-valued
{\em Erd\H{o}s--R\'enyi  chain} in which at each step an
additional vertex is added with the possible edges 
connecting it to each of the existing vertices
independently present with probability $p$ and absent with
probability $1-p$ for some fixed $0 < p < 1$ (the Doob--Martin compactification
does not depend on the value of $p$).

One of the major achievements of the
theory of graph limits has been to obtain concrete representations
of the limit objects corresponding to a convergent sequence of graphs
as so-called {\em graphons}. A graphon is a symmetric Borel function
$K:[0,1]^2 \to [0,1]$ and a random graph with the distribution
of the limit of the randomly sampled subgraphs of size $m$ corresponding to
a convergent sequence of graphs is obtained by choosing
$m$ points $U_1, \ldots, U_m$ uniformly at random
from $[0,1]$ and connecting vertex $i$ and $j$ with
conditional probability $K(U_i, U_j)$.  

In our main result, Theorem~\ref{T:main}, 
we obtain a similar concrete representation
of a point in the Doob--Martin boundary of the R\'emy chain 
as a rooted $\bR$-tree $\SSS$ equipped with a probability measure $\mu$ and
a function $V: \SSS^2 \to [0,1]$.  The limit in distribution
of the subtrees spanned
by $m+1$ uniformly chosen leaves is obtained by, loosely speaking,
looking at the subtree of $\SSS$ spanned by independent random points
$\xi_1, \ldots, \xi_{m+1}$ with distribution $\mu$ and declaring that 
with probability $V(\xi_i, \xi_j)$ leaf $i$ is below
and to the left while $j$ is below and to the right of the
most recent common ancestor of leaves $i$ and~$j$.  
Like all transient Markov chains, 
the R\'emy chain has the property that
$T_n$ converges almost surely as $n \to \infty$ in the Doob--Martin topology to
a random element of the Doob--Martin boundary.  The distribution
of the limit may be identified with that of the augmented
Brownian CRT described above: the underlying $\bR$-tree and
its root come from the Brownian excursion, and the probability
measure on the $\bR$-tree is the one lifted by the Brownian excursion
from Lebesgue measure on $[0,1]$. In this case the function $V$ takes
values in the set $\{0,1\}$ and is determined by the left--right
ordering coming from the Brownian excursion.  We will see
that it is not always possible to have the left--right
ordering be induced from one on the underlying $\bR$-tree
$\SSS$ and that cases do arise where it is necessary to
work with functions $V$ that take values strictly between $0$
and $1$.

Briefly, the strategy of the proof of Theorem \ref{T:main} 
will be as follows. 

(i) First we determine the  
backward transition dynamics of the R\'emy chain in 
Section ~\ref{S:infinite_bridges}. 
Understanding the Doob--Martin compactification
is equivalent to understanding all Markov chains 
with initial state $\aleph$ that have these backward transition dynamics.
We call any such chain $(T_n^\infty)_{n \in \bN}$
an {\em infinite R\'emy bridge}: the class
of infinite R\'emy bridges corresponds bijectively with the class of Doob $h$-transforms of the
original R\'emy chain as $h$ ranges over the 
nonnegative harmonic functions for the original chain normalized
so that $h(\aleph) = 1$.  This class of nonnegative harmonic functions
is a compact convex set, and an arbitrary such function has
a unique representation as an integral over extremal elements.  For a general Markov chain,
an extremal nonnegative harmonic function corresponds to a point in the Doob--Martin boundary, but
there may be points in the Doob--Martin boundary that
correspond to harmonic functions which are not extremal.  We show that this
is not the case for the R\'emy chain, and it follows that
the elements of the  Doob--Martin boundary of the R\'emy chain 
correspond bijectively to the infinite R\'emy bridges that are {\em extremal}
in the sense that they are not nondegenerate mixtures 
of infinite R\'emy bridges (equivalently, to the infinite R\'emy bridges 
that have trivial tail $\sigma$-fields).  

(ii) A key tool for obtaining a concrete characterization of
the extremal infinite R\'emy bridges will be the introduction of an auxiliary labeling
of the $n+1$ leaves of the tree $T_n^\infty$ by $[n+1]:= \{1,\ldots,n+1\}$ 
that has the properties that the labeling is uniformly distributed over the
$(n+1)!$ possible labelings for each $n$ and the new leaf
added at step $n+1$ is labeled with $n+2$ while the other leaves
keep the labels they had at step $n$.  
Such a labeling scheme is ``projective'' in the sense that the leaf--labeled
subtree of $T_{m+n}^\infty$ spanned by the leaves with labels
in $[m+1]$ coincides with the leaf--labeled version of $T_m^\infty$; 
more precisely, $T_m^\infty$ embeds into $T_{m+n}^\infty$ in the manner
defined in Section~\ref{DMker} via an injective map from the
vertices of $T_m^\infty$ into the vertices of $T_{m+n}^\infty$
that maps leaves to leaves in such a way that the image of the leaf
labeled $k$ in $T_m^\infty$ is mapped to the leaf labeled
$k$ in $T_{m+n}^\infty$ for $k \in [m+1]$. As we observe in
Section~\ref{S:N_tree}, for any
$i,j,k \in [m+1]$ there are twelve possibilities for
how the leaves labeled $i,j,k$ in the tree $T_m^\infty$  
sit in relation to each other; for example, one possibility
is that the most recent common ancestor of $i$ and $j$ is
a descendant of the most recent common ancestor of
$i$ and $k$ 
which is also the most recent common ancestor of
$j$ and $k$,
that $i$ is below and to the left 
and $j$ is below and to the right
of their most recent common ancestor,
that $i$ is below and to the left 
and $k$ is below and to the right
of their most recent common ancestor, 
and
that $j$ is below and to the left 
and $k$ is below and to the right
of their most recent common ancestor. 
Moreover, knowing which of these possibilities
holds for each triple $i,j,k$ uniquely determines 
the tree $T_m^\infty$ and its leaf labels.  A key feature of this labeling is
that the relative positions of the leaves labeled $i,j,k$ in the tree
$T_m^\infty$ is the same as the relative positions of the leaves
labeled $i,j,k$ in the tree $T_{m+n}^\infty$.
Because of this consistency there is a well-defined random
array indexed by $\{(i,j,k) \in \bN^3: i,j,k \; \text{distinct}\}$
that for any indices $(i,j,k)$ records for all $m$ such that
$\{i,j,k\} \subseteq [m+1]$ which of the twelve possibilities
holds for the relative positions of the leaves labeled $i,j,k$
in the tree $T_m^\infty$. This random array is jointly exchangeable.
It is possible to reconstruct the entire leaf--labeled version of
the infinite R\'emy bridge $(T_n^\infty)_{n \in \bN}$ from this array, and hence
the infinite R\'emy bridge itself by then simply discarding the leaf labels.
The infinite R\'emy bridge
is extremal if and only if this jointly exchangeable random array is ergodic
in the usual sense for jointly exchangeable random arrays. 

(iii) In Sections 
\ref{S:R_tree}, \ref{S:probmeas} and \ref{S:leftright}
we use ideas related to those in \cite{MR3851828, MR3112436}
and the Aldous--Hoover--Kallenberg theory of jointly exchangeable random arrays
to obtain a concrete description of the jointly exchangeable random arrays that
can arise from extremal infinite R\'emy bridges, and
it is the ingredients in this description that
appear in our above sketch of the statement of Theorem~\ref{T:main}. 
The $\{0,1\}$-valued random variables $W(\xi_i, U_i, \xi_j, U_j)$ figuring 
in the actual statement of Theorem~\ref{T:main}  indicate
whether leaf $i$ is below and to the left while $j$ is below and to the right 
of the most recent common ancestor of leaves $i$ and $j$,
with the above-mentioned $V(\xi_i, \xi_j)$ as the corresponding probabilities.

\bigskip
\noindent
{\bf NOTE: After this paper was published (see \cite{MR3601650}) we became aware of two points
that required some correction.  Neither of these affects the validity of the succeeding results.  We identify
the gaps at the places they appear by footnotes in Section \ref{S:N_tree} and show
how they can be repaired in Section~\ref{S:corrigenda}.}
\bigskip

\section{Background on Doob--Martin compactifications}
\label{S:Martin_general}

We restrict the following sketch of Doob--Martin 
compactification theory for discrete time Markov chains
to the situation of interest in the present paper. The
primary reference is \cite{MR0107098}, but useful reviews
may be found in 
\cite[Chapter 10]{MR0407981},
\cite[Chapter 7]{MR0415773}, 
\cite{MR1463727},
\cite[Chapter 7]{MR2548569},
\cite[Chapter III]{MR1796539}.

Suppose that  $(X_n)_{n \in \bN_0}$ is a discrete time Markov chain
with countable state space $E$ and transition matrix $P$.  
Suppose in addition that $E$ can be partitioned
as $E = \bigsqcup_{n \in \bN_0} E_n$, where
$E_0 = \{e\}$ for some distinguished state $e$, 
each set $E_n$ is finite, and the transition matrix
$P$ is such that $P(k,\ell) = 0$ unless $k \in E_n$ and 
$\ell \in E_{n+1}$ for some $n \in \bN_0$.
Define the {\em Green kernel} or {\em potential kernel} $G$ of
$P$ by 
\[G(i,j) := \sum_{n=0}^\infty P^n(i,j) 
= \bP^i\{X_n = j \; \text{for some $n \in \bN_0$}\}
=:\bP^i\{\text{$X$  hits $j$}\},
\] 
$i,j \in E$,
and assume that $G(e,j) > 0$ for all $j \in E$, so
that any state can be reached with positive probability starting from $e$.
The R\'emy chain belongs to this class.  The
state space $E$ of the R\'emy chain is the set of all binary trees, 
the distinguished state $e$ is the binary tree $\aleph$ with $3$ vertices,
and $E_n$ is the set of binary trees with $2n+3$ vertices.

If $Z$ is a $\bP^e$-a.s. bounded random variable 
that is measurable with respect to the tail $\sigma$-field of
$(X_n)_{n \in \bN_0}$, then
$\bE^e[Z \, | \, X_0, \ldots, X_n] = h(X_n)$ for some bounded harmonic function
$h$; that is $\sum_{j \in E} P(i,j) h(j) = h(i)$ for $i \in E$.  
By the martingale convergence theorem,
$\lim_{n \to \infty} h(X_n) = Z$ $\bP^e$-a.s.  
Conversely, if $h$ is a bounded harmonic function, then
$\lim_{n \to \infty} h(X_n)$ exists $\bP^e$-a.s. 
and the limit random variable is $\bP^e$-a.s. equal to a random variable
that is measurable with respect to the tail $\sigma$-field of 
$(X_n)_{n \in \bN_0}$.  

In order to characterize the bounded harmonic functions (and hence
the tail $\sigma$-field), it certainly suffices to determine what 
the nonnegative harmonic functions are.  The key to doing so is the
introduction of the {\em Doob--Martin kernel with reference state $e$} given by
\[ 
K(i,j) := \frac{G(i,j)}{G(e,j)} = 
\frac{\bP^i\{\text{$X$  hits $j$}\}}{\bP^e\{\text{$X$  hits $j$}\}}.
\] 
Observe that
\[
\sum_{j \in E} P(i,j) K(j,k)
=
K(i,k), \quad i \ne k,
\]
and so the function $K(\cdot,k)$ is, in some sense, ``almost harmonic'' and
becomes closer to being harmonic as $k$ ``goes to infinity''.
With this intuition in mind, 
it is natural to investigate sequences
$(j_n)_{n \in \bN}$ in $E$ such that the sequence of real numbers
$(K(i,j_n))_{n\in\bN}$ converges for all $i\in E$.

These considerations lead to the following construction.
If $j,k \in E$ with $j \ne k$,
then $K(\cdot,j) \ne K(\cdot,k)$, and so $E$ can be identified
with the collection of functions $K(\cdot,j)$, $j \in E$.  
Note that
\[
0 
\le 
K(i,j) 
\le 
\frac{\bP^i\{\text{$X$  hits $j$}\}}{\bP^e\{\text{$X$  hits $i$}\} \bP^i\{\text{$X$  hits $j$}\}}
=
\frac{1}{\bP^e\{\text{$X$  hits $i$}\}},
\]
and so the set of functions $\{K(\cdot,j) : j\in E\}$
is a pre-compact subset of $\bR_+^E$ equipped with the
usual product topology. Its closure $\bar E$
is the {\em Doob--Martin compactification} of $E$. 
The set $\partial E := \bar E \setminus E$
is the {\em Doob--Martin boundary} of $E$.
By construction, a sequence $(j_n)_{n \in \bN}$ in $E$ converges
to a point in $\bar E$
if and only if the sequence of real numbers
$(K(i,j_n))_{n\in\bN}$ converges for all $i\in E$, and each function
$K(i,\cdot)$ extends continuously to $\bar E$. The
resulting function $K: E \times \bar E \rightarrow \bR$ is
the {\em extended Doob-Martin kernel}.  

A specific subset $\dmin E$, the \emph{minimal} boundary, of the full boundary $\partial E$ 
is of particular importance from a geometric as well as probabilistic point of view. Let
$\bH_{1,+}$ be the set of harmonic functions $h:E\to \bR_+$ with $h(e)=1$. 
This is a compact convex set, and its extreme points are those harmonic functions
$h\in\bH_{1,+}$ with the property that $a g < h$ implies $g=h$ whenever $0 < a < 1$ and $g\in\bH_{1,+}$. 
We have $K(\cdot,y)\in \bH_{1,+}$ for all $y\in\partial E$, and we write $\dmin E$ for the set of those 
boundary points that correspond to extremal harmonic functions. The set $\dmin E$ is a $G_\delta$.
With this notation in place, any 
$h\in\bH_{1,+}$ has a unique representation
\begin{equation*}
  h(x) = \int K(x,y)\, \mu(dy),
\end{equation*}
where $\mu$ is a probability measure that assigns all of its mass to $\dmin E$.

A first major probabilistic consequence 
of the Doob--Martin compactification is that the limit 
$X_\infty:=\lim_{n \rightarrow \infty} X_n$ exists $\bP^e$-almost 
surely in the topology of~$\bar E$ and that the 
distribution of this limit is given by the measure $\mu$ representing
the trivial element $h\equiv 1$ of $\bH_{1,+}$.

In terms of analysis, the vector space $\bH_b$ of bounded harmonic functions
endowed with the supremum norm is a Banach space (the {\em Poisson
boundary} of the Markov chain) and this Banach space is isomorphic
to the $L^\infty$ space associated with the measure space consisting of
$\partial E$ equipped with its Borel $\sigma$-field 
and the probability measure 
given by the distribution of $X_\infty$ under $\bP^e$.
The tail $\sigma$-field of 
$(X_n)_{n \in \bN_0}$ coincides $\bP^e$-almost surely
with the $\sigma$-field generated by $X_\infty$ and so 
the Poisson boundary captures how the process
can ``go to infinity'' and what probabilities are associated with the
various alternatives.

The second consequence of the Doob--Martin compactification is
that not only does it contain information about how the Markov chain
behaves at large times when ``left to its own devices'', but also,
somewhat loosely speaking, how it can be conditioned to behave
at large times. Each $j \in E = \bigsqcup_{n \in \bN_0} E_n$
belongs to a unique $E_n$ whose index $n$ we denote by $N(j)$.  If the chain
starts in state $e$, then $N(j)$ is the only time that there is
positive probability the chain will be in state $j$. Write 
$(X_0^j, \ldots, X_{N(j)}^j)$ for the {\em bridge} obtained by
starting the chain in state $e$ and conditioning it to be in
state $j$ at time $N(j)$.  This process 
is a Markov chain with forward transition probabilities
\begin{align*}
\bP\{X_{n+1}^j = i'' \, | \, X_n^j = i'\}
\ & =\
\frac{
\bP^e\{X_n = i', \, X_{n+1} = i'', \, X_{N(j)} = j\}
}
{
\bP^e\{X_n = i', \, X_{N(j)} = j\}
} \\
& =\ 
\frac{
\bP^e\{ \text{$X$  hits $i'$}\} P(i',i'') \bP^{i''}\{ \text{$X$  hits $j$}\}
}
{
\bP^e\{ \text{$X$  hits $i'$}\} \bP^{i'}\{ \text{$X$  hits $j$}\}
} \\
& =\ 
\frac{
P(i',i'') \bP^{i''}\{ \text{$X$  hits $j$}\} / \bP^e\{\text{$X$  hits $j$}\}
}
{
\bP^{i'}\{ \text{$X$  hits $j$}\} / \bP^e\{\text{$X$  hits $j$}\}
} \\
& =\ 
K(i',j)^{-1} P(i',i'') K(i'',j). \\
\end{align*}
Moreover,
\[
\begin{split}
\bP\{X_n^j = i' \, | \, X_{n+1}^j = i''\}
& =
\frac{
\bP^e\{X_n = i', \, X_{n+1} = i'', \, X_{N(j)} = j\}
}
{
\bP^e\{X_{n+1} = i'', \, X_{N(j)} = j\}
} \\
& = 
\frac{
\bP^e\{X_n = i', \, X_{n+1} = i''\} \bP^{i''}\{X_{N(j)} = j\}
}
{
\bP^e\{X_{n+1} = i''\} \bP^{i''}\{X_{N(j)} = j\}
} \\
& =
\bP^e\{X_n = i' \, | \, X_{n+1} = i''\}, \\
\end{split}
\]
and so $(X_0^j, \ldots, X_{N(j)}^j)$ has the same backward transition
probabilities as $(X_n)_{n \in \bN_0}$.

Suppose now that $(j_k)_{k \in \bN}$ is a sequence
of elements of the state space $E$ that converges to infinity
in the one-point compactification of $E$ or, equivalently,
$N(j_k) \to \infty$ as $k \to \infty$.  As observed in \cite{MR0426176}, such a sequence
$(j_k)_{k \in \bN}$ converges in the Doob--Martin topology if and
only if finite initial segments of the corresponding bridges
converge in distribution.  Moreover, two
sequences of states converge to the same limit if and only if
the limiting distributions of finite
initial segments are the same.  For a sequence
$(j_k)_{k \in \bN}$ that converges to the point $y$ in the Doob--Martin boundary,
the limiting distributions of the 
initial segments define the distribution of an $E$-valued
process $(X_n^{(h)})_{n \in \bN_0}$ that is Markovian with
forward transition probabilities $P^{(h)}$ given by 
\[
P^{(h)}(i,j) := h(i)^{-1} P(i,j) h(j), \quad i,j \in E^{(h)},
\]
where $h(i) = \lim_{k \to \infty} K(i,j_k) = K(i,y)$ 
and 
\[
\begin{split}
E^{(h)} 
& := \{i \in E : h(i) > 0\} \\
& = \{i \in E: \lim_{k \to \infty} \bP\{X_{N(i)} = i \, | \, X_{N(j_k)} = j_k\} > 0\}, \\
\end{split}
\]
and the same backward transition probabilities as $(X_n)_{n \in \bN_0}$.
This Markov chain $(X_n^{(h)})_{n \in \bN_0}$ is an $h$-transform 
using the harmonic function $h$.  Moreover, if $(y_k)_{k \in \bN}$
is a sequence of points in the Doob--Martin boundary, then
$\lim_{k \to \infty} y_k = y$ for some point $y$ in the Doob--Martin boundary
if and only if the initial segments of 
$(X_n^{(K(\cdot, y_k))})_{n \in \bN_0}$
converge in distribution to the corresponding 
initial segments of 
$(X_n^{(K(\cdot, y))})_{n \in \bN_0}$.

We call any Markov process $(Y_n)_{n \in \bN_0}$ with $Y_0 = e$
and the same backward transition probabilities
as $(X_n)_{n \in \bN_0}$ an {\em infinite bridge} for $(X_n)_{n \in \bN_0}$.
The distribution of an infinite bridge is a mixture of distributions of infinite
bridges that have trivial tail $\sigma$-fields, and we call the latter
{\em extremal} infinite bridges. If $(j_k)_{k \in \bN}$ converges to a point $y$
in the Doob--Martin boundary, then the 
corresponding harmonic function $h = K(\cdot,y)$ is extremal if 
and only if the limit infinite bridge  $(X_n^{(h)})_{n \in \bN_0}$
is extremal.

%

\section{The Doob--Martin kernel of the R\'emy chain}
\label{DMker}

We return from the general setting of the
previous section to consideration of the R\'emy chain.  
Given two binary trees $\ss$ and $\tt$ with $2m+1$ and $2(m+n)+1$ vertices,
we wish to derive a formula for the multi-step
transition probability
\[
p(\ss,\tt) := \bP\{T_{m+n} = \tt \, | \, T_m = \ss\}
\]
and hence obtain a formula for the Doob--Martin kernel
with reference state $\aleph$, since
\[
\begin{split}
K(\ss,\tt) 
& := \frac{p(\ss, \tt)}{p(\aleph, \tt)} \\
& = \frac{\bP\{T_{m+n} = \tt \, | \, T_m = \ss\}}{\bP\{T_{m+n} = \tt\}} \\
& = \frac{\bP\{T_{m+n} = \tt, \; T_m = \ss\}}{\bP\{T_m = \ss\} \bP\{T_{m+n} = \tt\}} \\
& = \frac{1}{\bP\{T_m = \ss\}} \bP\{T_m = \ss \, | \, T_{m+n} = \tt\} \\
& = C_m \bP\{T_m = \ss \, | \, T_{m+n} = \tt\}, \\
\end{split}
\]
where we recall that the $m^{\mathrm{th}}$ Catalan number $C_m$ is 
the number of binary trees with $2m+1$ vertices.
For this we need the notion
of one binary tree being embedded in another, and this requires us to introduce
some preliminary definitions.  

To begin, we define a partial order
$<$ on the vertices of a binary tree by declaring that $u < v$
for two vertices $u$ and $v$ if $u \ne v$ and $u$ is on the (unique)
path leading from the root to $v$.  We say that $v$ is
{\em below} $u$.  Given two vertices $x$ and $y$, there is a unique
vertex $z$ such that $z \le x$, $z \le y$, and $w < z$ for any
other vertex $w$ such that $w \le x$ and $w \le y$.  We
say that $z$ is the {\em most recent common ancestor} of $x$ and $y$
and write $z = x \wedge y$.

If $u < v$ and the unique path from $u$ to $v$ passes through the
left (resp. right) child of $u$, then we write $u <_L v$ (resp. $u <_R v$)
and say that $v$ is {\em below and to the left} (resp. {\em below and
to the right}) of $u$.  Note that $<_L$ and $<_R$ are partial orders
with the property that if two vertices of the tree are comparable in
one order, then they are not comparable in the other.  Note also that
$u < v$ if and only if $u <_L v$ or $u <_R v$.

If we think of a binary tree as a subset of $\{0,1\}^*$, then for
two vertices $u = u_1 \ldots u_m$ and $v = v_1 \ldots v_n$ we have:
\begin{itemize}
\item
$u < v$ if and only if $m < n$ and $u_k = v_k$ for $1 \le k \le m$,
\item
the most recent common ancestor $u \wedge v$ of $u$ and $v$ is the vertex
$w = w_1 \ldots w_p$, where $p = \max\{k : u_k = v_k\}$ (where the
maximum of the empty set is $0$) and $w_k = u_k = v_k$ for 
$1 \le k \le p$,
\item
$u <_L v$ if and only if $m < n$, $u_k = v_k$ for $1 \le k \le m$,
and $v_{m+1} = 0$,
\item
$u <_R v$ if and only if $m < n$, $u_k = v_k$ for $1 \le k \le m$,
and $v_{m+1} = 1$.
\end{itemize}

\begin{definition}
An {\em embedding} of a binary tree $\ss$ into a binary tree $\tt$ is
a map from the vertex set of $\ss$ into the vertex set of $\tt$
such that the following hold.
\begin{itemize}
\item
The image of a leaf of $\ss$ is a leaf of $\tt$.
\item
If $u,v$ are vertices of $\ss$ such that
$v$ is below and to the left (resp. right) of $u$,
then the image of $v$ in $\tt$ is below and to the left (resp. right)
of the image of $u$ in $\tt$.
\end{itemize}
Figure~\ref{fig:embedding_example} illustrates this definition.
\end{definition}

\begin{figure}
	\centering
		\includegraphics[width=10cm]{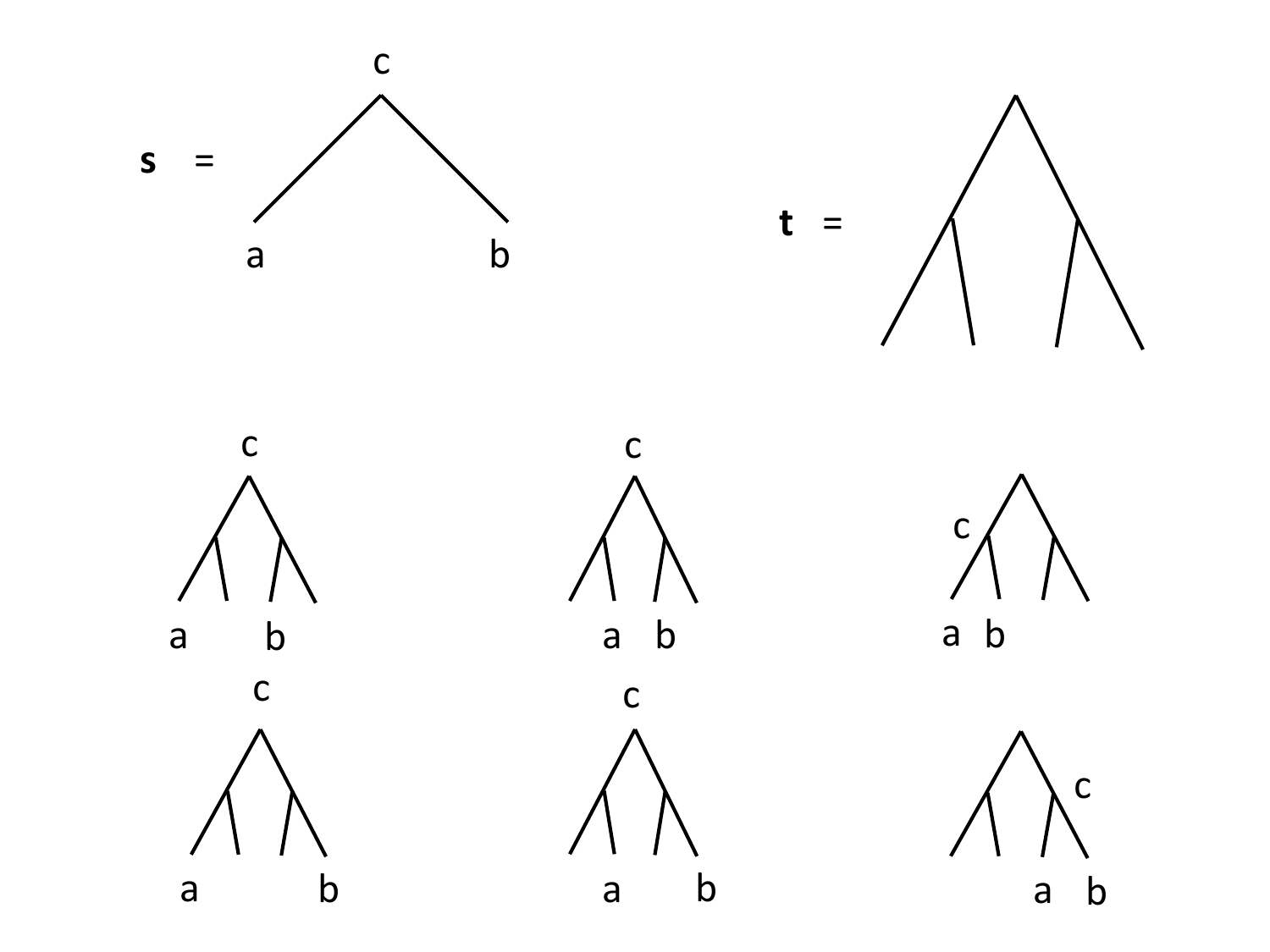}
	\caption{\small All the embeddings of the unique binary  tree 
	$\ss = \aleph$ with $3$
	 vertices into a particular tree $\tt$ with $7$ vertices.}
	\label{fig:embedding_example}
\end{figure}

\begin{remark}
Note that an embedding of $\ss$ into $\tt$ is uniquely determined by the 
images of the leaves of $\ss$, 
because if $x$ and $y$ are vertices of $\ss$, then the image
of the most recent common ancestor of $x$ and $y$ in $\ss$ must be
the most recent common ancestor in $\tt$ of the images of $x$ and $y$.
\end{remark}

\begin{notation}
Write $N(\ss, \tt)$ for the number of embeddings of $\ss$ into $\tt$. For $k=1,2,\ldots$, write $\tt^c_k$ for the complete binary tree with  $2^k$ leaves, that is the binary tree with $2^k$ leaves
such that every leaf is graph distance $k$ from the root.  (In the representation
of binary trees as subsets of $\{0,1\}^\star$, $\tt^c_k$ is the subset consisting of words with length at most $k$
and the leaves are the words with length $k$.)  
\end{notation}

\begin{example}
\label{E:complete_binary_embedding}
We want to identify
the number $N(\ss,\tt^c_k)$ of embeddings of a given tree $\ss$ into $\tt^c_k$, the complete
binary tree with $2^k$ leaves.

It will be useful to introduce the infinite complete binary
tree.  This is the set $\{0,1\}^* \sqcup \{0,1\}^\infty$.  
For distinct points $x$ and $y$ in $\{0,1\}^\infty$
with $x = u_1 u_2 \ldots$ and $y = v_1 v_2 \ldots$, set
$x \wedge y = u_1 \ldots u_h = v_1 \ldots v_h \in \{0,1\}^\star$, where
$h = \max\{g :  u_1 \ldots u_g = v_1 \ldots v_g\}$.
We say that $x$ is below and to the left of $x \wedge y$ and
$y$ is below and to the right of $x \wedge y$ if $u_{h+1}=0$
and $v_{h+1} = 1$.  

Using the same notation, 
put $r(x,y) = 2^{-h}$ and set $r(z,z) = 0$ for $z \in \{0,1\}^\infty$.
Then $r$ is a metric on $\{0,1\}^\infty$ that induces the (compact)
product topology on $\{0,1\}^\infty$.  We can equip $\{0,1\}^\infty$
with the probability measure $\kappa$
that is the product of the uniform probability
measures on each of the factors 
(that is, $\kappa$ is fair coin-tossing measure).
The $\kappa$-measure of any ball with diameter $2^{-\ell}$ is $2^{-\ell}$.

If $x_1, \ldots, x_{m+1}$
are distinct points in $\{0,1\}^\infty$, then these points induce a
(finite) binary tree with $m+1$ leaves in the obvious way:
we may identify the most recent common ancestor of the leaves
corresponding to $x_i$ and $x_j$ with
$x_i \wedge x_j$ and say that the point corresponding to
$x_i$ is below and to the left of the most
recent common ancestor of the points corresponding to $x_i$ and $x_j$ in the reduced tree
if $x_i$ is below and to the left of $x_i \wedge x_j$, etc.
Call this tree $T(x_1, \ldots, x_{m+1})$.  Observe that
$T(x_1, \ldots, x_{m+1}) = T(x_{\pi(1)}, \ldots, x_{\pi(m+1)})$
for any permutation $\pi$ of $\{1,2,\ldots,m+1\}$.
 
Suppose that the tree $\ss$ has $m+1$ leaves.  
If the leaves of an embedding
of $\ss$ into $\tt^c_k$
are $y_i = u_{i1} \ldots u_{ik}$
for $1 \le i \le m+1$ and we set 
$x_i = u_{i1} \ldots u_{ik} u_{i,k+1} u_{i,k+2} \ldots$ for any
choice of $u_{i,k+1},  u_{i,k+2}, \ldots$, $1 \le i \le m+1$, then
$T(x_1, \ldots, x_{m+1}) = \ss$.  Conversely, if 
$x_i = u_{i1} u_{i2} \ldots \in \{0,1\}^\infty$, $1 \le i \le m+1$, 
are such that
$T(x_1, \ldots, x_{m+1}) = \ss$ and 
$r(x_i,x_j) > 2^{-k}$ for $1 \le i \ne j \le m+1$, then
putting $y_i = u_{i1} \ldots u_{ik}$ for $1 \le i \le m+1$
gives the leaves of an embedding of $\ss$ into $\tt^c_k$.

With the notation $x=(x_1,\ldots, x_{m+1})$ it follows that
\[
\begin{split}
& \frac{1}{(m+1)!}
\frac{
\kappa^{\otimes (m+1)} \{x : T(x) = \ss
\text{ and } r(x_i,x_j) > 2^{-k} \text{ for }1 \le i \ne j \le m+1\}
}
{
\kappa^{\otimes (m+1)} \{x : r(x_i,x_j) > 2^{-k} \text{ for }1 \le i \ne j \le m+1\}
} \\
& \quad =
2^{-(m+1) k} N(\ss,\tt^c_k). \\
\end{split}
\]
Indeed, the left hand side counts the fraction of all those of the 
(in total $2^{k(m+1)}$) mappings from $[m+1]$ to $[2^k]$ which 
correspond to an embedding of $\ss$ into $\tt^c_k$.
In particular,
\[
\lim_{k \to \infty} 2^{-(m+1) k} N(\ss,\tt^c_k)
=
\frac{1}{(m+1)!}
\kappa^{\otimes (m+1)} \{(x_1, \ldots, x_{m+1}) 
: T(x_1, \ldots, x_{m+1}) = \ss\}.
\]
\end{example}

\begin{theorem}
\label{T:transition_prob}
Suppose that $\ss$ and $\tt$ are two binary trees with, respectively, 
$2m+1$ and $2(m+n)+1$ vertices.  Then, the probability that the
R\'emy chain transitions from $\ss$ to $\tt$ in $n$ steps is
\[
p(\ss,\tt)
=
n! 
\frac{1}{(2m+1) \times (2m+3) \times \cdots \times (2(m+n) - 1)}
\frac{1}{2^{n}} N(\ss,\tt),
\]
where $N(\ss,\tt)$ is the number of ways of embedding $\ss$ into $\tt$.
\end{theorem}

\begin{proof}
We condition on the event $\{T_m = \ss\}$  
and say that a vertex of $T_{m+n}$ is a {\em clonal descendant} 
of a vertex $v \in \ss$ if it is $v$ itself, a clone of $v$, 
a clone-of-a-clone of $v$, etc. 
We can then decompose $T_{m+n}$ 
into connected pieces according to their
clonal descent from the vertices of $\ss$
-- see Figure~\ref{fig:clonal_descent}
for a schematic representation of such a decomposition.

\begin{figure}
	\centering
		\includegraphics[width=9cm]{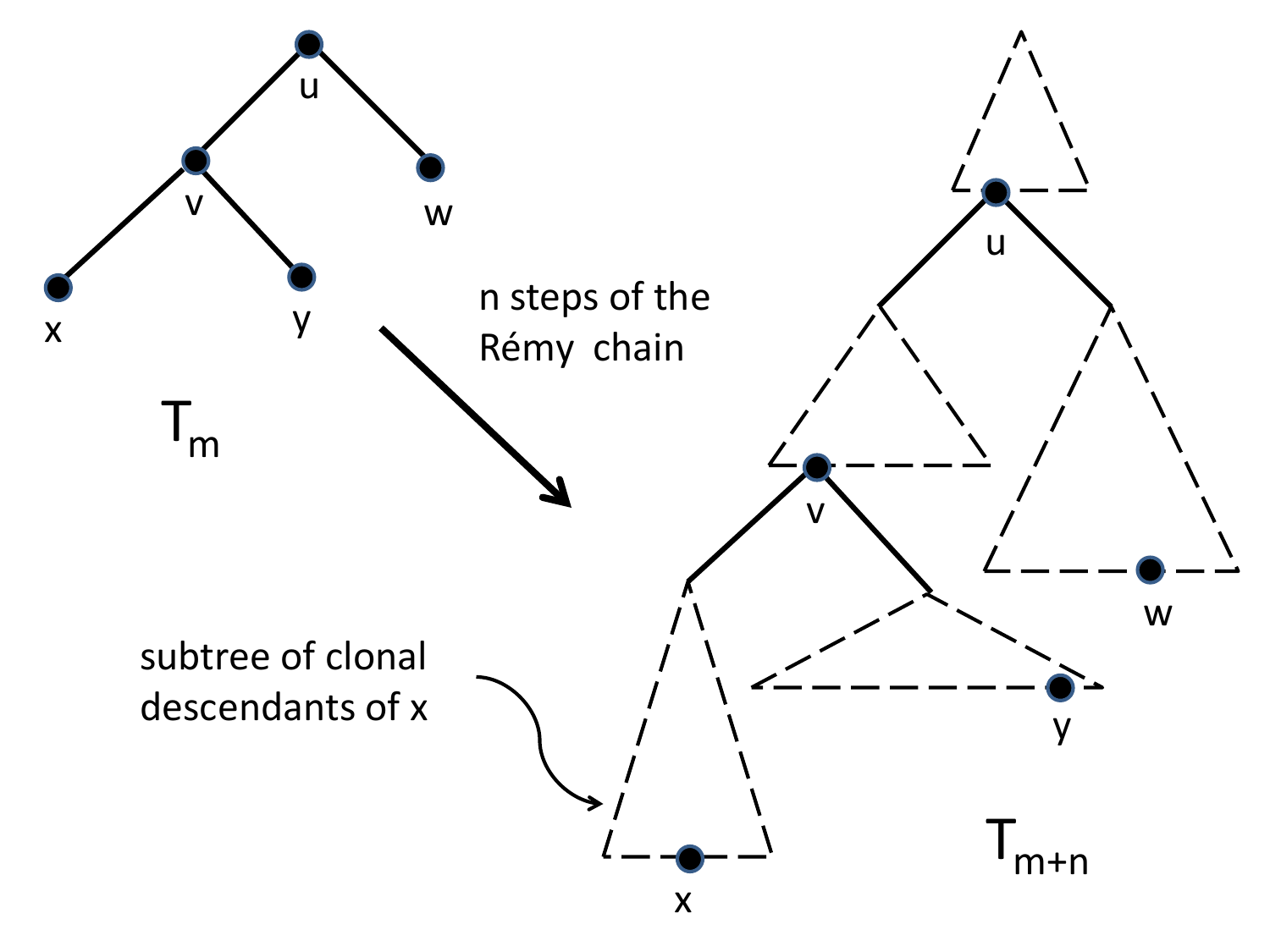}
	\caption{\small Decomposition of the tree $T_{m+n}$ via clonal descent 
	from the vertices of the tree $T_m = \ss$.}
	\label{fig:clonal_descent}
\end{figure}

It follows from the definition of the R\'emy chain that the
numbers of clonal descendants of the $2m+1$ vertices
of $\ss$ are the result of $n$ steps in a {\em P\'olya urn} that starts
with $2m+1$ balls of different colors
and at each stage a ball is chosen uniformly
at random and replaced along with {\bf two} balls of the same color.

Because the R\'emy chain generates uniformly distributed binary trees,
it further follows that conditional on the various
numbers of clonal descendants, 
the respective binary trees of clonal descendants 
are independent and uniformly distributed.

Moreover, a straightforward induction shows that,
conditional on the trees of clonal descendants, the ancestral vertices
from $\ss$ are located at independently and uniformly chosen leaves
of their respective trees of clonal descendants.

Therefore, if we label the vertices of $\ss$ with 
$1, \ldots, 2m+1$, then the conditional probability given 
$\{T_m = \ss\}$ that the operation of a further $n$ steps
of R\'emy's algorithm results in a binary tree
$\tt$ enhanced with a particular clonal descent decomposition
in which $2 n_j+1$ vertices are clonal descendants of vertex $j$
of $\ss$ for $1 \le j \le 2m+1$ is 
\[
\begin{split}
& \frac{n!}{n_1! \cdots n_{2m+1}!}
\frac{\prod_{j=1}^{2m+1} [1 \times 3 \times \cdots \times (2n_j - 1)]}
{(2m+1) \times (2m+3) \times \cdots \times (2(m+n)-1)} \\
& \qquad 
\times \prod_{j=1}^{2m+1} \frac{1}{C_{n_j}}
\prod_{j=1}^{2m+1}\frac{1}{n_j+1} \\
& \quad =
n! 
\frac{1}{(2m+1) \times (2m+3) \times \cdots \times (2(m+n) - 1)}
\frac{1}{2^{n}},\\
\end{split}
\]
and the result is immediate.
\end{proof}

\begin{remark}
An alternative method for proving Theorem~\ref{T:transition_prob}
is to use arguments similar to those used in the Introduction to
show that R\'emy's algorithm does indeed generate uniform random binary
trees.  More precisely, 
let $\tilde \ss$ be a tree with $m+1$ leaves labeled by $[m+1]$
and let
$\tilde \tt$ be a tree with $(m+n)+1$ leaves labeled by $[(m+n)+1]$.
Recalling the construction of the enhanced chain 
$\tilde T_1, \tilde T_2, \ldots$,
the conditional probability of the event $\{\tilde T_{m+n} = \tilde \tt\}$
given the event $\{\tilde T_m = \tilde \ss\}$ is either zero if the
leaf--labeled binary tree induced by the leaves of $\tilde \tt$ labeled
with $[m+1]$ is not $\tilde \ss$ or 
\[
\frac{1}{(2m+1) \times (2m+3) \times \cdots \times (2(m+n) - 1)}
\frac{1}{2^{n}}
\]
if it is,
because, as in the Introduction, the leaf--labeling dictates the order in 
which vertices must be cloned, as well as the associated
choices of left-right re-attachments.  If $\ss$ and $\tt$ are unlabeled
binary trees with $m+1$ and $(m+n)+1$ leaves, respectively, then for any
labeling of the leaves of $\ss$ 
to give a leaf--labeled binary tree $\tilde \ss$, the
number of ways of labeling $\tt$ to give a leaf--labeled binary tree 
$\tilde \tt$ such that the leaf--labeled binary tree induced by the leaves
labeled with $[m+1]$ is just $n! N(\ss,\tt)$, 
because any admissible labeling of $\tt$ corresponds to an embedding of
$\ss$ into $\tt$ composed with a labeling (with  $\{m+1, \ldots, (m+n) + 1\}$) 
of those leaves of $\tt$ that are
not in the image
of $\ss$.
\end{remark}

\begin{corollary}\label{cor37}
Suppose that $\ss$ and $\tt$ are two binary trees with, respectively, 
$2m+1$ and $2(m+n)+1$ vertices.
Then, the corresponding Doob--Martin kernel is
\[
K(\ss,\tt) = C_{m+n} \, p(\ss,\tt)=
2^m 
\frac{1 \times 3 \times \cdots \times (2m-1)}
{(n+1) \times (n+2) \times \cdots \times (m+n+1)}
N(\ss,\tt).
\]
\end{corollary} 

\begin{proof}
This is immediate from the definition of the Doob--Martin kernel and 
Theorem~\ref{T:transition_prob}.
\end{proof}

\begin{notation}
Given $m \in \bN$ and  a binary tree $\tt$ with 
$2(m+n)+1$ vertices for some $n \in \bN$, 
define $S_m^\tt$ to be the random binary
tree embedded in $\tt$ that is obtained by picking $m+1$ leaves of $\tt$
uniformly at random without replacement -- see Figure~\ref{fig:embeddings}.
\end{notation}

\begin{figure}[ht]
		\includegraphics[height=5cm]{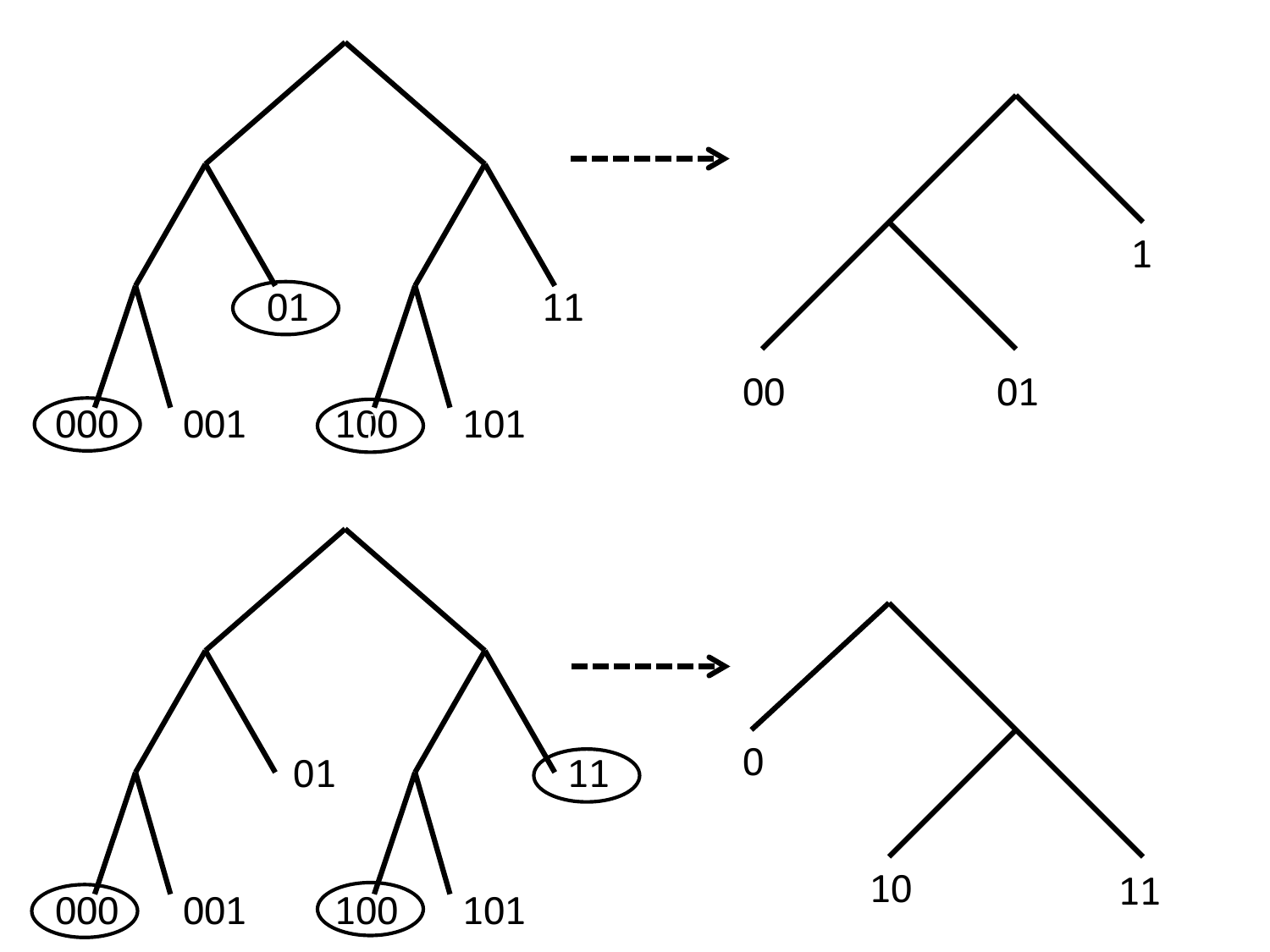}
	\caption{\small Two possible realizations of the random tree
	$S_m^\tt$ when $m=2$, $n=3$ and $\tt$ is the binary tree
with $11 = 2(2+3)+1$ vertices
depicted twice on the left side along with an indication
of its representation as a subset of $\{0,1\}^*$. 
Picking the $m+1=3$ leaves $000$, $01$ and $100$ out of the 
$(m+n)+1 = 6$ leaves of $\tt$ as shown in the
top row results in a realization of
$S_m^\tt$ that has leaves $00$, $01$ and $1$ in its representation
as a subset of $\{0,1\}^*$, 
while picking the leaves $001$, $100$ and $11$ of  $\tt$ 
as shown in the bottom row results 
in a realization of $S_m^\tt$ that has leaves $0$, $10$ and $11$
in its representation as a subset of $\{0,1\}^*$.}
	\label{fig:embeddings}
\end{figure}

\begin{corollary}
\label{C:D-M_sampling}
Suppose that $\ss$ and $\tt$ are two binary trees with, respectively, 
$2m+1$ and $2(m+n)+1$ vertices.  Then,  
\[
\bP\{S_m^\tt = \ss\}
= \frac{1}{C_m} K(\ss,\tt).  
\]
\end{corollary}

\begin{proof}
It suffices to observe that
\[
\begin{split}
\bP\{S_m^\tt = \ss\}
& = \frac{N(\ss,\tt)}{\binom{m+n+1}{m+1}} \\
& = \frac{(m+1)!}{(n+1)\times (n+2) \times \cdots \times (m+n+1)} N(\ss,\tt) \\
&= \frac{(m+1)!}{ 2^m \times 1\times3\times\cdots\times(2m-1)} K(\ss,\tt) \\
& = \frac{1}{C_m} K(\ss,\tt). \qedhere
\end{split} 
\]
\end{proof}

The following result is immediate
from Corollary~\ref{C:D-M_sampling}.
It shows that convergence of a sequence of binary trees
to a point in the Doob--Martin boundary is equivalent to the convergence
in distribution of the random embedded subtrees resulting from sampling
a finite number of leaves uniformly at random.  Thus convergence
in our setting is, as remarked in the Introduction,
 analogous to the notion of convergence of dense
graph sequences as explored in the theory of graph limits, where a
sequence of larger and larger graphs converges to a limit if the
random subgraphs defined by restriction to a finite number of vertices
sampled uniformly at random converge in distribution (see, for example,
\cite[Chapter 13]{MR3012035}).  The latter notion of convergence is
metrized by a very natural metric called the {\em cut metric}
that is, a priori, unrelated to sampling from a graph
and it would be interesting to know if there is an analogous
object that metrizes the notion of convergence of binary trees
in our setting.

\begin{corollary}\label{C:D-M_sampling1}
A sequence $(\tt_k)_{k \in \bN}$ of binary trees with the number
of leaves of $\tt_k$ going to infinity as $k \to \infty$ converges in
the Doob--Martin compactification if and only if 
for each $m \in \bN$ the sequence of random
binary trees $(S_m^{\tt_k})_{k \in \bN}$ converges in distribution.  
Moreover, two such convergent sequences of binary trees
$(\tt_k')_{k \in \bN}$ and $(\tt_k'')_{k \in \bN}$
converge to the same point in the Doob--Martin boundary if and only if 
for all $m \in \bN$ the
limiting distribution of $S_m^{\tt_k'}$ as $k \to \infty$ coincides
with the limiting distribution of $S_m^{\tt_k''}$ as $k \to \infty$.
\end{corollary}

\begin{example}
\label{E:complete_binary_converges}
Recall that
$\tt_k^c$ is the complete binary tree with $2^k$ leaves. It follows from 
Corollary \ref{cor37},
 the last equality in Example~\ref{E:complete_binary_embedding} 
and Corollary~\ref{C:D-M_sampling1}
that the sequence $(\tt_k^c)_{k \in \bN}$
converges in the Doob--Martin topology with
\[
\begin{split}
\lim_{k \to \infty} K(\ss,\tt_k^c)
& =
2^m 
(1 \times 3 \times \cdots \times (2m-1))
\frac{1}{(m+1)!} \\
& \quad \times \kappa^{\otimes (m+1)} \{(x_1, \ldots, x_{m+1}) 
: T(x_1, \ldots, x_{m+1}) = \ss \} \\
& =
C_m \kappa^{\otimes (m+1)} \{(x_1, \ldots, x_{m+1}) 
: T(x_1, \ldots, x_{m+1}) = \ss \}\\
\end{split}
\]
for a binary tree $\ss$ with $m+1$ leaves.
Equivalently,
\[
\lim_{k \to \infty}
\bP\{T_m = \ss \, | \, T_{2^k-1} = \tt_k^c\}
=
\kappa^{\otimes (m+1)} \{(x_1, \ldots, x_{m+1}) 
: T(x_1, \ldots, x_{m+1}) = \ss \}.
\]

The latter probability can be evaluated quite explicitly.  
Let $X_1, \ldots, X_{m+1}$ be independent, identically distributed 
$\{0,1\}^\infty$-valued random variables with common distribution $\kappa$.
We label the balls of $\{0,1\}^\infty$ that have diameter $2^{-k}$ with the
elements of $\{0,1\}^k$ by declaring that $B_{u_1 \ldots u_k}$ is the unique
ball containing all points of the form $u_1 \ldots u_k u_{k+1} u_{k+2} \ldots$
for arbitrary $u_{k+1}, u_{k+2}, \ldots \in \{0,1\}$.  There is a random
integer $R$ such that $\{X_1, \ldots X_{m+1}\} \subset B_{u_1 \ldots u_R}$
for some $u_1, \ldots, u_R \in \{0,1\}$, but 
$\{X_1, \ldots X_{m+1}\} \not \subset B_{u_1 \ldots u_R 0}$
and
$\{X_1, \ldots X_{m+1}\} \not \subset B_{u_1 \ldots u_R 1}$.
Observe that
\[
\begin{split}
& \bP\{\#\{i : X_i \in B_{u_1 \ldots u_R 0}\} = h, \, \#\{j : X_j \in B_{u_1 \ldots u_R 1}\} 
= m+1 - h \; | \; R, u_1, \ldots, u_R\} \\
& \quad =
\frac{
\binom{m+1}{h} \left(\frac{1}{2}\right)^{m+1}
}
{
1 - 2 \left(\frac{1}{2}\right)^{m+1}
} \\
\end{split}
\]
for
$1 \le h \le m$.
Moreover, given that $\#\{i : X_i \in B_{u_1 \ldots u_R 0}\} = h$
and  $\#\{j : X_j \in B_{u_1 \ldots u_R 1}\} = m+1 - h$, the set of locations
of the $X_i$ that fall in $B_{u_1 \ldots u_R 0}$ and the set of locations
of the $X_i$ that fall in $B_{u_1 \ldots u_R 1}$ are independent, with the
former random set being conditionally distributed as $h$ i.i.d. draws from
the probability measure $\kappa$ restricted to $B_{u_1 \ldots u_R 0}$ 
and renormalized to be a probability measure, and with the latter
random set being conditionally distributed as $m+1-h$ i.i.d. draws from
the probability measure $\kappa$ restricted to $B_{u_1 \ldots u_R 1}$ 
and renormalized to be a probability measure.  
Label the internal vertices of $\ss$ with $1,\ldots,m$. Let
$\alpha_\ell$ (resp. $\beta_\ell$) be the number of leaves of $\ss$ that
are below and to the left (resp. below and to the right) of vertex $\ell$
and write $\gamma_\ell := \alpha_\ell + \beta_\ell$ for the
the total number of leaves below the vertex labeled $\ell$.
It follows that
\[
\begin{split}
& \kappa^{\otimes (m+1)} \{(x_1, \ldots, x_{m+1}) : 
T(x_1, \ldots, x_{m+1}) = \ss\} \\
& \quad =
\prod_{\ell=1}^m 
\frac{(\alpha_\ell + \beta_\ell)!}
{\alpha_\ell! \beta_\ell!}
\left(\frac{1}{2}\right)^{\alpha_\ell+\beta_\ell}
\left(1 - \left(\frac{1}{2}\right)^{\alpha_\ell+\beta_\ell-1}\right)^{-1} \\
& \quad =
(m+1)! \frac{1}{2^m}
\prod_{\ell=1}^m 
\left(2^{\alpha_\ell+\beta_\ell-1} - 1\right)^{-1} \\
& \quad = 
(m+1)! \frac{1}{2^m}
\prod_{\ell=1}^m 
\left(2^{\gamma_\ell-1} - 1\right)^{-1},\\
\end{split}
\]
where the second equality results from a telescope product along the binary tree  $\ss$.
In particular, the function that maps $\ss$ to
\[
C_m (m+1)! 
\frac{1}{2^m}
\prod_{\ell=1}^m 
\left(2^{\gamma_\ell-1} - 1\right)^{-1}
=
1 \times 3 \times \cdots \times (2m-1) \times
\prod_{\ell=1}^m 
\left(2^{\gamma_\ell-1} - 1\right)^{-1}
\]
is harmonic for the R\'emy chain.
We can write this function more compactly as
\[
\mathfrak h : \ss
\mapsto
1 \times 3 \times \cdots \times (2m-1)
\times
\prod_v 
\left(2^{\# \ss(v)-1} - 1\right)^{-1},
\]
where the product is over the interior vertices of $\ss$,
and $\# \ss(v)$ is the number of leaves below the interior vertex $v$.

It is instructive to check
directly that this function is indeed harmonic.
Suppose that in one step of the chain
starting from the tree $\ss$ with $2m+1$ vertices
the vertex $v$ of $\ss$ is chosen to be cloned.
This produces a tree $\tt$ with $2m+1$ old vertices that we
can identify with the vertices of $\ss$ and two new vertices
that we will call $x$ and $y$, with $x$ an interior vertex and
$y$ a leaf.  If $u \ne v$ is an interior vertex of $\ss$ that is on the 
path from the root to 
$v$ (that is, $u$ is an ancestor of $v$), 
then $\# \tt(u) = \# \ss(u) + 1$.  For any other interior vertex $u$
of $\ss$ we have  $\# \tt(u) = \# \ss(u)$. 
Lastly, $\# \tt(x) = \# \ss(v) + 1$, where we put $\# \ss(v) = 1$ if
$v$ is a leaf of $\ss$.  Therefore, if $v$ is an interior vertex of
$\ss$, then
\[
\begin{split}
& 1 \times 3 \times \cdots \times (2m+1)
\times
\prod_w 
\left(2^{\# \tt(w)-1} - 1\right)^{-1} \\
& \quad =
1 \times 3 \times \cdots \times (2m+1) 
\times
\left[\prod_{u < v} 
\left(2^{\# \ss(u)} - 1\right)^{-1}\right] \\
& \qquad \times
\left(2^{\# \ss(v)-1} - 1\right)^{-1}
\times
\left(2^{\# \ss(v)} - 1\right)^{-1} \\
& \qquad \times
\left[\prod_{u \not \le v} 
\left(2^{\# \ss(u) - 1} - 1\right)^{-1}\right], \\
\end{split}
\] 
whereas if $v$ is a leaf, then
\[
\begin{split}
& 1 \times 3 \times \cdots \times (2m+1)
\times
\prod_w 
\left(2^{\# \tt(w)-1} - 1\right)^{-1} \\
& \quad =
1 \times 3 \times \cdots \times (2m+1)
\times
\left[\prod_{u < v} 
\left(2^{\# \ss(u)} - 1\right)^{-1}\right]
\times
\left[\prod_{u \not \le v} 
\left(2^{\# \ss(u) - 1} - 1\right)^{-1}\right]. \\
\end{split}
\]

Writing $I$ for the set of internal
vertices of $\ss$ and $L$ for the leaves, we therefore see that harmonicity of 
$\mathfrak h$ is equivalent to
\[
\begin{split}
& \sum_{v \in I} \prod_{u < v} 
\frac{
2^{\# \ss(u)-1} - 1
}
{
2^{\# \ss(u)} - 1
}
\frac{1}{2^{\# \ss(v)} - 1} \\
& \qquad +
\sum_{v \in L} \prod_{u < v}
\frac{
2^{\# \ss(u)-1} - 1
}
{
2^{\# \ss(u)} - 1
} \\
& \quad = 
\sum_{v \in \ss}
\prod_{u < v} 
\frac{
2^{\# \ss(u)-1} - 1
}
{
2^{\# \ss(u)} - 1
}
\frac{1}{2^{\# \ss(v)} - 1} \\
& \quad =
1. \\
\end{split}
\]

This, however, is clear by induction.  It is certainly true if $\ss$ 
is the trivial binary tree with a single vertex or the binary tree
$\aleph$ with three vertices.  Assuming for some binary tree $\ss$ with
$m+1$ leaves that it is true for all binary trees with fewer leaves, 
we see from a consideration of the the 
left and right subtrees below the root of $\ss$
that the sum in question is
\[
\frac{1}{2^{m+1} - 1}
+
\frac{
2^{m} - 1
}
{
2^{m+1} - 1
}
[1 + 1]
=
1,
\]
as required.

The one-step transition probability for the corresponding
Doob $h$-transformed chain is, for binary trees $\ss$ and
$\tt$ with $2m+1$ and $2m+3$ vertices,
\[
\begin{split}
& \left[
1 \times 3 \times \cdots \times (2m-1) \times
\prod_u
\left(2^{\# \ss(u)-1} - 1\right)^{-1}
\right]^{-1} \\
& \qquad \times
\frac{1}{2(2m+1)} N(\ss,\tt) \\
& \qquad \times
1 \times 3 \times \cdots \times (2m+1) \times
\prod_v
\left(2^{\# \tt(v) - 1} - 1\right)^{-1} \\
& \quad =
\frac{1}{2} 
\frac{
\prod_u 
\left(2^{\# \ss(u) - 1} - 1\right)
}
{
\prod_{v}
\left(2^{\# \tt(v) - 1} - 1\right)
}
N(\ss,\tt), \\
\end{split}
\]
where the products in $u$ run over the interior vertices of $\ss$
and the products in $v$ run over the interior vertices of $\tt$.

It is apparent from the above that one step of the $h$-transformed chain 
starting from the state $\ss$ can be described as follows.

\begin{itemize}
\item
Pick a vertex $v$ of $\ss$ with probability
\[
\prod_{u < v} 
\frac{
2^{\# \ss(u)-1} - 1
}
{
2^{\# \ss(u)} - 1
}
\frac{1}{2^{\# \ss(v)} - 1}.
\] 
\item
Cut off the subtree rooted at $v$ and set it aside.
\item
Attach a copy of the tree $\aleph$ with $3$ vertices to the end of the edge
that previously led to $v$.
\item
Re-attach the subtree rooted at $v$ uniformly at random
to one of the two leaves in the copy of $\aleph$.
\end{itemize}

\end{example}

\section{Infinite R\'emy bridges}
\label{S:infinite_bridges}

Given a binary tree $\tt$ with $2 m(\tt) + 1$ vertices, write
$ T_1^\tt (=\aleph),T_2^\tt, \ldots, T_{m(\tt)}^{\tt}$ for the bridge process
obtained by conditioning $T_1, \ldots, T_{m(\tt)}$
on the event $\{T_{m(\tt)} = \tt\}$.

Recall from Section~\ref{S:Martin_general}
that a sequence
$(\tt_k)_{k \in \bN}$ with $m(\tt_k) \to \infty$ converges
in the  Doob--Martin topology if and only if for each $\ell \in \bN$
the random $\ell$-tuple $(T_1^{\tt_k}, \ldots, T_{\ell}^{\tt_k})$
converges in distribution.  Moreover, the various limits define a set of
consistent distributions and hence the distribution of
a Markov chain $(T_n^\infty)_{n \in \bN}$
with $T_1^\infty = \aleph$.

Note that if $\ss, \tt$ are binary trees with $2m+1$ and
$2m+3$ vertices, respectively, then, using Theorem \ref{T:transition_prob},
\[
\begin{split}
\bP\{T_m^{\tt_k} = \ss \, | \, T_{m+1}^{\tt_k} = \tt\}
& =
\frac{
\bP\{T_m^{\tt_k} = \ss, \; T_{m+1}^{\tt_k} = \tt\}
}
{
\bP\{T_{m+1}^{\tt_k} = \tt\}
} \\
& =
\frac{
\bP\{T_{m+1}^{\tt_k} = \tt \, | \, T_m^{\tt_k} = \ss\}
 \bP\{T_m^{\tt_k} = \ss\}
}
{
\bP\{T_{m+1}^{\tt_k} = \tt\}
} \\
& =
\frac{p(\aleph, \ss) p(\ss, \tt) p(\tt, \tt_k)}
{p(\aleph, \tt) p(\tt, \tt_k)} \\
& = 
C_m^{-1} 
\frac{1}{2m+1} \frac{1}{2} N(\ss,\tt) / C_{m+1}^{-1} \\
&= \frac{(m+1)! m!}{(2m)!}
\frac{1}{2m+1} \frac{1}{2} N(\ss,\tt) \frac{(2(m+1))!}{(m+2)! (m+1)!} \\
& = \frac{1}{m+2} N(\ss,\tt).\\
\end{split}
\]

Therefore, any finite bridge $(T_n^\tt)_{n=1}^{m(\tt)}$ 
and any limit of finite bridges $(T_n^\infty)_{n \in \bN}$
evolves one step backward in time 
as follows:
\begin{itemize}
\item
Pick a leaf uniformly at random.
\item
Delete the chosen leaf and its sibling (which may or may not be a leaf).
\item
If the sibling is not a leaf, then close up the resulting
gap  by attaching the subtree
below the sibling to the parent of the chosen leaf and the sibling.
\end{itemize}

As we have already explained in the Introduction, 
understanding the Doob--Martin compactification
is equivalent to understanding all Markov chains 
with initial state $\aleph$ that have these backward transition dynamics.
We call any such a process an {\em infinite R\'emy bridge}.

\begin{example}
\label{E:coin_tossing_bridge}
Suppose that $(\tt_k)_{k \in \bN}$ is the binary tree depicted in
Figure~\ref{fig:infinite_bridge_example_part1}.
\begin{figure}
	\centering
		\includegraphics[height=5cm]{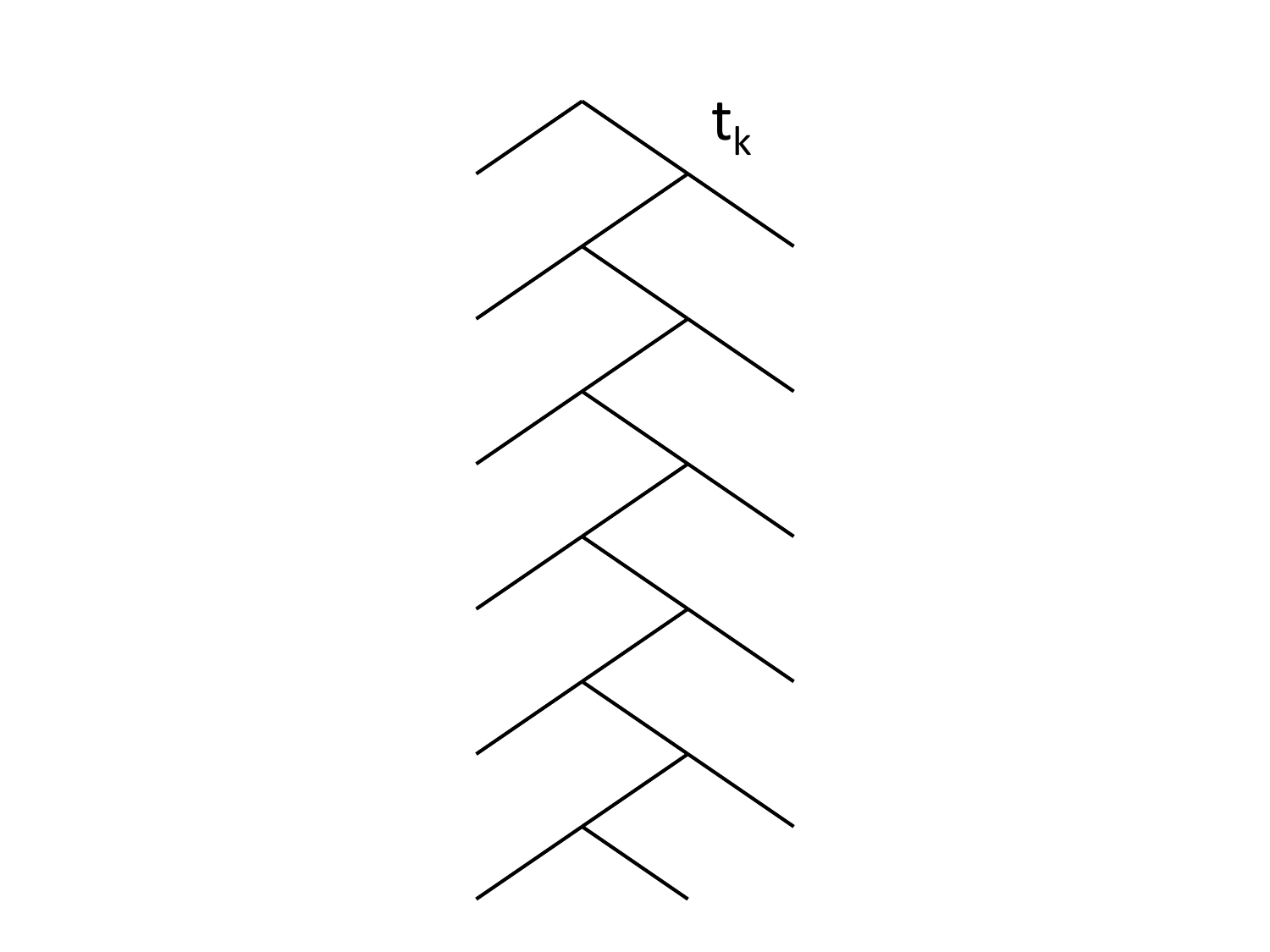}
	\caption{\small The binary tree $\tt_k$ of Example~\ref{E:coin_tossing_bridge}.
	This tree has $2k+1$ vertices and
	consists of a single spine with 
	  leaves hanging off to the left and right alternately.} 
	\label{fig:infinite_bridge_example_part1}
\end{figure}
It is not hard to see that the sequence $(\tt_k)_{k \in \bN}$
converges in the  Doob--Martin topology and that the value at time $n$
of the corresponding limit bridge $(T_n^\infty)_{n \in \bN}$
can be represented as the
subset of $\{0,1\}^\star$ that consists of the vertices 
$\emptyset, \epsilon_1, \epsilon_1 \epsilon_2,  \ldots, 
\epsilon_1 \epsilon_2 \cdots \epsilon_n$, where 
$\epsilon_1, \ldots, \epsilon_n$ are independent $\{0,1\}$-valued random
variables with $\bP\{\epsilon_i=0\} = \bP\{\epsilon_i=1\}=\frac{1}{2}$
for $1 \le i \le n$, plus the vertices 
$\bar \epsilon_1, \epsilon_1 \bar \epsilon_2, \ldots,
\epsilon_1 \epsilon_2 \cdots \bar \epsilon_n$, 
where $\bar \epsilon_i = 1 - \epsilon_i$ for $1 \le i \le n$.
In other words, $T_n^\infty$ consists of
a ``spine'' that moves to the left
or right depending on the successive tosses of a fair coin plus
the minimal set of extra
leaves that are required to yield a valid binary tree 
-- see Figure~\ref{fig:infinite_bridge_example_part2}.
We stress that $T_{n+1}^\infty$ is not obtained by simply appending
an extra independent fair coin toss $\epsilon_{n+1}$ to the end
of the sequence $\epsilon_1, \ldots, \epsilon_n$.  Rather, if we write
$\epsilon_1^n, \ldots, \epsilon_n^n$ for the coin tosses that correspond
to $T_n^\infty$ and $\epsilon_1^{n+1}, \ldots, \epsilon_{n+1}^{n+1}$
for the coin tosses that correspond to $T_{n+1}^\infty$, then 
$\epsilon_1^{n+1}, \ldots, \epsilon_{n+1}^{n+1}$ is obtained 
from $\epsilon_1^n, \ldots, \epsilon_n^n$ by inserting an
additional independent toss uniformly at random into one of
the $n+1$ ``slots'' associated with the latter sequence --
before the first toss, between two successive tosses, or
after the last toss.
\begin{figure}
	\centering
		\includegraphics[height=6cm]{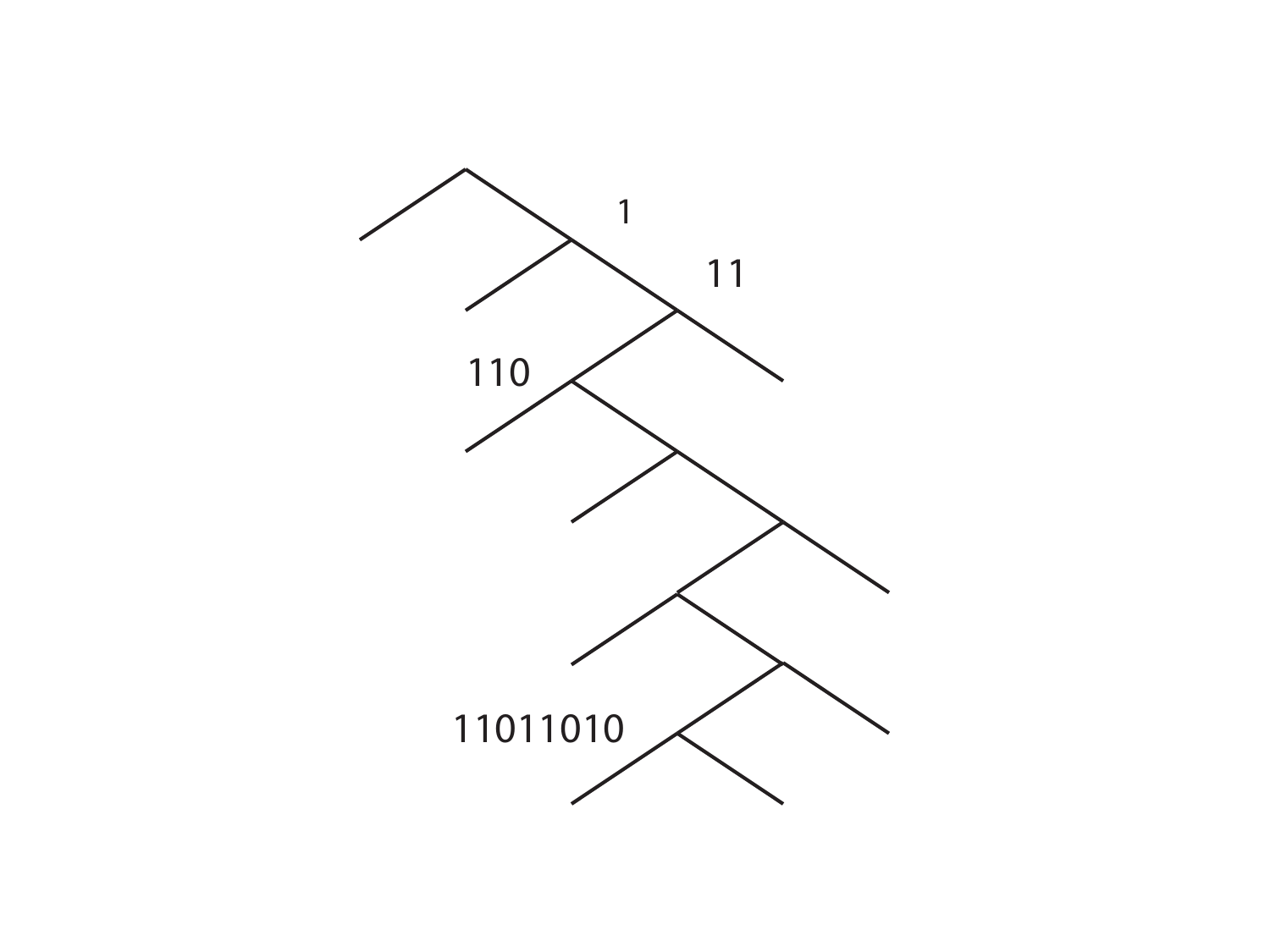}
	\vspace{-1cm}\caption{\small A realization at time $n=9$ of the 
	infinite R\'emy bridge arising from the sequence of trees depicted in 
	Figure~\ref{fig:infinite_bridge_example_part1}.
	The random tree consists of leaves hanging 
	off a single spine that moves to the
	left or right according to successive tosses of a fair coin.} 
	\label{fig:infinite_bridge_example_part2}
\end{figure}
\end{example}

\begin{example}\label{E:MBfulltree}
We know from Example~\ref{E:complete_binary_converges} that if $\tt_k^c$
is the complete binary tree with $2^k$ leaves, then the sequence
$(\tt_k^c)_{k \in \bN}$ converges in the Doob--Martin topology.  
Moreover, it is clear that the value
at time $n$ of the corresponding infinite R\'emy bridge
$(T_n^\infty)_{n \in \bN}$ is obtained by picking $n+1$ points 
from $\{0,1\}^\infty$ independently according
to the probability measure $\kappa$ and taking the 
finite binary tree they induce 
-- see Figure~\ref{fig:complete_binary_tree_example}.
\begin{figure}[ht]
	\centering
		\includegraphics[height=5cm]{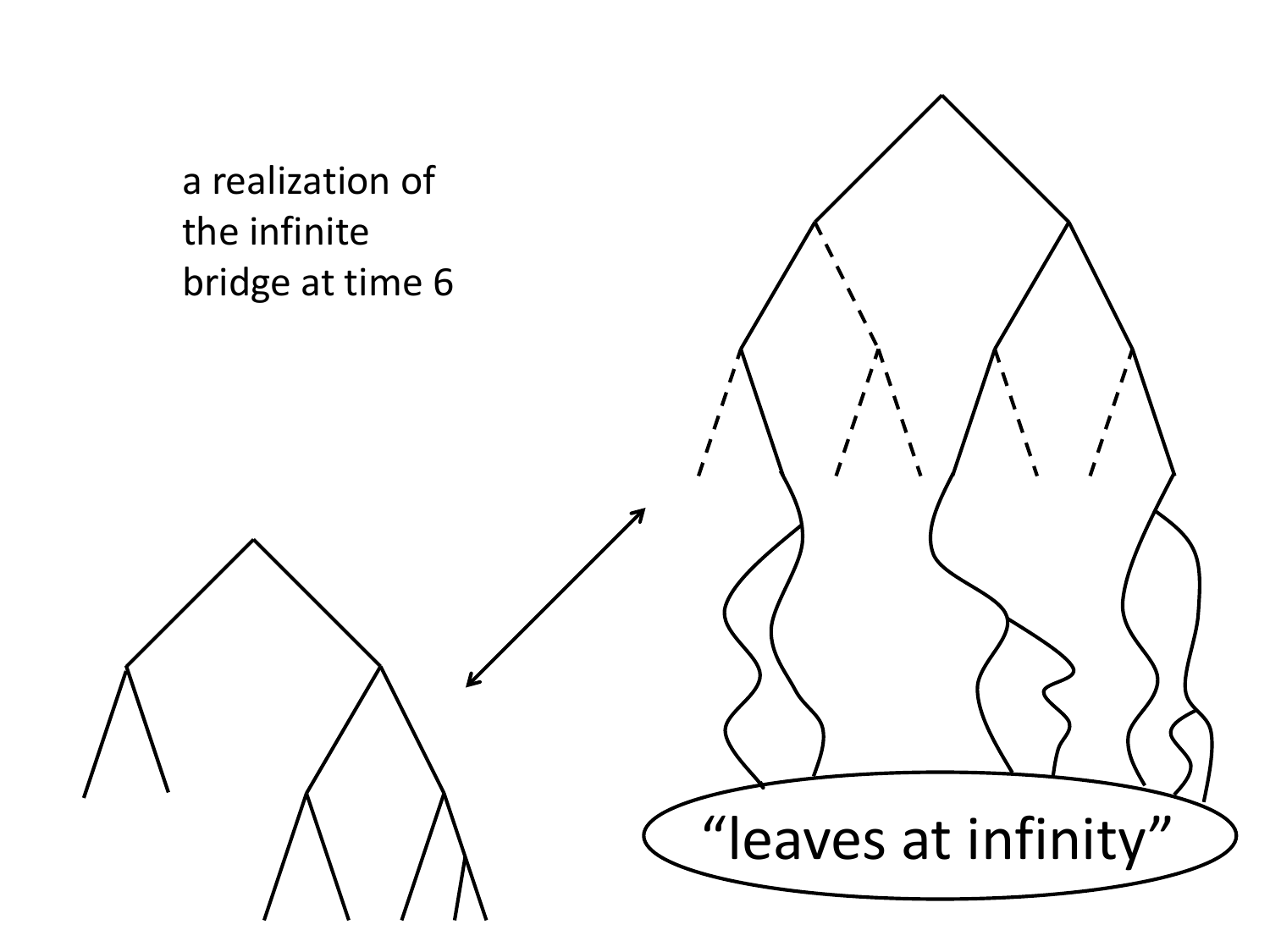}
	\caption{\small If $\tt_k^c$ is the complete binary tree
	 with $2^k$ leaves, then $\lim_k \tt_k^c$ exists in the 
	  Doob--Martin topology. The random value at time $n$ of the
	resulting infinite R\'emy bridge 
	can be built by choosing $n+1$ points independently and
	uniformly at
	random from the leaves at infinity of the infinite complete binary tree
	and constructing the tree they induce.}
	\label{fig:complete_binary_tree_example}
\end{figure}
\end{example}

\section{Labeled infinite R\'emy bridges and didendritic systems}
\label{S:N_tree}

Consider a binary tree $T''$ with $n+2$ leaves.  Label the leaves
of $T''$ with $[n+2]$ uniformly at random (that is, all $(n+2)!$
labelings are equally likely).  Now apply the following deterministic
procedure to produce a binary tree $T'$ with $n+1$ leaves and
a labeling of those leaves
by $[n+1]$. 
\begin{itemize}
\item
Delete the leaf labeled $n+2$, along with its sibling
(which may or may not be a leaf).
\item 
If the sibling of the leaf labeled $n+2$ is also a leaf, then 
assign the sibling's label to the common parent (which is now a leaf).
\item
If the sibling of the leaf labeled $n+2$ is not a leaf, then attach
the subtree below the sibling to the common parent with its 
leaf labels unchanged and leave all other leaf labels unchanged.
\end{itemize}
Clearly, the distribution of $T'$ is that arising from
one step starting from $T''$
of the backward R\'emy dynamics (that is, the common backward
dynamics of all infinite R\'emy bridges). Moreover, the labeling of $T'$ 
by $[n+1]$ is uniformly distributed over the $(n+1)!$ possible labelings.

Now suppose that $(T_n^\infty)_{n \in \bN}$ is an infinite R\'emy bridge.
For some $N$, let $S_N$ be a random binary tree with the same
distribution as $T_N^\infty$.  Label $S_N$ 
uniformly at random with $[N+1]$ to produce a labeled binary tree
$\tilde S_N$. Apply the above deterministic procedure successively for 
$n=N-1, \ldots, 1$ to produce labeled binary trees 
$\tilde S_{N-1}, \ldots, \tilde S_1$, where $\tilde S_n$ has $n+1$
leaves labeled by $[n+1]$ for $1 \le n \le N-1$.  
Write $S_n$ for the underlying binary tree obtained by removing
the labels of $\tilde S_n$.
It follows from the observation above that the sequence $(S_1, \ldots, S_N)$
has the same joint distribution as $(T_1^\infty, \ldots, T_N^\infty)$.
Note that the distribution of the sequence
$(\tilde S_1, \ldots,  \tilde S_N)$ is uniquely
determined by the distribution of $T_N^\infty$ and hence, 
{\em a fortiori}, by the joint distribution of $(T_n^\infty)_{n \in \bN}$.
Note also that if we perform this construction for two different values
of $N$, say $N' < N''$, to produce, with the obvious notation, sequences
$(\tilde S_1', \ldots, \tilde S_{N'}')$
and
$(\tilde S_1'', \ldots, \tilde S_{N''}'')$,
then
$(\tilde S_1', \ldots, \tilde S_{N'}')$
has the same distribution as
$(\tilde S_1'', \ldots, \tilde S_{N'}'')$.

By Kolmogorov's extension theorem we may therefore
suppose that there is a Markov process
$(\tilde T_n^\infty)_{n \in \bN}$
such that for each $n \in \bN$ 
the random element $\tilde T_n^\infty$ is a leaf-labeled binary
tree with $n+1$ leaves labeled by $[n+1]$
and the following hold.
\begin{itemize}
\item
The binary tree obtained by removing the labels of $\tilde T_n^\infty$
is $T_n^\infty$.
\item
For every $n \in \bN$, the conditional distribution
of $\tilde T_n^\infty$ given $T_n^\infty$ is uniform over the
$(n+1)!$ possible labelings of $T_n^\infty$.
\item
In going backward from time $n+1$ to time $n$,
$\tilde T_{n+1}^\infty$ is transformed into $\tilde T_n^\infty$
according to the deterministic procedure described above.
\end{itemize}

The distribution of
$(\tilde T_n)_{n \in \bN}$
is uniquely specified by the distribution
of $(T_n^\infty)_{n \in \bN}$
and the above requirements. 
Because of this distributional uniqueness, we refer to 
$(\tilde T_n^\infty)_{n \in \bN}$ as {\bf the}
{\em labeled version} of $(T_n^\infty)_{n \in \bN}$
and $(T_n^\infty)_{n \in \bN}$ as the {\em unlabeled
version} of $(\tilde T_n^\infty)_{n \in \bN}$.
In a similar vein, we will talk about objects
such as the ``leaf of $T_n^\infty$ labeled with $i \in [n+1]$.''

We have just described the construction of
a labeled infinite R\'emy bridge 
from an unlabeled one.
Using the Doob--Martin boundary, we can view this
construction from a slightly different point of view as follows.

\begin{remark}
\label{R:Qy}
Recall that for binary trees $\ss$ and $\tt$ with $n+1$ and
$n+2$ leaves, respectively, the backwards transition probability 
$q(\ss, \tt) := \bP\{T_n^\infty = \ss \, | \, T_{n+1}^\infty = \tt\}$
is the same for all infinite R\'emy bridges $(T_n^\infty)_{n \in \bN}$.  For
$y$ in the Doob--Martin boundary of the R\'emy chain, write
$\bQ^y$ for the distribution of the infinite R\'emy bridge associated
with $y$; that is, $\bQ^y$ is the distribution of the
Doob $h$-transform of the R\'emy chain for
the harmonic function $K(\cdot, y)$.  Thus $\bQ^y$ assigns mass
$q(\tt_1, \tt_2) q(\tt_2, \tt_3) \cdots q(\tt_{n-1}, \tt_n) K(\tt_n, y)$
to the set of paths that begin with 
the sequence of states $\tt_1, \tt_2, \ldots, \tt_{n-1}, \tt_n$.
The distribution of an arbitrary infinite R\'emy bridge
is of the form $\int \bQ^y \, \mu(dy)$ for some probability measure $\mu$
concentrated on the Doob--Martin boundary of the R\'emy chain, and this
representation is unique if $\mu$ is required
to be concentrated on the minimal boundary.
For labeled binary trees $\tilde \ss$ and $\tilde \tt$ with $n+1$ and
$n+2$ leaves, respectively, write  
$\tilde q(\tilde \ss, \tilde \tt) 
:= \bP\{\tilde T_n^\infty = \tilde \ss \, | \, \tilde T_{n+1}^\infty = \tilde \tt\}$
for the backwards transition probability
common to all labeled infinite R\'emy bridges $(\tilde T_n^\infty)_{n \in \bN}$. 
The construction of a labeled infinite R\'emy bridge $(\tilde T_n^\infty)_{n \in \bN}$
corresponding to an infinite R\'emy bridge $(T_n^\infty)_{n \in \bN}$ can be described
as follows: if $(T_n^\infty)_{n \in \bN}$ has distribution $\int \bQ^y \, \mu(dy)$,
then $(\tilde T_n^\infty)_{n \in \bN}$ has distribution $\int \tilde \bQ^y \, \mu(dy)$,
where $\tilde \bQ^y$ is the probability measure that assigns mass 
$\tilde q(\tilde \tt_1, \tilde \tt_2) \tilde q(\tilde \tt_2, \tilde \tt_3) \cdots \tilde q(\tilde \tt_{n-1}, \tilde \tt_n) 
\frac{1}{(n+1)!} K(\tt_n, y)$
to the set of paths that begin with the sequence of states 
$\tilde \tt_1, \tilde \tt_2, \ldots, \tilde \tt_{n-1}, \tilde \tt_n$ 
and $\tt_n$ is the binary tree obtained by removing the labels from $\tilde \tt_n$.
\end{remark}

It will be convenient for later use to be more concrete
about the structure of the extra randomness introduced
by labeling.  

\begin{definition}
\label{D:choice_variables}
Define a sequence of random variables $(L_n)_{n \in \bN}$ by setting
$L_n := k \in [n+1]$ if the leaf labeled $n+1$ in $\tilde T_n^\infty$
(and hence the one removed to form  $T_{n-1}^\infty$ from $T_n^\infty$
and $\tilde T_{n-1}^\infty$ from $\tilde T_n^\infty$) is the $k^{\mathrm{th}}$
smallest of the leaves of $T_n^\infty$ in the lexicographic order on $\{0,1\}^*$
(recall that $v_1 \ldots v_s$ is smaller than $w_1 \ldots w_t$ in the lexicographic
order if there is some $r < s \wedge t$ such that $v_q = w_q$ for $q \le r$,
$v_{r+1} = 0$, and $w_{r+1} = 1$).
For any $n \in \bN$ it is clear that 
$L_1, L_2, \ldots, L_n, (T_n^\infty, T_{n+1}^\infty, \ldots)$
are independent and that $L_n$ is uniformly distributed on $[n+1]$.
\end{definition}

By construction, $(\tilde T_1^\infty, \ldots, \tilde T_n^\infty)$
is a measurable function of $(L_1, \ldots, L_n)$ and $T_n^\infty$.
It might be expected from this observation that the entire labeled infinite R\'emy bridge
$(\tilde T_n^\infty)_{n \in \bN}$ (and hence, {\em a fortiori}, the infinite
R\'emy bridge $(T_n^\infty)_{n \in \bN}$) may be constructed from $(L_n)_{n \in \bN}$
and ``boundary conditions'' in the tail $\sigma$-field 
$\bigcap_{m \in \bN} \sigma\{T_n^\infty: n \ge m\}$.  The next result
shows that this is indeed the case.\footnote{The proof of Lemma \ref{L:tail_relation}  is incorrect but the result itself is valid:
see Section~\ref{S:corrigenda} for details.  The result is not used later in the paper.
We have kept the result and its incorrect proof so that, up to the note at the end of Section 
\ref{S:intro},  three footnotes in Section \ref{S:N_tree} and an update in the references, the content of Sections \ref{S:intro}--\ref{S:leftright} is identical to the published version \cite{MR3601650}.}
%
\begin{lemma}
\label{L:tail_relation}
For an infinite R\'emy bridge $(T_n^\infty)_{n \in \bN}$,
its labeled version $(\tilde T_n^\infty)_{n \in \bN}$, and the selection
sequence $(L_n)_{n \in \bN}$,
\[
\begin{split}
\sigma\{\tilde T_n^\infty: n \in \bN\}
& =
\bigcap_{m \in \bN} \sigma\{\tilde T_n^\infty: n \ge m\} \\
& =
\sigma\{L_p : p \in \bN\} \vee \bigcap_{m \in \bN} \sigma\{T_n^\infty: n \ge m\},
\quad \text{$\bP$-a.s.} \\
\end{split}
\]
\end{lemma}

\begin{proof}
Because $\tilde T_m^\infty$ is a measurable function of $\tilde T_n^\infty$ for $m < n$,
the first two $\sigma$-fields are clearly equal, and since $(L_n)_{n \in \bN}$ and
$(T_n^\infty)_{n \in \bN}$ are measurable functions of $(\tilde T_n^\infty)_{n \in \bN}$,
it is also clear that these two $\sigma$-fields
 both contain the third $\sigma$-field.  Now $\tilde T_m^\infty$
is a measurable function of $T_n^\infty$ and $L_1, \ldots, L_n$ for $m < n$, and so to complete
the proof it suffices to show that
\[
\begin{split}
& \sigma\{L_p : p \in \bN\} \vee \bigcap_{m \in \bN} \sigma\{T_n^\infty: n \ge m\} \\
& \quad \supseteq 
\bigcap_{m \in \bN} \left(\sigma\{L_p : p \in \bN\} \vee  \sigma\{T_n^\infty: n \ge m\}\right), 
\quad \text{$\bP$-a.s.} \\
\end{split}
\]

That is, setting $\cF := \sigma\{L_p : p \in \bN\}$, $\cG_m := \sigma\{T_n^\infty: n \ge m\}$,
and $\cG_\infty := \bigcap_{m \in \bN} \cG_m$, it is enough to establish that
\[
\bigcap_{m \in \bN} \left(\cF \vee \cG_m\right) = \cF \vee \cG_\infty.
\]
Let $(\omega, A) \mapsto \bP^\cF(\omega, A)$, $\omega \in \Omega$, $A \in \cG_1$,
 be the conditional probability kernel on $\cG_1$ given $\cF$.
Because each $\sigma$-field $\cG_m$ is countably generated, the desired equality will follow from 
the implication (d) $\Longrightarrow$ (a) of the main theorem
of \cite{MR699981} if we can show that there is a countably generated $\sigma$-field
$\cH$ such that $\cG_\infty = \cH \mod \bP^\cF(\omega, \cdot)$ for $\bP$-a.e. $\omega \in \Omega$.
Because $\cF$ and $\cG_\infty$ are independent, $\bP^\cF(\omega, \cdot)$ restricted to $\cG_\infty$
coincides with $\bP$ restricted to $\cG_\infty$ for $\bP$-a.e. $\omega \in \Omega$. Now $(\omega, A) \mapsto \bP(A)$, 
$\omega \in \Omega$, $A \in \cG_1$, is certainly the conditional probability kernel on $\cG_1$ given the
trivial $\sigma$-field $\{\emptyset, \Omega\}$.  Moreover, since 
\[
\bigcap_{m \in \bN} \left(\{\emptyset, \Omega\} \vee \cG_m\right) = \{\emptyset, \Omega\} \vee \cG_\infty
\]
obviously holds, it follows from the implication (a) $\Longrightarrow$ (d) of the main theorem
of \cite{MR699981} that such a countably generated $\cH$ does indeed exist.  Alternatively, because
$\cG_1$ is countably generated,  $L^1(\Omega,\cG_1, \bP)$ contains a countable dense subset $C$.
Let $D$ be a collection of random variables that contains a version of $\bE[\xi \, | \, \cG_\infty]$
for each $\xi \in C$.  It is clear that $D$ is dense in $L^1(\Omega, \cG_\infty, \bP)$ and it suffices
to take $\cH$ to be the $\sigma$-field generated by $D$.
\end{proof}

We now want to use the labeled infinite R\'emy bridge to build an infinite 
binary-tree-like structure for which the set $\bN$ plays the
role of the leaves.   Interior vertices in this infinite 
binary-tree-like structure have a left child and a right child,
but we show that if we forget about this ordering, then there is
an $\bR$-tree such that, loosely speaking, the tree-like structure is that of the
tree spanned by countably many points picked
independently according to a certain probability
measure on the $\bR$-tree.  The $\bR$-tree is nonrandom when
the infinite R\'emy bridge is extremal, but even in that case further randomization may
be necessary to reconstitute the left versus right ordering of children.
 
 As will become clear in Section~\ref{S:leftright}, 
 Example~\ref{E:coin_tossing_bridge} 
 gives rise to a situation in which additional randomization 
 is required to ``distinguish left from right'' 
 after the countable collection of points has been sampled
 in order to fully reconstitute
the binary tree-like structure; 
that is, it is not possible to impose a planar structure on
the $\bR$-tree so that the left versus right ordering
of children in the subtree spanned by the
sampled points is inherited from the planar structure on the $\bR$-tree.
The essential point here is that there is no Borel subset
$A$ of the unit interval with Lebesgue measure $\frac{1}{2}$
such that if $U$ is a uniform random variable on
the unit interval the random variables $U$ and $\ind_A(U)$
are independent.

However, no such additional randomization is necessary in 
Example~\ref{E:MBfulltree} and the associated $\bR$-tree can be augmented
with a planar structure that induces the desired one on the subtree spanned
by the sampled points.

\begin{definition}
If $i,j \in [n+1]$ are the labels of two leaves of $T_n^\infty$ that
are represented by the words
$u_1 \ldots u_k$ and $v_1 \ldots v_\ell$ in $\{0,1\}^\star$, then
set $[i,j]_n := u_1 \ldots u_m = v_1 \ldots v_m$, where 
$m:= \max\{h : u_h = v_h\}$.   That is, $[i,j]_n$ is the 
{\em most recent common ancestor} in $T_n^\infty$ of the leaves labeled
$i$ and $j$.  Note that $[i,i]_n$ is just the leaf labeled $i$ and every
internal vertex of $T_n^\infty$ is of the form $[i,j]_n$ for at least one
pair $(i,j)$ with $i \ne j$.
\end{definition}

\begin{definition}
Define an equivalence relation $\equiv$ on the Cartesian product
$\bN \times \bN$ by declaring that
$(i',j') \equiv (i'',j'')$ if and only if $[i',j']_n = [i'',j'']_n$
for some (and hence all) $n$ such that $i',j',i'',j'' \in [n+1]$.
We write $\langle i,j \rangle$ for the equivalence class of the pair $(i,j)$.
We will see that we can
think of the equivalence classes as being the vertices of
a binary-tree-like object.  For $i \in \bN$
the equivalence class of the pair $(i,i)$ has only one element
and it will sometimes be convenient to denote this equivalence class
simply by $i$.  With this convention, we regard $\langle i,j \rangle$
as being the most recent common ancestor of the leaves $i$ and $j$.
\end{definition}

\begin{definition}
Define a partial order $<_L$
on the set of equivalence classes by
declaring for $(i',j'), (i'',j'') \in \bN \times \bN$
that
$\langle i',j' \rangle <_L \langle i'',j''\rangle$ if and only if 
for some (and hence all) $n$ such that $i',j',i'',j'' \in [n+1]$
we have $[i',j']_n = u_1 \ldots u_k$
and $[i'',j'']_n = u_1 \ldots u_k   0 v_1 \ldots v_\ell$ 
for some $u_1, \ldots, u_k, v_1, \ldots, v_\ell \in \{0,1\}$.
We interpret the ordering $\langle i',j' \rangle <_L \langle i'',j''\rangle$
as the ``vertex'' $\langle i'',j''\rangle$ being below and to the left
of the ``vertex'' $\langle i',j' \rangle$.
Similarly, we define another partial order $<_R$ by
declaring 
that
$\langle i',j' \rangle <_R \langle i'',j''\rangle$ if and only if 
for some (and hence all) $n$ such that $i',j',i'',j'' \in [n+1]$
we have $[i',j']_n = u_1 \ldots u_k$
and $[i'',j'']_n = u_1 \ldots u_k   1 v_1 \ldots v_\ell$ 
for some $u_1, \ldots, u_k, v_1, \ldots, v_\ell \in \{0,1\}$.
We interpret the ordering $\langle i',j' \rangle <_R \langle i'',j''\rangle$
as the ``vertex'' $\langle i'',j''\rangle$ being below and to the right
of the ``vertex'' $\langle i',j' \rangle$.
\end{definition}

\begin{remark}
\label{R:didendritic_properties}
The equivalence relation $\equiv$ and the partial orders $<_L$ and
$<_R$ have a number of simple properties that it is useful to record.
\begin{itemize}
\item
For $i,j \in \bN$, $(i,j) \equiv (j,i)$.
\item
For $i,j, k\in \bN$, $(i,j) \not \equiv (k,k)$ unless $i=j=k$.
\item
For $i,j \in \bN$ with $i \ne j$, either
$\langle i,j \rangle <_L \langle i,i \rangle$
and
$\langle i,j \rangle <_R \langle j,j \rangle$,
or
$\langle i,j \rangle <_R \langle i,i \rangle$
and
$\langle i,j \rangle <_L \langle j,j \rangle$.
\item
For $h,i,j,k \in \bN$, if
$\langle h,i \rangle <_L \langle j,k \rangle$, then
$\langle h,i \rangle \not <_R \langle j,k \rangle$.
\item
For $h,i,j,k \in \bN$, if
$\langle h,i \rangle <_R \langle j,k \rangle$, then
$\langle h,i \rangle \not <_L \langle j,k \rangle$.
\item
For $f,g,h,i,j,k \in \bN$, if
$\langle f,g \rangle <_L \langle h,i \rangle$ and
$\langle h,i \rangle <_R \langle j,k \rangle$, then
$\langle f,g \rangle <_L \langle j,k \rangle$.
\item
For $f,g,h,i,j,k \in \bN$, if
$\langle f,g \rangle <_R \langle h,i \rangle$ and
$\langle h,i \rangle <_L \langle j,k \rangle$, then
$\langle f,g \rangle <_R \langle j,k \rangle$.
\end{itemize}
\end{remark}
%
\begin{definition}\label{D:didendritic_properties}An equivalence relation on $\bN \times \bN$ and two partial orders
on the associated equivalence classes form a {\em didendritic system}
if they satisfy the conditions listed in Remark~\ref{R:didendritic_properties}.\footnote{This definition does not capture the class of objects we intend it to.
An additional axiom is needed so that what follows, particularly Remark~\ref{R:didendritic_to_bridge},
is correct.  See Section~\ref{S9_2}.}
 (We
have coined the word ``didendritic'' from the Greek roots 
``$\delta\iota\varsigma$'' = ``two, twice or double'' and
``$\delta\epsilon\nu\delta\rho\iota\tau\eta\varsigma$'' 
= ``of or pertaining to a tree, tree-like'' 
as an adjective meaning ``binary tree-like''.)
\end{definition}

\begin{notation}
From now on, we will use the notations $\equiv$, 
$\langle \cdot, \cdot \rangle$, $<_L$, and $<_R$
to denote the equivalence relation, equivalence classes,
and the two partial orders of an arbitrary didendritic system.
\end{notation}

%
%

\begin{remark}
\label{R:didendritic_to_bridge}
Given a didendritic system 
$(\equiv, \, \langle \cdot, \cdot \rangle, \, <_L, \, <_R)$
and $n \in \bN$, there is a unique binary tree with $n+1$
leaves labeled by $i = \langle i, i \rangle$, $i \in [n+1]$,
and internal vertices labeled by $\langle i,j \rangle$, 
$i,j \in [n+1]$, $i \ne j$.  Using the representation of
binary trees as subsets of $\{0,1\}^*$, the root $\emptyset$
is labeled by the unique equivalence class $\langle p,q \rangle$,
$p,q \in [n+1]$,
such that there is no equivalence class 
$\langle r,s \rangle$, $r,s \in [n+1]$, for which 
$\langle r,s \rangle <_L \langle p,q \rangle$
or
$\langle r,s \rangle <_R \langle p,q \rangle$.
If the equivalence class
$\langle h,i \rangle$, $h,i \in [n+1]$, is the label of 
the vertex $v_1 \ldots v_r$ of the tree and the equivalence class
$\langle j,k \rangle$, $j,k \in [n+1]$
is such that $\langle h,i \rangle <_L \langle j,k \rangle$
(respectively, $\langle h,i \rangle <_R \langle j,k \rangle$)
and there is no equivalence class $\langle \ell, m \rangle$
with $\langle h,i \rangle <_L \langle \ell, m \rangle$ 
(respectively, and $\langle h,i \rangle <_R \langle \ell, m \rangle$)
and
$\langle \ell, m \rangle <_L \langle j,k \rangle$ 
or
$\langle \ell, m \rangle <_R \langle j,k \rangle$,
then
$v_1 \ldots v_r 0$ (respectively, $v_1 \ldots v_r 1$)
is a vertex of the tree with label $\langle j,k \rangle$.

If a labeled binary tree is constructed in this way and another
one is constructed from the same didendritic system with $n$ replaced by
$n+1$, then the first labeled binary tree can be produced 
from the second as follows.
\begin{itemize}
\item
The leaf labeled $n+2 = \langle n+2, n+2 \rangle$ is deleted, 
along with its sibling (which may or may not be a leaf).
\item 
If the sibling of the leaf labeled $n+2$ is also a leaf, then 
the common parent (which is now a leaf) is assigned the sibling's label.
\item
If the sibling of the leaf labeled $n+2$ is not a leaf, then the subtree
below the sibling is attached to the common parent.  The labelings
of the vertices in the subtree are unchanged and the common parent is
assigned the sibling's label.
\end{itemize}
\end{remark}

\begin{definition}\label{def5_8}
Given a didendritic system 
$\DD = (\equiv, \, \langle \cdot, \cdot \rangle, \, <_L, \, <_R)$
and a permutation $\sigma$ of $\bN$ such
that $\sigma(i) = i$ for all but finitely many $i \in \bN$,
the didendritic system 
$\DD^\sigma = (\equiv^\sigma, \, \langle \cdot, \cdot \rangle^\sigma, \, <_L^\sigma, \, <_R^\sigma)$ is defined by
\begin{itemize}
\item
$(i',j') \equiv^\sigma (i'',j'')$ 
if and only if 
$(\sigma(i'), \sigma(j')) \equiv (\sigma(i''), \sigma(j''))$,
\item
$\langle i, j \rangle^\sigma$ is the equivalence class of
the pair $(i,j)$ for the equivalence relation $\equiv^\sigma$,
\item
$\langle h,i \rangle^\sigma <_L^\sigma \langle j,k \rangle^\sigma$
if and only if
$\langle \sigma(h), \sigma(i) \rangle <_L \langle \sigma(j), \sigma(k) \rangle$,
\item
$\langle h,i \rangle^\sigma <_R^\sigma \langle j,k \rangle^\sigma$
if and only if
$\langle \sigma(h), \sigma(i) \rangle <_R \langle \sigma(j), \sigma(k) \rangle$.
\end{itemize}
A random didendritic system 
$\DD = (\equiv, \, \langle \cdot, \cdot \rangle, \, <_L, \, <_R)$
is {\em exchangeable} if for each permutation $\sigma$ of $\bN$ such
that $\sigma(i) = i$ for all but finitely many $i \in \bN$ 
the random didendritic system 
$\DD^\sigma$
has the same distribution as
$\DD$.
\end{definition}

In view of the similarity of the procedure described before 
Definition~\ref{def5_8} and the procedure described at the beginning of  this section,  
the following result is obvious and shows that characterizing the
family of infinite R\'emy bridges is equivalent to characterizing the
family of exchangeable random didendritic systems.

\begin{lemma}
The random didendritic system  corresponding to the labeled version
of an infinite R\'emy bridge is exchangeable.  Conversely, the sequence of random 
labeled binary trees produced from an
exchangeable random didendritic system
by the procedure described in
Remark~\ref{R:didendritic_to_bridge} is an infinite R\'emy bridge.
\end{lemma}

With this result in mind, we now explore what sort of information
is required to uniquely specify a didendritic system.
From Remark~\ref{R:didendritic_to_bridge},
the subtree spanned by three distinct labeled leaves
$i,j,k \in \bN$ is one of twelve isomorphism types that
we depict in Figure~\ref{fig:all_three_leaf_trees} along with
notations for each one. 

\begin{figure}
	\centering
\includegraphics[width=9cm]{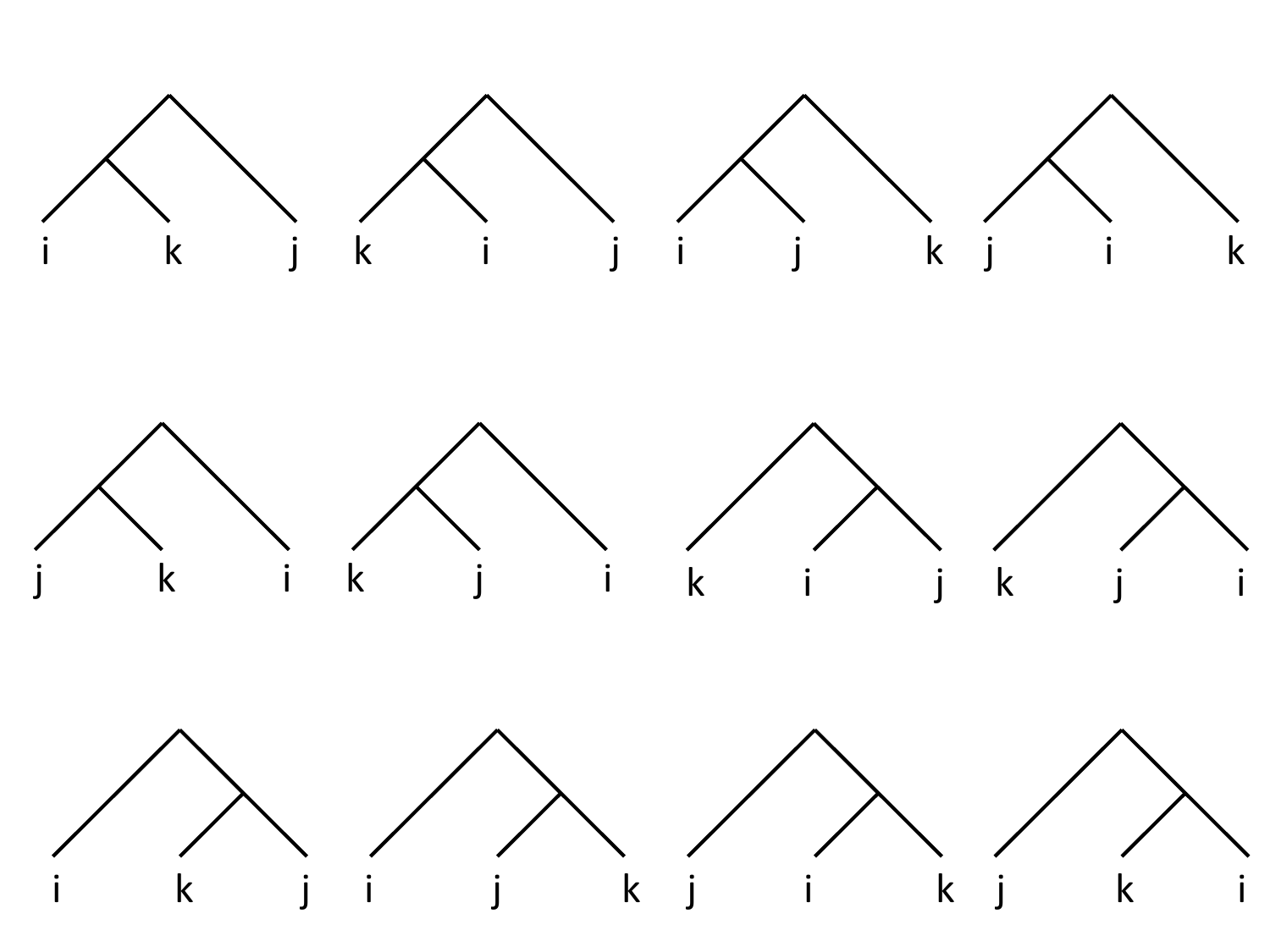}
\caption{\small The isomorphism types for the subtree spanned by $3$ leaves
of a leaf--labeled binary tree.  Going left to right and from top to bottom,
we denote these types by $((i,k),j)$, $((k,i),j)$, $((i,j),k)$,
$\ldots$, $(i,(j,k))$, $(j,(i,k))$, $(j,(k,i))$.}
	\label{fig:all_three_leaf_trees}
\end{figure}

\begin{lemma}
\label{L:puzzling}
Any didendritic system 
$(\equiv, \, \langle \cdot, \cdot \rangle, \, <_L, \, <_R)$
is uniquely determined by the isomorphism types of the
subtrees spanned by all triples of distinct labeled leaves.
\end{lemma}

\begin{proof}
Observe that, $\langle h,i \rangle <_L \langle j,k \rangle$
for $h,i,j,k \in \bN$
if and only if either one of the following six conditions
holds or one of the three
similar sets of six conditions with the roles of $h$ and $i$ 
interchanged or the roles of $k$ and $j$ interchanged holds:
\begin{itemize}
\item
$((i,j),h) \, \& \, ((i,k),h) \, \& \, ((j,k),h) \, \& \, (i,(j,k))$,
\item
$((j,i),h) \, \& \, ((k,i),h) \, \& \, ((j,k),h) \, \& \, ((j,k),i)$,
\item
$((i,j),h) \, \& \, ((i,k),h) \, \& \, ((j,k),h) \, \& \, ((i,j),k)$,
\item
$((j,i),h) \, \& \, ((i,k),h) \, \& \, ((j,k),h) \, \& \, ((j,i),k)$,
\item
$((j,i),h) \, \& \, ((i,k),h) \, \& \, ((j,k),h) \, \& \, (j,(i,k))$,
\item
$((j,i),h) \, \& \, ((k,i),h) \, \& \, ((j,k),h) \, \& \, (j,(k,i))$.
\end{itemize}

\medskip
\noindent Moreover, 
$(h,i) \equiv (j,k)$ (that is, $\langle h, i \rangle = \langle j,k \rangle$
if and only if for all $\ell,m \in \bN$,
$\langle h, i \rangle <_L \langle \ell, m \rangle
\Longleftrightarrow
\langle j, k \rangle <_L \langle \ell, m \rangle$
and
$\langle h, i \rangle <_R \langle \ell, m \rangle
\Longleftrightarrow
\langle j, k \rangle <_R \langle \ell, m \rangle$.
\end{proof}

\begin{remark}
\label{R:didentritic_array}
It follows from Lemma~\ref{L:puzzling} that any didendritic system has a unique
coding as an array indexed by $\{(i,j,k) \in \bN^3: i,j,k \; \text{distinct}\}$,
where the $(i,j,k)$ entry records the isomorphism type of the subtree spanned
by the leaves labeled $i,j,k$.  The triply indexed random array
corresponding to an exchangeable random didendritic system is jointly
exchangeable in the usual sense for random arrays 
(see, for example, \cite[Section 7.1]{MR2161313}).
\end{remark}

\begin{definition}\label{D:partialorder}
Define a third partial order $<$ on the set
of equivalence classes of $\bN \times \bN$
by declaring that
$\langle h,i \rangle < \langle j,k \rangle$
if either
$\langle h,i \rangle <_L \langle j,k \rangle$
or
$\langle h,i \rangle <_R \langle j,k \rangle$.
We interpret the ordering 
$\langle h,i \rangle < \langle j,k \rangle$
as the ``vertex'' $\langle j,k \rangle$ being
below the ``vertex'' $\langle h,i \rangle$.
\end{definition}

\begin{remark}
It is easy to see that if $\langle h,i \rangle$
and $\langle j,k \rangle$ are two equivalence classes,
then there is a unique ``most recent common ancestor''
$\langle \ell, m \rangle$ such that
$\langle \ell, m \rangle \le \langle h,i \rangle$,
$\langle \ell, m \rangle \le \langle j,k \rangle$,
and if $\langle p,q \rangle$ also has these two properties,
then $\langle p,q \rangle \le \langle \ell, m \rangle$.
Moreover, we can choose $\ell,m$ so that
$\ell \in \{h,i\}$ and $m \in \{j,k\}$.  Indeed, for any
$n \in \bN$ we can, by Remark~\ref{R:didendritic_to_bridge}, 
think of the equivalence classes
$\{\langle i, j \rangle : i,j \in [n+1]\}$ as the vertices
of a binary tree with its leaves labeled by $[n+1]$.
When the didendritic system was constructed 
from the labeled version $(\tilde T_n^\infty)_{n\in \mathbb N}$ 
of an infinite R\'emy bridge $(T_n^\infty)_{n\in \mathbb N}$,
this leaf--labeled binary tree is just $\tilde T_n^\infty$. 
\end{remark}

\begin{lemma}
\label{L:partial_order_plus_left-right}
Any didendritic system 
$(\equiv, \, \langle \cdot, \cdot \rangle, \, <_L, \, <_R)$
is uniquely determined by the equivalence relation $\equiv$,
the partial order $<$, and a determination for each pair
of distinct labeled leaves $i,j \in \bN$ whether
\[
\begin{split}
\langle i,j \rangle <_L i \quad & \text{and} \quad  \langle i,j \rangle <_R j \\
                                & \text{or} \\
\langle i,j \rangle <_L j \quad & \text{and} \quad  \langle i,j \rangle <_R i. \\	
\end{split}
\]
\end{lemma}

\begin{proof}
Because of Lemma~\ref{L:puzzling}, it suffices to show that it is possible
to reconstruct from the given data the isomorphism types 
of the subtrees spanned by all triples of distinct labeled leaves.
For distinct $i,j,k \in \bN$, the isomorphism type assignment $((i,k),j)$ is equivalent to
\[
\langle i,j \rangle = \langle k,j \rangle < \langle i,k \rangle
\]
and
\[
\begin{split}
\langle i,k \rangle <_L i & \quad \&  \quad \langle i,k \rangle <_R k \\
\langle i,j \rangle <_L i & \quad \&  \quad \langle i,j \rangle <_R j \\
\langle k,j \rangle <_L k & \quad \&  \quad \langle k,j \rangle <_R j. \\
\end{split}
\]
Similar observations for the other eleven isomorphism types establish the result.
\end{proof}

\begin{remark}
\label{R:jointly_exchangeable}
We have seen that any infinite R\'emy bridge $(T_n^\infty)_{n \in \bN}$
has a uniquely defined labeled version $(\tilde T_n^\infty)_{n \in \bN}$
(in the sense that the distribution of the sequence
$(\tilde T_n^\infty)_{n \in \bN}$ is uniquely
specified by the distribution of the sequence $(T_n^\infty)_{n \in \bN}$)
and also that a labeled infinite R\'emy bridge corresponds, via a bijection
between infinite bridge paths and didendritic systems,
to a unique exchangeable random didendritic system. 

Our aim is to find concrete representations of the extremal infinite R\'emy bridges 
(recall that an infinite R\'emy bridge is extremal if it has a trivial tail $\sigma$-field). 
To this end, it will be useful to relate the extremality of an infinite R\'emy bridge to 
properties of the associated exchangeable random didendritic system.
We say that an exchangeable random didendritic system $\DD$ is {\em ergodic} if 
\[
\bP(\{\DD \in A\} \triangle \{\DD^\sigma \in A\}) = 0
\] 
for all permutations $\sigma$ of $\bN$ such
that $\sigma(i) = i$ for all but finitely many $i \in \bN$ implies that 
\[
\bP\{\DD \in A\} \in \{0,1\}.
\] 
By classical results on ergodic decompositions (see, for example \cite[Theorem~A 1.4]{MR2161313}),
an exchangeable random didendritic system with distribution $\epsilon$ is ergodic
if and only if there is no decomposition 
$\epsilon = p' \epsilon' + p'' \epsilon''$,
where $\epsilon', \epsilon''$ are distinct distributions 
of exchangeable random didendritic systems, $p',p'' > 0$, and $p' + p'' = 1$.
Also, it follows from Remark~\ref{R:didentritic_array} and a result of Aldous
(see, for example, \cite[Lemma~7.35]{MR2161313}) that ergodicity 
is further equivalent to the independence of the exchangeable random didendritic systems induced
by disjoint subsets of $\bN$, where here we extend the definition of a didendritic
system in the obvious manner to allow an equivalence relation and partial orders that are defined on an underlying
countable (possibly finite) set other than $\bN$.
\end{remark}

\begin{proposition}
\label{P:extremal_vs_ergodic}
An infinite R\'emy bridge is extremal if and only if
the associated exchangeable random didendritic system is ergodic.
\end{proposition}

\begin{proof}
Let $\cT$ 
be the set of sequences of binary trees 
that can arise as a sample path of an infinite R\'emy bridge  and 
let $\tilde \cT$ 
be the set of sequences of leaf--labeled binary trees that
can arise as a sample path of a labeled infinite R\'emy bridge. 

The distribution $\alpha$ of an infinite R\'emy bridge has a unique representation
of the form $\alpha = \int \bQ^y \, \mu(dy)$ for a probability measure $\mu$
concentrated on the minimal Doob--Martin boundary of the R\'emy chain,
where $\bQ^y$ is the distribution of the infinite R\'emy bridge
corresponding to the boundary point $y$.
The infinite R\'emy bridge is extremal if and only if $\mu$ is
a point mass, which is in turn equivalent to the condition
that  it is not possible to write
$\alpha = p' \alpha' + p'' \alpha''$, where $\alpha', \alpha''$
are distributions of infinite R\'emy bridges,
$p',p'' > 0$, and $p' + p'' = 1$.

Recall from Remark~\ref{R:Qy}
that if $\alpha = \int \bQ^y \, \mu(dy)$ is the distribution of
an infinite R\'emy bridge, where
$\bQ^y$ is the distribution of the infinite R\'emy bridge
corresponding to the boundary point $y$, 
then $\Lambda(\alpha) := \int \tilde \bQ^y \, \mu(dy)$
is the distribution of the associated labeled infinite R\'emy bridge,
where $\tilde \bQ^y$ is the distribution of the labeled infinite R\'emy bridge
corresponding to the boundary point $y$.
Writing $\phi: \tilde \cT \to \cT$ for the map that removes the labels from each tree
in a path, we see that the map $\Lambda$ is bijective with inverse $\Upsilon$
given by $\Upsilon(\tilde \alpha) = \tilde \alpha \circ \phi^{-1}$
when $\tilde \alpha$ is the distribution of a labeled infinite R\'emy bridge.

It is clear that if $\alpha, \alpha', \alpha''$
are distributions of infinite R\'emy bridges,
$p',p'' > 0$, $p' + p'' = 1$, and
$\alpha = p' \alpha' + p'' \alpha''$,
then $\Lambda(\alpha) = p' \Lambda(\alpha') + p'' \Lambda(\alpha'')$.
Similarly, if $\tilde \alpha, \tilde \alpha', \tilde \alpha''$
are distributions of labeled infinite R\'emy bridges,
$p',p'' > 0$, $p' + p'' = 1$, and
$\tilde \alpha = p' \tilde \alpha' + p'' \tilde \alpha''$,
then $\Upsilon(\tilde \alpha) = p' \Upsilon(\tilde \alpha') + p'' \Upsilon(\tilde \alpha'')$.
In short, an infinite R\'emy bridge has a nontrivial
tail $\sigma$-field if and only if it is distributed as a nontrivial mixture
of infinite R\'emy bridges, and this in turn is equivalent to the associated
labeled infinite R\'emy bridge being distributed as a nontrivial mixture
of labeled infinite R\'emy bridges.

Let $\cD$ be the set of  didendritic systems.
Write $\psi: \tilde \cT \to \cD$ for the map that takes
a sequence that can arise as a sample path of a labeled
infinite R\'emy bridge and turns it into a didendritic system.

Because $\psi$ is a bijection,
a probability measure $\gamma$ on $\tilde \cT$ that is the distribution
of a labeled infinite R\'emy bridge is a nontrivial mixture 
$\gamma = p' \gamma' + p'' \gamma''$,
where $p',p'' > 0$, $p' + p'' = 1$, and $\gamma', \gamma''$ are distinct
distributions of labeled infinite R\'emy bridges, if and only if
$\gamma \circ \psi^{-1} = p' \epsilon' + p'' \epsilon''$, where
$\epsilon', \epsilon''$ are distinct probability measures
on $\cD$ (in which case $\epsilon' = \gamma' \circ \psi^{-1}$
and $\epsilon'' = \gamma'' \circ \psi^{-1}$).  

Combining all of the above equivalent conditions establishes the result.
\end{proof}

\begin{remark}
\label{R:extremal_vs_ergodic}
The equivalence of Proposition~\ref{P:extremal_vs_ergodic} is central 
to the subsequent development and so we sketch the following alternative
``bare hands'' proof that is also interesting in its own right.\footnote{This alternative proof is incorrect, see Section \ref{S9_1}.}

Consider an infinite R\'emy bridge $(T_n^\infty)_{n \in \bN}$,
its labeled version $(\tilde T_n^\infty)_{n \in \bN}$, 
the corresponding sequence $(L_n)_{n \in \bN}$ defined
in Definition~\ref{D:choice_variables}, and 
the associated exchangeable random didendritic system
$\DD = (\equiv, \, \langle \cdot, \cdot \rangle, \, <_L, \, <_R)$.

Fix $m \in \bN$.  For $n \ge m$, let $\tilde T_{m,n}^\infty$ be the random partially leaf--labeled
tree that is obtained from $\tilde T_n^\infty$ 
by removing those labels that belong to $[m+1]$.
Thus, $m+1$ leaves of $\tilde T_{m,n}^\infty$ have no labels and the remaining 
$(n+1) - (m+1)$ leaves are labeled by elements of $[n+1] \setminus [m+1]$.
The $\sigma$-field consisting of events of the form $\{\DD \in A\}$
where $A$ is such that $\bP(\{\DD \in A\} \triangle \{\DD^\sigma \in A\}) = 0$
for all permutations $\sigma$ of $\bN$ that fix $\bN \setminus [m+1]$
is $\bP$-a.s. equal to 
$\sigma\{\tilde T_{m,n}^\infty: n \ge m\}$.

To establish Proposition~\ref{P:extremal_vs_ergodic}
it will therefore suffice to show that the $\sigma$-field
$\bigcap_{m \in \bN} \sigma\{\tilde T_{m,n}^\infty: n \ge m\}$
is $\bP$-trivial if and only if the $\sigma$-field
$\bigcap_{m \in \bN} \sigma\{T_n^\infty: n \ge m\}$
is $\bP$-trivial.  The former $\sigma$-field
contains the latter, and hence it further suffices to show
that if the latter $\sigma$-field is  $\bP$-trivial, then so is the former.
We therefore suppose from now on that 
$\bigcap_{m \in \bN} \sigma\{T_n^\infty: n \ge m\}$
is $\bP$-trivial.

For any $m \le n \le p$, the random partially leaf--labeled tree 
$\tilde T_{m,n}^\infty$ is a measurable function of $L_{m+1}, \ldots, L_p$ and
$T_p^\infty$, so that
\[
\sigma\{\tilde T_{m,n}^\infty: n \ge m\} 
\subseteq 
\sigma\{L_k : k > m\} 
\vee 
\sigma\{T_q^\infty : q \ge p\},
\]
for any $p \ge m$.  

An argument similar to that in the proof of Lemma~\ref{L:tail_relation} 
combined with the $\bP$-triviality of $\bigcap_{p \ge m} \sigma\{T_q^\infty : q \ge p\}$
gives 
\[
\begin{split} 
\sigma\{\tilde T_{m,n}^\infty: n \ge m\}
& \subseteq
\bigcap_{p \ge m} \left(\sigma\{L_k : k > m\} \vee \sigma\{T_q^\infty : q \ge p\}\right) \\
& =
\sigma\{L_k : k > m\} \vee \bigcap_{p \ge m} \sigma\{T_q^\infty : q \ge p\} \\
& = 
\sigma\{L_k : k > m\} \quad \text{$\bP$-a.s.} \\
\end{split}
\]

Since $\bigcap_{m \in \bN} \sigma\{L_k : k > m\}$ is $\bP$-trivial by Kolmogorov's
zero-one law, it follows that 
$\bigcap_{m \in \bN} \sigma\{\tilde T_{m,n}^\infty: n \ge m\}$ is also $\bP$-trivial,
as required.
\end{remark}

The next result shows that identifying the Doob--Martin boundary of the
R\'emy chain is equivalent to characterizing the extremal infinite
R\'emy bridges.

\begin{corollary}
\label{C:all_minimal}
If $y$ is an element of the Doob--Martin boundary of the R\'emy chain,
then the corresponding nonnegative harmonic function $K(\cdot,y)$ is extremal;
equivalently, the corresponding infinite R\'emy bridge is extremal.
There is thus a bijective correspondence between the Doob--Martin boundary of the R\'emy chain 
and the set of extremal infinite R\'emy bridges.
\end{corollary}

\begin{proof}
Suppose that $(\tt_p)_{p \in \bN}$ is a sequence of binary trees, 
where $\tt_p$ has $m(\tt_p) + 1$
leaves and $m(\tt_p) \to \infty$ as $p \to \infty$.  Suppose, moreover, that 
$\lim_{p \to \infty} \tt_p = y$ for some $y$ in the Doob--Martin boundary of 
the R\'emy chain.
We have to show that the harmonic function $K(\cdot,y)$ is extremal.  
Writing $(T_n^\infty)_{n \in \bN}$ for the
infinite R\'emy bridge associated with $y$, 
this is equivalent to showing that the tail
$\sigma$-field of $(T_n^\infty)_{n \in \bN}$ is $\bP$-a.s. trivial.  
By
Proposition~\ref{P:extremal_vs_ergodic}, this is further equivalent to establishing
that the exchangeable random didendritic system $\DD$ associated with 
$(T_n^\infty)_{n \in \bN}$
is ergodic, which, as we observed in Remark~\ref{R:jointly_exchangeable}, is the same
as proving that the exchangeable random didendritic systems $\DD$ induces on
disjoint (finite) subsets of $\bN$ are independent (recall from Remark~\ref{R:jointly_exchangeable}
our comment about generalizing the notion of a didendritic system from the setting where the underlying set
is $\bN$ to the setting where the underlying set is an arbitrary countable set).

Recall that
$(T_1^{\tt_p}, \ldots, T_{m(\tt_p)}^{\tt_p})$ denotes the R\'emy bridge to $\tt_p$.
For any $\ell \in \bN$, $T_\ell^{\tt_p}$ converges in distribution to $T_\ell^\infty$ as $p \to \infty$.
We can build a labeled version $(\tilde T_1^{\tt_p}, \ldots, \tilde T_{m(\tt_p)}^{\tt_p})$
of $(T_1^{\tt_p}, \ldots, T_{m(\tt_p)}^{\tt_p})$ in much the same way that we built a labeled version
of an infinite R\'emy bridge: $\tilde T_{m(\tt_p)}^{\tt_p}$ consists of 
the binary tree $T_{m(\tt_p)}^{\tt_p} = \tt_p$ with its $m(t_p)+1$ leaves labeled uniformly at random with
$[m(t_p)+1]$ and the backward evolution of such a labeled finite R\'emy bridge is the same 
as that of the labeled infinite R\'emy bridge.  It is clear that 
$\tilde T_\ell^{\tt_p}$ converges in distribution to $\tilde T_\ell^\infty$ as $p \to \infty$
for all $\ell \in \bN$: indeed,  $\tilde T_\ell^{\tt_p}$ and $\tilde T_\ell^\infty$
are just $T_\ell^{\tt_p}$ and $T_\ell^\infty$, respectively, equipped with uniform random labelings
of their $\ell+1$ leaves by $[\ell+1]$.

Suppose that $\ell \le m(\tt_p)$.
The labeled binary tree $\tilde T_\ell^{\tt_p}$ 
(respectively, $\tilde T_{m(\tt_p)}^{\tt_p}$) 
can be coded bijectively by an
exchangeable random didendritic system $\DD_{\ell,p}$ 
(respectively, $\DD_p$) on the finite set $[\ell+1]$ (respectively, $[m(\tt_p)+1]$),
and $\DD_{\ell,p}$ is the didendritic system on $[\ell+1]$ induced by $\DD_p$.
The labeled tree $\tilde T_\ell^\infty$ can be coded bijectively by
an exchangeable random didendritic system $\DD_{\ell,\infty}$ on the finite set $[\ell+1]$,
and $\DD_{\ell,\infty}$ is the didendritic system on $[\ell+1]$ induced by $\DD$.
It follows from the convergence in distribution of
$\tilde T_\ell^{\tt_p}$ to $\tilde T_\ell^\infty$ as $p \to \infty$ for all $\ell \in \bN$
that $\DD_{\ell,p}$ converges in distribution to $\DD_{\ell,\infty}$
as $p \to \infty$ for all $\ell \in \bN$.

Let $\cI$ denote the set of twelve possible isomorphism types for a labeled binary
tree with three leaves.  We know from Lemma~\ref{L:puzzling} that
$\DD$ can be coded bijectively by a jointly exchangeable random array $\ZZ_\infty$, say,
indexed by $\{(i,j,k) : i,j,k \in \bN, \, \text{$i,j,k$ distinct}\}$
with values in $\cI$.  Similarly,  $\DD_{\ell,p}$, $\DD_p$, and
$\DD_{\ell,\infty}$ can be coded bijectively by arrays that we denote
by $\ZZ_{\ell,p}$, $\ZZ_p$, and $\ZZ_{\ell, \infty}$.  
The array $\ZZ_{\ell,p}$ (respectively, $\ZZ_{\ell,\infty}$)
is just the subarray of $\ZZ_p$ (respectively, $\ZZ_\infty$) consisting of the entries indexed by
$\{(i,j,k) : i,j,k \in [\ell+1], \, \text{$i,j,k$ distinct}\}$.  It follows from the
convergence of $\DD_{\ell,p}$ in distribution to $\DD_{\ell,\infty}$ that
$\ZZ_{\ell,p}$ converges in distribution to $\ZZ_{\ell,\infty}$
as $p \to \infty$ for all $\ell \in \bN$.

Suppose that $H_1, \ldots, H_s$ are disjoint finite subsets of $\bN$.  We need to show that
the exchangeable random didendritic systems that $\DD$ induces on these sets are independent.
This is equivalent to establishing that the subarrays of $\ZZ_\infty$ consisting of
entries indexed by $\{(i,j,k) : i,j,k \in H_r, \, \text{$i,j,k$ distinct}\}$, $1 \le r \le s$,
are independent.  Taking $\ell$ so that $H_1 \sqcup \cdots \sqcup H_s \subseteq [\ell+1]$, this
is the same as proving that the subarrays of $\ZZ_{\ell,\infty}$ consisting
of entries indexed by these same sets of triples are independent.

We can build the array $\ZZ_{\ell,p}$ using the binary tree $\tt_p$
and random variables $\xi_1, \ldots, \xi_{\ell+1}$ that form a sequence of uniform random draws 
without replacement from the leaves of $\tt_p$: the $(i,j,k)$ entry of the array is
the isomorphism type of the subtree of $\tt_p$ spanned by the leaves $\xi_i, \xi_j, \xi_k$.  Let $\dag$
be an element not in $\cI$, take $\zeta_1, \ldots, \zeta_{\ell+1}$ to be independent uniform random draws
(with replacement) from the leaves of $\tt_p$, and define an array $\ZZ_{\ell,p}^\dag$ with the same
index set as $\ZZ_{\ell,p}$ but with values in $\cI \sqcup \{\dag\}$
by letting the $(i,j,k)$ entry of the array be
the isomorphism type of the subtree of $\tt_p$ spanned by the leaves $\zeta_i, \zeta_j, \zeta_k$
if $\zeta_i, \zeta_j, \zeta_k$ are distinct and $\dag$ otherwise.  A familiar coupling argument
shows that it is possible to construct $\xi_1, \ldots, \xi_{\ell+1}$ and $\zeta_1, \ldots, \zeta_{\ell+1}$
on the same probability space in such a way that 
$\bP\{\exists  1 \le i \le \ell +1 : \xi_i \ne \zeta_i\}$ depends on $\tt_p$ only through $m(\tt_p)$
and converges to zero as $m(\tt_p) \to \infty$;
more specifically, we first
construct $\zeta_1, \ldots, \zeta_{\ell+1}$, set 
$(\xi_1, \ldots, \xi_{\ell+1}) = (\zeta_1, \ldots, \zeta_{\ell+1})$
if $\zeta_1, \ldots, \zeta_\ell$ are distinct and let 
$(\xi_1, \ldots, \xi_\ell)$ be some other independent
sequence of uniform draws without replacement from the leaves of $\tt_p$ otherwise.
Thus, $\bP\{\ZZ_{\ell,p} \ne \ZZ_{\ell,p}^\dag\}$
depends on $t_p$ only through $m(t_p)$ and converges to zero as $m(t_p) \to \infty$.
The subarrays of $\ZZ_{\ell,p}^\dag$ consisting of
entries indexed by $\{(i,j,k) : i,j,k \in H_r, \, \text{$i,j,k$ distinct}\}$, $1 \le r \le s$,
are obviously independent because they are built from the binary tree $\tt_p$ and the disjoint collections of
random variables $\{\zeta_i : i \in H_r\}$, $1 \le r \le s$.

Combining the convergence in distribution of $\ZZ_{\ell,p}$ to $\ZZ_{\ell, \infty}$ as $p \to \infty$,
the convergence to zero as $p \to \infty$ of the total variation distance between the distribution of 
$\ZZ_{\ell,p}$ and the distribution of $\ZZ_{\ell,p}^\dag$, and the observation that the subarrays of
$\ZZ_{\ell,p}^\dag$ consisting of entries indexed by $\{(i,j,k) : i,j,k \in H_r, \, \text{$i,j,k$ distinct}\}$, 
$1 \le r \le s$, are independent, it is clear that the subarrays of $\ZZ_{\ell,\infty}$ 
indexed by these same sets of triples are independent, as required.
\end{proof}

\section{A real tree associated with an extremal infinite R\'emy bridge}
\label{S:R_tree}

With Corollary~\ref{C:all_minimal} in hand, the task of
identifying the Doob--Martin boundary of the R\'emy chain
reduces to characterizing the extremal infinite R\'emy bridges,
where we stress that such a characterization will also determine
the topological structure of the boundary because convergence of
boundary points is equivalent to convergence of finite-dimensional
distributions of the corresponding infinite R\'emy bridges.

The construction of Section~\ref{S:N_tree} used the labeled version
of an infinite R\'emy bridge (equivalently, an exchangeable random didendritic system)
to provide an embedding of $\bN$ as the leaves of a tree-like
combinatorial object whose vertices correspond to
equivalence classes of the didendritic system's equivalence relation.  
In this section we embed this tree-like
object into an $\bR$-tree by constructing a metric on the
set of equivalence classes.  We assume throughout this section that
$(\equiv, \, \langle \cdot, \cdot \rangle, \, <_L, \, <_R)$
is an ergodic exchangeable random didendritic system
and that $(\tilde T_n^\infty)_{n \in \bN}$ is the labeled version
of the associated extremal infinite R\'emy bridge.

Consider $i,j \in \bN$.  
For $p \in \bN$ set
\begin{equation}
\label{E:Ip_definition}
I_p := \ind\{ \langle i,j \rangle \le p\} 
\end{equation}
(recall our convention of writing $p$
for the equivalence class $\langle p,p \rangle$). 

By construction, the sequence of random
variables $(I_p)_{p > i \vee j}$ is exchangeable.  
Hence, by de Finetti's theorem and the strong law of large numbers,
\begin{equation}
\label{E:distance_definition}
d(i, j)
:=
\lim_{n \to \infty} \frac{1}{n} 
\sum_{p=1}^n I_p
\end{equation}
exists almost surely.

\begin{lemma}
\label{L:genuine_metric}
Almost surely, 
$d$ is a ultrametric on $\bN$.
That is, almost surely the following hold.
\begin{itemize}
\item
For all $i,j \in \bN$, $d(i,j) \ge 0$, and $d(i,j)=0$
if and only if $i=j$.
\item
For all $i,j \in \bN$, $d(i,j) = d(j,i)$.
\item
For all $i,j,k \in \bN$,
$d(i,k) \le d(i,j) \vee d(j,k)$.
\end{itemize}  
{\em A fortiori}, $d$ is almost surely a metric on $\bN$.
\end{lemma}

\begin{proof}
We first show for fixed distinct $i,j \in \bN$
that $d(i,j) > 0$ almost surely.
By exchangeability, de Finetti's theorem 
and the strong law of large numbers, 
the event $\{d(i,j) = 0\}$
coincides almost surely with the event
$\{I_p = 0 \; \forall p \notin \{i,j\}\}
= \{\nexists p \notin \{i,j\} : \langle i,j \rangle \le p\}$. 
For $i,j \in [n+1]$  the event
$\{I_p = 0 \; \forall p \in [n+1], \, p \notin \{i,j\}\}
=
\{\nexists p \in [n+1] \setminus \{i,j\}  
: \langle i,j \rangle \le p\}$
is the event that in the representation of $\tilde T_n^\infty$ 
as a subset of
$\{0,1\}^*$ labeled by $[n+1]$,
there is an interior vertex $u_1 \ldots u_\ell$ such that
$i$ labels $u_1 \ldots u_\ell 0$ and $j$ labels $u_1 \ldots u_\ell 1$
or {\em vice versa} (that is, the two leaves of $\tilde T_n^\infty$ labeled
by $i$ and $j$ are siblings and form what is often called a ``cherry'').
Now, the number of cherries in $\tilde T_n^\infty$ is at most 
$\lfloor \frac{n+1}{2} \rfloor$, and so the probability that $i$
and $j$ label the leaves of the same cherry is at most
$2 \lfloor \frac{n+1}{2} \rfloor \frac{1}{n+1}\frac{1}{n}$.  Thus,
\[
\bP\{d(i,j) = 0\}
= \lim_{n \to \infty} \bP\{I_p = 0 \; \forall p \in [n+1], \, p \notin \{i,j\}\} 
= 0.
\]

It is clear that almost surely $d(i,i)=0$ and $d(i,j) = d(j,i)$.

Lastly, for $i,j,k \in \bN$ we have that 
$\langle i, j \rangle 
= \langle j, k \rangle \le \langle k, i \rangle$
or one of the two other similar inequalities obtained by
cyclically permuting $i,j,k$ holds.  Therefore,
$d(k,i) \le d(i,j) = d(j,k)$ almost surely
or one of the two other similar inequalities obtained by
cyclically permuting $i,j,k$ holds.
\end{proof}

For $t \in \bR_+$ define an equivalence relation 
$\sim_t$ on $\bN$ by declaring that $i \sim_t j$
if and only if $d(i,j) \le t$.  Note that we can
identify $\bN$ with the equivalence classes of
$\sim_0$.  We now extend the metric $d$
to a metric on the set $\UU^o$ of pairs of the form $(B,t)$, where
$t \in \bR_+$ and $B$ is an equivalence class of $\sim_t$.
Given an equivalence class $A$ 
of $\sim_s$ and  an equivalence class $B$ of $\sim_t$,
set
\[
H((A,s),(B,t)) := 
\inf\{u \ge s \vee t : k \sim_u \ell \; \forall k \in A, \ell \in B\}
\]
and
\[
d((A,s),(B,t)) :=
\frac{1}{2}([H((A,s),(B,t)) - s] + [H((A,s),(B,t)) - t]).
\]
For $i,j \in \bN$ we have $H((\{i\},0), (\{j\},0)) = d(i,j)$
and so $d((\{i\},0), (\{j\},0)) = d(i,j)$, confirming that 
we have an extension of the original definition of $d$.
It is straightforward to check that this extension of
$d$ is a metric on $\UU^o$
that satisfies the {\em four-point condition};
that is, for $4$ elements $w,x,y,z \in \UU^0$ 
at least one of the following
conditions holds
\begin{itemize}
\item
$d(w,x) + d(y,z) \le d(w,y) + d(x,z) = d(w,z) + d(x,y)$,
\item
$d(w,z) + d(x,y) \le d(w,x) + d(y,z) = d(w,y) + d(x,z)$,
\item
$d(w,y) + d(x,z) \le d(w,z) + d(x,y) = d(w,x) + d(y,z)$.
\end{itemize}
It is, moreover, not difficult to show that the metric space
$(\UU^o,d)$ is connected and hence it is an $\bR$-tree 
(see \cite[Example~3.41]{MR2351587} for more details).
The completion $(\UU, d)$ of $(\UU^o,d)$ is
also an $\bR$-tree that is complete and separable.

There is a natural partial
order on the $\bR$-tree $(\UU^o, d)$
defined by the requirement that the pair
$(A,s)$ precedes the pair $(B,t)$ if
$A \supseteq B$ and $s > t$. 
If we consider the subtree of $(\UU, d)$
(equivalently, of $(\UU^o, d)$)
spanned by the set $\{(\{i\},0) : i \in [n+1]\}$, then
combinatorially we have a leaf--labeled tree.
The vertices of this combinatorial tree correspond to
pairs of the form $(B_{ij},d(i,j))$, $i,j\in [n+1]$, where 
$B_{ij}$ is the equivalence class 
$\{k \in \bN : d(i,k) \le d(i,j)\} 
= \{k \in \bN : d(j,k) \le d(i,j)\}$. 
Moreover, the combinatorial tree inherits the partial
order from $(\UU^o, d)$ and the vertex $(B_{ij},d(i,j))$
is the most recent common ancestor of the leaves
$(\{i\},0)$ and $(\{j\},0)$ in this partial order.

We claim that this leaf--labeled tree with its partial order
is isomorphic to  $\tilde T_n^\infty$, with the
vertex $(B_{ij},d(i,j))$ corresponding to the vertex $[i,j]_n$ and,
in particular, the leaf $(\{i\},0)$ corresponding to the leaf $i$.
This is equivalent to showing the following.

\begin{lemma}
For distinct 
$i,j,k \in [n+1]$,
$[i,k]_n = [j,k]_n < [i,j]_n$ if and only if 
$d(i,k) = d(j,k) > d(i,j)$.
\end{lemma}

\begin{proof}
It suffices to show that $[i,k]_n = [j,k]_n$ if and only if
$d(i,k) = d(j,k)$ and $[j,k]_n < [i,j]_n$ if and only if
$d(j,k) > d(i,j)$.  Note that $[i,k]_n = [j,k]_n$
if and only if it is not the case that 
$[i,k]_n < [j,k]_n$ or $[i,k]_n > [j,k]_n$.  Similarly,
$d(i,k) = d(j,k)$ if and only if it is not the case that
$d(i,k) > d(j,k)$ or $d(i,k) < d(j,k)$.  It will thus further
suffice to show for distinct $i,j,k \in [n+1]$ that
$[j,k]_n < [i,j]_n$ if and only if
$d(j,k) > d(i,j)$.

It is clear that if $d(j,k) > d(i,j)$, then
$\langle j,k \rangle < \langle i,j \rangle$ and hence
$[j,k]_n < [i,j]_n$.  For the reverse implication, we certainly
have that $[j,k]_n < [i,j]_n$ 
(and hence $\langle j,k \rangle < \langle i,j \rangle$)
implies that $d(j,k) \ge d(i,j)$, and thus we only need to
rule out the possibility of equality.

By exchangeability, de Finetti's theorem 
and the strong law of large numbers, 
the event $\{\langle j,k \rangle < \langle i,j \rangle
, \, d(j,k) = d(i,j)\}$ coincides almost surely with
the event 
\[
\{\langle j,k \rangle < \langle i,j \rangle\}
\cap \{\nexists p \in \bN \setminus \{k\} : 
\langle j,k \rangle \le p, \; \langle i,j \rangle \not \le p\}.
\]
In order to show that the probability of the latter event is zero,
it suffices to show that for $m \ge n$ the probability of the event
\[
\{[j,k]_m < [i,j]_m\}
\cap \{\nexists p \in [m+1] \setminus \{k\} : 
[j,k]_m \le p, \; [i,j]_m  \not \le p\}
\]
converges to zero as $m \to \infty$.  In words, the last event
occurs when the sibling of the most recent common ancestor
in $\tilde T_m^\infty$ of the leaves labeled $i$ and $j$
is a leaf and that leaf is labeled by $k$.  If we condition
on $T_m^\infty$ and the locations of the leaves labeled
$i$ and $j$, then the conditional probability of the last event
is either $\frac{1}{m-1}$ or $0$, depending on whether the sibling of
the most recent common ancestor of the leaves labeled $i$ and $j$
is a leaf, and so the (unconditional) probability of the last event
certainly converges to zero as $m \to \infty$.
\end{proof}

Write $\TT^o$ for the subtree of $\UU^o$ (and hence of
$\UU$) spanned by the set $\{(\{i\},0) : i \in \bN \}$ 
and let $\TT$ be the closure of
$\TT^o$ in $\UU$.  We denote the restriction of
the metric $d$ to $\TT$ also by $d$.
From the above considerations we infer immediately the following.

\begin{proposition}
There is an injective mapping from the set of equivalence classes 
$\langle i, j \rangle$, $i,j \in \bN$, of the ergodic didendritic system
$(\equiv, \, \langle \cdot, \cdot \rangle, \, <_L, \, <_R)$
into the complete, separable $\bR$-tree $\TT$ constructed above
such that the distance $d(i,j)$ defined by 
\eqref{E:distance_definition} coincides with
the distance in $\TT$ between the images
of equivalence classes $\langle i,i \rangle$
and $\langle j,j \rangle$.
\end{proposition}

From now on we will, with a slight abuse of notation,
think of the equivalence classes $\langle i,j \rangle$, $i,j \in \bN$, 
(including the leaves $i=\langle i,i \rangle$, $i \in \bN$)
as being elements of the $\bR$-tree $(\TT,d)$.

\begin{remark}
Consider two equivalence classes $\langle h,i \rangle$
and $\langle j,k \rangle$.  Recall that the most recent
common ancestor of  $\langle h,i \rangle$
and $\langle j,k \rangle$ is of the form $\langle \ell, m \rangle$,
where $\ell \in \{h,i\}$ and $m \in \{j,k\}$.  In terms
of the metric $d$, $\ell$ and $m$ are any such pair for which
$d(\ell, m) =  d(h,j) \vee d(h,k) \vee d(i,j) \vee d(i,k)$.
We therefore have
\[
\begin{split}
d(\langle h,i \rangle, \langle j,k \rangle)
& =
\frac{1}{2}([d(\ell,m) - d(h,i)] + [d(\ell,m) - d(j,k)]) \\
& =
d(h,j) \vee d(h,k) \vee d(i,j) \vee d(i,k)
- \frac{1}{2}(d(h,i) + d(j,k)). \\
\end{split}
\]
In particular, 
\[
d(i,\langle i,j\rangle) = \frac{1}{2} d(i,j),
\]
as we would expect.
\end{remark}

\begin{remark}
\label{R:root_of_T}
It follows from the construction of $\TT$ that 
$\max\{d(x,y) : x,y \in \TT\} \le 1$.
For $n \in \bN$, let $\rho_n$ be the
most recent common ancestor of $1,2, \ldots, n+1$
with respect to the partial order $<$. Note that
$\rho_n = \langle i, j \rangle \in \TT$
for distinct $i,j \in [n+1]$.  
The successive points $\rho_1, \rho_2, \ldots$ are linearly ordered along
a geodesic ray in $\TT$.  Because $\TT$ is a complete separable $\bR$-tree
with a finite diameter, it follows that
$(\rho_n)_{n \in \bN}$ is a Cauchy sequence and hence
convergent to a point $\rho \in \TT$.  We can, as with
any rooted $\bR$-tree, define a partial
order on $\TT$ by declaring that $x$ precedes
$y$ if and only if $x \ne y$ and $x$ belongs to the geodesic segment
$[\rho,y]$ between $\rho$ and $y$ 
(equivalently, $[\rho,x] \subsetneq [\rho,y]$). 
\end{remark} 

The following result is now immediate. 

\begin{proposition}
\label{P:partial_order_extends}
The partial order on $\TT$ defined by the root $\rho$
extends the partial order
$<$ on the equivalence classes $\{\langle i,j \rangle : i,j \in \bN\}$,
and the most recent common ancestor of $\langle h, i \rangle$ and
$\langle j, k \rangle$ is the equivalence class $\langle \ell, m \rangle$
such that 
$[\rho, \langle \ell, m \rangle] 
= [\rho, \langle h, i \rangle] \cap [\rho, \langle j,k \rangle]$.
\end{proposition}   

\begin{example}
\label{E:coin_tossing_tree}
Consider the infinite R\'emy bridge in Example~\ref{E:coin_tossing_bridge}.  
A concrete realization of the $\bR$-tree $(\TT,d)$ can be constructed as follows.  Let 
$(U_n)_{n \in \bN}$ be a sequence of independent random variables 
that each have the uniform distribution on $[0,1]$.  Take
the interval $[0,\frac{1}{2}]$ 
and build an $\bR$-tree by, for each $n \in \bN$,
attaching one end of a closed line segment
of length $\frac{1}{2} U_n$ to the point 
$\frac{1}{2} U_n \in [0,\frac{1}{2}]$ and labeling
the other end of the line segment with $n$.  The distance between
$i$ and $j$ in the resulting $\bR$-tree is then 
\[
\left|\frac{1}{2} U_i - \frac{1}{2} U_j\right| 
+ \frac{1}{2} U_i + \frac{1}{2} U_j
=
U_i \vee U_j.
\]
For $i \ne j$ we can identify $\langle i,j \rangle$ with 
$\frac{1}{2} (U_i \vee U_j) \in [0,\frac{1}{2}]$.
For $i \ne j$ and $k \ne \ell$
we have $\langle i,j \rangle < k$
if $U_i \vee U_j > U_k$ and 
$\langle i,j \rangle < \langle k, \ell \rangle$
if $U_i \vee U_j > U_k \vee U_\ell$.
Note that, as required, the distance between $i$ and $j$ is
\[
U_i \vee U_j 
= \lim_{n \to \infty} 
\frac{1}{n} \sum_{p=1}^n \ind\{U_i \vee U_j \ge U_p\}
=
\lim_{n \to \infty} 
\frac{1}{n} \sum_{p=1}^n \ind\{\langle i,j \rangle \le p\}.
\]
The root $\rho$ is the point $\frac{1}{2}$ in the interval $[0,\frac{1}{2}]$.
\end{example}

\begin{example}
\label{E:MBfulltree_real}
Consider the infinite R\'emy bridge in Example~\ref{E:MBfulltree}.
A concrete realization of the $\bR$-tree $(\TT,d)$ can be constructed as follows.
Take the complete binary tree $\{0,1\}^*$ and join two elements of the form
$v_1 \ldots v_k$ and $v_1 \ldots v_k v_{k+1}$ with a segment of length
$\frac{1}{2^{k+2}}$.  This gives an $\bR$-tree such that if $u_1 \ldots u_m$
and $v_1 \ldots v_n$ are elements of $\{0,1\}^*$ for which $p = \max\{j: u_j = v_j\}$, 
then the distance between the corresponding points
in the $\bR$-tree is 
\[
\left(\frac{1}{2^{p+2}} + \frac{1}{2^{p+3}} + \cdots + \frac{1}{2^{m+1}} \right)
+
\left(\frac{1}{2^{p+2}} + \frac{1}{2^{p+3}} + \cdots + \frac{1}{2^{n+1}} \right).
\]
We can identify $(\TT,d)$ with the completion of this $\bR$-tree.  There is
a bijective correspondence between $\{0,1\}^\infty$ and the points ``added'' in passing to the completion.
The distance between the points in the completion corresponding to $u_1 u_2 \ldots$ and $v_1 v_2 \ldots$
in $\{0,1\}^\infty$ with $p = \max\{j: u_j = v_j\}$ is 
\[
\left(\frac{1}{2^{p+2}} + \frac{1}{2^{p+3}} + \cdots \right)
+
\left(\frac{1}{2^{p+2}} + \frac{1}{2^{p+3}} + \cdots \right)
=
\frac{1}{2^p}.
\]
\end{example}

\section{The sampling measure on the real tree}
\label{S:probmeas}

Throughout this section, let $(\TT,d)$ be the $\bR$-tree
constructed in Section~\ref{S:R_tree} from 
an ergodic exchangeable random didendritic system
$\DD = (\equiv, \, \langle \cdot, \cdot \rangle, \, <_L, \, <_R)$
(equivalently, from the labeled version $(\tilde T_n^\infty)_{n \in \bN}$ 
of an extremal infinite R\'emy bridge $(T_n^\infty)_{n \in \bN}$).  
Recall from Proposition~\ref{P:partial_order_extends}
that we can extend the partial order $<$ to all of $\TT$.

\begin{definition}
Suppose that $\VV$ is a a complete separable $\bR$-tree with finite diameter.
A {\em leaf} of  $\VV$ is a point $x \in \VV$
such that there do not exist two points $y,z \in \VV$ for which
$x$ is in the interior of the segment between $y$ and $z$. The
$\bR$-tree $\VV$ is spanned by its set of leaves. 

An {\em isolated leaf} of a complete separable
$\bR$-tree $\VV$ is a leaf $x \in \VV$
such that for some $\epsilon$ the open ball of radius $\epsilon$
centered at $x$ is a half-open line segment with $x$ at the closed
end of the segment.  There is a maximal such $\epsilon$ and
we write $[x, \Pi(x))$ for the corresponding half-open line
segment.  For a leaf $x$ that is not isolated, we set $\Pi(x) := x$.

The {\em core} of $\VV$ is the subtree $\Gamma(\VV)$ spanned
by the set of points of the form $\Pi(x)$ as $x$ ranges over
the leaves of $\VV$.  It is not hard to show that
$\Gamma(\VV)$ is a closed $\bR$-tree and that $\Pi(x)$ is the unique point of
$\Gamma(\VV)$ that is closest to the leaf $x$ and so we think of
$\Pi(x)$ as the {\em point of attachment} of $x$ to the core.
Also, if for a leaf $x \in \VV$
we let $\VV^x$ be the closure of the subtree of $\VV$ spanned by the leaves
of $\VV$ other than $x$, then $\Gamma(\VV) = \bigcap_x \VV^x$.
\end{definition}

\begin{lemma}
\label{L:core_properties}
\begin{itemize}
\item[a)]
The core of $\TT$ is the closure of the subtree spanned by the set 
$\{\langle i,j \rangle : i,j \in \bN, \, i \ne j\}$.
\item[b)]
For all $i \in \bN$, 
\[
\begin{split}
d(i,\Pi(i)) 
& = \inf\{d(i,\langle i,j \rangle) : j \in \bN, \, j \ne i\} \\
& = \inf\{d(j,\langle i,j \rangle) : j \in \bN, \, j \ne i\} \\
& = \frac{1}{2} \inf\{d(i,j) : j \in \bN, \, j \ne i\} \\
\end{split}
\]
and if $(j_n)_{n \in \bN}$ is any sequence in $\bN \setminus \{i\}$
such that $\lim_{n \to \infty} d(i, \langle i,j_n \rangle) = d(i,\Pi(i))$,
then $\Pi(i) = \lim_{n \to \infty} \langle i,j_n \rangle$.
\item[c)]
For $i \in \bN$, $\Pi(i) \le i$.
\item[d)]
For $i,j \in \bN$ with $i \ne j$, $\Pi(i) \ne \Pi(j)$.
\item[e)]
For $i,j \in \bN$ with $i \ne j$, the most recent common ancestor of
$\Pi(i)$ and $\Pi(j)$ in the partial order that the core $\Gamma(\TT)$
inherits from $\TT$ is $\langle i,j \rangle$ 
and
\[
d(i,j) 
= 
\lim_{n \to \infty} \frac{1}{n} \sum_{p=1}^n
\ind\{\langle i , j \rangle \le \Pi(p)\}.
\]
\item[f)] Under our standing ergodicity assumption,
the isometry class of
$\Gamma(\TT)$ together with the partial order on $\Gamma(\TT)$
inherited from the partial order $<$ are both constant almost surely.
\end{itemize}
\end{lemma} 

\begin{proof}
Parts (a), (b) and (c) are straightforward and are left to the reader.
For part (d) suppose that $\Pi(i) = \Pi(j)$ for $i \ne j$.
By part (c), $\Pi(i) = \Pi(j) \le \langle i,j \rangle$. Thus, 
by part (b), $\Pi(i) = \Pi(j) = \langle i,j \rangle$
and $d(i,\Pi(i)) = d(i,\Pi(j)) = d(j,\Pi(i)) = d(j,\Pi(j)) =
d(i,\langle i, j \rangle) =  d(j,\langle i, j \rangle) =
\frac{1}{2} d(i,j)$.  This is not possible unless $i$ and $j$ are
both isolated.  By the definition of $d(i,j)$, there are infinitely many
$p \in \bN \setminus \{i,j\}$ such that $\langle i,j \rangle \le p$.
For any such $p$ we must have either 
$\langle i,j \rangle < \langle i,p \rangle$ or
$\langle i,j \rangle < \langle j,p \rangle$, so that
$d(i, \langle i,p \rangle) < d(i, \langle i, j \rangle)$
or
$d(j, \langle j,p \rangle) < d(j, \langle i, j \rangle)$,
but this contradicts $\Pi(i) = \Pi(j) = \langle i,j \rangle$.

Part (e) is also clear and is left to the reader.  

For part (f), note first of all that if 
$\sigma$ is a permutation of $\bN$ such
that $\sigma(i) = i$ for all but finitely many $i \in \bN$ and 
$(\equiv^\sigma, \, \langle \cdot, \cdot \rangle^\sigma, \, <_L^\sigma, \, <_R^\sigma)$ is 
the random didendritic system defined in Definition~\ref{def5_8}, then
the the isometry class of $\Gamma(\TT)$ as a random complete separable metric space is unchanged if we replace
$(\equiv, \, \langle \cdot, \cdot \rangle, \, <_L, \, <_R)$
by
$(\equiv^\sigma, \, \langle \cdot, \cdot \rangle^\sigma, \, <_L^\sigma, \, <_R^\sigma)$.
Our standing ergodicity
assumption gives that the isometry class of $\Gamma(\TT)$ is constant almost surely.

The root $\rho$ defined in Remark~\ref{R:root_of_T}
is an element of $\Gamma(\TT)$.  It is clear that the location of $\rho$
is unchanged if we replace 
$\DD = (\equiv, \, \langle \cdot, \cdot \rangle, \, <_L, \, <_R)$
by
$\DD^\sigma = (\equiv^\sigma, \, \langle \cdot, \cdot \rangle^\sigma, \, <_L^\sigma, \, <_R^\sigma)$, 
and so the restriction of the random partial order $<$ to $\Gamma(\TT)$
is also constant.
\end{proof}

\begin{example}
\label{E:coin_tossing_core}
Consider the $\bR$-tree $\TT$ constructed in
Example~\ref{E:coin_tossing_tree} from the
infinite R\'emy bridge introduced in Example~\ref{E:coin_tossing_bridge}.
The core of $\TT$ is the interval $[0,\frac{1}{2}]$.
\end{example}

Consider the maps $\kappa_-: \bN \to \bN$ and
$\kappa_+: \bN \to \bN$ given
by $\kappa_-(n) = 2n-1$ and $\kappa_+(n) = 2n$, $n \in \bN$.
Define the exchangeable random didendritic systems 
$\DD_- = (\equiv_-, \langle \cdot, \cdot \rangle_-, <_{-,L}, <_{-,R})$
and
$\DD_+ = (\equiv_+, \langle \cdot, \cdot \rangle_+, <_{+,L}, <_{+,R})$
by
\[
\begin{split}
& (h,i) \equiv_- (j,k) \Longleftrightarrow (\kappa_-(h),\kappa_-(i)) \equiv (\kappa_-(j),\kappa_-(k)) \\
& \text{and} \\
& (h,i) \equiv_+ (j,k) \Longleftrightarrow (\kappa_+(h),\kappa_+(i)) \equiv (\kappa_+(j),\kappa_+(k)), \\
\end{split}
\]
\[
\begin{split}
& \text{$\langle i, j \rangle_-$ is the $\equiv_-$ equivalence class of $(i,j)$} \\
& \text{and} \\
& \text{$\langle i, j \rangle_+$ is the $\equiv_+$ equivalence class of $(i,j)$}, \\
\end{split}
\]
\[
\begin{split}
& \langle h,i \rangle_- <_{-,L} \langle j,k \rangle_- 
\Longleftrightarrow 
\langle \kappa_-(h),\kappa_-(i) \rangle <_L \langle \kappa_-(j),\kappa_-(k) \rangle \\
& \text{and} \\
& \langle h,i \rangle_+ <_{+,L} \langle j,k \rangle_+ 
\Longleftrightarrow 
\langle \kappa_+(h),\kappa_+(i) \rangle <_L \langle \kappa_+(j),\kappa_+(k) \rangle, \\
\end{split}
\]
and
\[
\begin{split}
& \langle h,i \rangle_- <_{-,R} \langle j,k \rangle_- 
\Longleftrightarrow 
\langle \kappa_-(h),\kappa_-(i) \rangle <_R \langle \kappa_-(j),\kappa_-(k) \rangle \\
& \text{and} \\
& \langle h,i \rangle_+ <_{+,R} \langle j,k \rangle_+ 
\Longleftrightarrow 
\langle \kappa_+(h),\kappa_+(i) \rangle <_R \langle \kappa_+(j),\kappa_+(k) \rangle. \\
\end{split}
\]
Define the partial orders $<_-$ and $<_+$ on 
$\{\langle i,j \rangle_+:  i,j \in \bN\}$, respectively,
by declaring that 
\[ 
\begin{split}
& \langle h, i \rangle_- <_- \langle j,k \rangle_- 
\Longleftrightarrow 
\langle \kappa_-(h) \kappa_-(i) \rangle < \langle \kappa_-(j), \kappa_-(k) \rangle \\
& \text{and} \\
& \langle h, i \rangle_+ <_- \langle j,k \rangle_+ 
\Longrightarrow 
\langle \kappa_+(h) \kappa_+(i) \rangle < \langle \kappa_+(j), \kappa_+(k) \rangle, \\
\end{split}
\]
or, equivalently,
\[ 
\begin{split}
& \langle h, i \rangle_- <_- \langle j,k \rangle_- 
\Longleftrightarrow
\langle h, i \rangle_- <_{-,L} \langle j,k \rangle_-
\; \text{or} \;
\langle h, i \rangle_- <_{-,R} \langle j,k \rangle_- \\ 
& \text{and} \\
& \langle h, i \rangle_+ <_+ \langle j,k \rangle_+ 
\Longleftrightarrow
\langle h, i \rangle_+ <_{+,L} \langle j,k \rangle_+ 
\; \text{or} \;
\langle h, i \rangle_+ <_{+,R} \langle j,k \rangle_+. \\
\end{split}
\]


By exchangeability, the random didendritic systems
$\DD_-$, $\DD_+$
and
$\DD$
have the same distribution. By the ergodicity of $\DD$,
the random didendritic systems $\DD_-$ and $\DD_+$ are independent.
Therefore,
 if we construct random partially ordered $\bR$-trees 
$(\TT^-, <_-)$ and $(\TT^+, <_+)$ from 
$\DD_-$
and
$\DD_+$
in the same manner
that $(\TT, <)$ was constructed from 
$\DD$,
then $(\TT^-, <_-)$ and $(\TT^+, <_+)$ are independent and each have
the same distribution as $(\TT, <)$. 
 
By de Finetti's theorem and the strong law of large numbers,
\[
\begin{split}
 \lim_{n \to \infty} \frac{1}{n} \sum_{p=1}^n \ind\{\langle i,j \rangle \le p\}
& =
\lim_{n \to \infty} \frac{1}{n} \sum_{p=1}^n \ind\{\langle i,j \rangle \le 2p-1\} \\
& = \lim_{n \to \infty} \frac{1}{n} \sum_{p=1}^n \ind\{\langle i,j \rangle \le 2p\} \\
\end{split}
\]
for any $i,j \in \bN$.  Therefore, the distance between
$k$ and $\ell$ in $\TT^-$
(resp. $\TT^+$)
is the same as the distance between
$\kappa_-(k)$ and $\kappa_-(\ell)$ 
(resp. $\kappa_+(k)$ and $\kappa_+(\ell)$)
in $\TT$,
and hence we may (and will) identify $\TT^-$ and $\TT^+$ with the closures
in $\TT$ of the respective sets $2 \bN - 1$ and $2 \bN$.

By part (f) of
Lemma~\ref{L:core_properties}, (the isometry class of) the core $\Gamma(\TT)$
is constant almost surely.
The set $\Gamma(\mathbf{T})$ is the closure of the subtree spanned by
the set of attachment points $\{\Pi(i) : i \in \mathbb{N}\}$ and,
by part (a) of
Lemma~\ref{L:core_properties},
also the closure of the subtree spanned by the set of points
$\{\langle j,k \rangle : j,k \in \mathbb{N}, \, j \ne k\}$.
It is clear that 
$\Gamma(\TT^-) \subseteq \Gamma(\TT)$
and
$\Gamma(\TT^+) \subseteq \Gamma(\TT)$.
It follows from Remark 11.2 and the
second proof of Theorem 3.9(i) in
\cite{MR3806409} that almost surely for any $\delta > 0$ and $i \in \mathbb{N}$,
there exists $j,k \in 2 \mathbb{N} - 1$
(resp. $j,k \in 2 \mathbb{N}$)
with $d(\Pi(i), \langle j,k \rangle) < \delta$ and hence
$\Gamma(\mathbf{T}^-) = \Gamma(\mathbf{T}^+) = \Gamma(\mathbf{T})$.

For the benefit of the reader, we sketch the
argument from \cite{MR3806409} in our notation.  Fix $\epsilon > 0$ and a
deterministic sequence $0 < h_1^{(\epsilon)}  < h_2^{(\epsilon)} < \ldots \uparrow \infty$
 such that $h_1^{(\epsilon)} < \epsilon$,  $h_{n+1}^{(\epsilon)} - h_n^{(\epsilon)} < \epsilon$,
and
$\mathbb{P}\{d(i, \Pi(i)) = h_n^{(\epsilon)}\} = 0$
for all $i,n \in \mathbb{N}$.  Set $h_0^{(\epsilon)} = 0$.
Define an exchangeable equivalence relation
$\sim^\epsilon$ on $\mathbb{N}$ by declaring that $i \sim^\epsilon j$,
$i,j \in \mathbb{N}$, $i \ne j$, if
$d(i, \Pi(i)), d(j, \Pi(j)), d(i,\langle i,j \rangle) = d(j,\langle i,j \rangle)
\in [h_{n-1}^{(\epsilon)},h_n^{(\epsilon)})$ for some $n \in \mathbb{N}$.  Note
that if $i \sim^\epsilon j$, $i,j \in \mathbb{N}$, $i \ne j$, then
\[
d(\Pi(i), \langle i,j \rangle) \vee d(\Pi(j), \langle i,j \rangle) < \epsilon.
\]
The crucial observation, stated and proved as (10.7) of \cite{MR3806409}, 
is that almost surely the
exchangeable equivalence relation $\sim^\epsilon$ does not have any singleton
equivalence classes.  It then follows from Kingman's paintbox construction
of exchangeable equivalence relations that almost surely for any $i \in \mathbb{N}$
there exists $j,k \in 2 \mathbb{N} - 1$ (respectively, $j,k \in 2 \mathbb{N}$)
with $i,j,k$ distinct and $i \sim^\epsilon j \sim^\epsilon k$, and hence
\[
d(\Pi(i), \langle j,k \rangle)
\le
d(\Pi(i), \langle i,j\rangle) 
+ d(\langle i,j\rangle, \Pi(j)) 
+ d(\Pi(j), \langle j,k \rangle)
< 3 \epsilon.
\]


Let $\Pi_-$ and $\Pi_+$ be the analogues
of $\Pi$ for $(\TT^-, <_-)$ and $(\TT^+, <_+)$.  
For $i \in \bN$ we have that $\Pi_+(i)$, the closest point
in $\Gamma(\TT^+)$ to $i$ (where we stress that $i$ labels an element of
$\TT^+$), is an element of $\Gamma(\TT^-) = \Gamma(\TT^+)$.  
It follows from the exchangeability
inherent in our construction that $(\Pi_+(i))_{i \in \bN}$ 
is an exchangeable
sequence of random elements of $\Gamma(\TT^-)$.  By our
standing ergodicity assumption and de Finetti's theorem, the random elements
in this sequence are independent and identically distributed,
and it is a consequence of part (d) of Lemma~\ref{L:core_properties} that
their common distribution is, prefiguring the notation in
the statement of Proposition~\ref{P:up_to_left_right} below, 
a diffuse probability measure $\mu$ on
the $\bR$-tree $\SSS := \Gamma(\TT)$ that is contained in $\TT$ and rooted in $\theta := \rho$.
The probability measure $\mu$ and the $\bR$-tree $\SSS$
are the objects addressed in this
section's title.  

For $i,j \in \bN$ with $i \ne j$, part (e) of
Lemma~\ref{L:core_properties} gives that
$\langle i,j \rangle_+$ is
the most recent common ancestor of
$\Pi_+(i)$ and $\Pi_+(j)$ in the partial order $<_+$.
Moreover, $\langle i,j \rangle_+ \le_+ \Pi_+(p)$
if and only if 
$[\rho,\Pi_+(i)] \cap [\rho,\Pi_+(j)]
\subseteq
[\rho,\Pi_+(i)] \cap [\rho,\Pi_+(p)]$
or
$[\rho,\Pi_+(i)] \cap [\rho,\Pi_+(j)]
\subseteq
[\rho,\Pi_+(j)] \cap [\rho,\Pi_+(p)]$,
where
$[\rho,x]$ is the geodesic segment between $\rho$ and
$x$ in $\Gamma(\TT)=\Gamma(\TT^-)=\Gamma(\TT^+)$, and so
if we write $d_+(i,j)$ for the distance between
$i,j \in \bN$, $i \ne j$, in $\TT^+$, we have from part (e)
of Lemma~\ref{L:core_properties} that
\[
\begin{split}
d_+(i,j)
=
\lim_{n \to \infty} 
\frac{1}{n} \sum_{p=1}^n 
\ind\Bigl(
& \{[\rho,\Pi_+(i)] \cap [\rho,\Pi_+(j)]
\subseteq
[\rho,\Pi_+(i)] \cap [\rho,\Pi_+(p)]\} \\
& \quad \cup
\{[\rho,\Pi_+(i)] \cap [\rho,\Pi_+(j)]
\subseteq
[\rho,\Pi_+(j)] \cap [\rho,\Pi_+(p)]\}
\Bigr). \\
\end{split}
\]

Because $(\TT^+, <_+)$ has the same distribution as
$(\TT, <)$, we have established the following result.

\begin{proposition}
\label{P:up_to_left_right}
Suppose that $\equiv$, $\langle \cdot, \cdot \rangle$
and $<$ are the equivalence relation on $\bN \times \bN$,
the equivalence classes and the partial order on
those equivalence classes arising from 
an ergodic exchangeable random didendritic system
(equivalently, from the labeled version of an
extremal infinite R\'emy bridge ).
There is a complete separable  $\bR$-tree $\SSS$,
a point $\theta \in \SSS$, and a diffuse probability measure
$\mu$ on $\SSS$ such that the following hold.
Let $\xi_1, \xi_2, \ldots$ be i.i.d. random elements of $\SSS$
with common distribution $\mu$.  Define a random equivalence relation
$\equiv_{\#}$ on $\bN \times \bN$ by declaring that 
$(i,i) \equiv_{\#} (k,\ell)$ if and only if $(i,i) = (k,\ell)$,
and 
$(i,j) \equiv_{\#} (k,\ell)$ for $i \ne j$ and $k \ne \ell$
if and only if
 $[\theta,\xi_i] \cap [\theta,\xi_j] = [\theta,\xi_k] \cap [\theta,\xi_\ell]$, where
$[\theta,x]$ is the geodesic segment between $\theta$ and $x$, in $\SSS$.  Denote the
equivalence class containing $(i,j) \in \bN \times \bN$ by 
$\langle i,j \rangle_{\#}$.  Define a partial order $<_{\#}$ on the
set of equivalence classes by declaring that
$\langle i,j \rangle_{\#} <_{\#} \langle i,i \rangle_{\#}$
for all $i \ne j$ and 
that $\langle i,j \rangle_{\#} <_{\#} \langle k,\ell \rangle_{\#}$
for $i \ne j$ and $k \ne \ell$ if 
$[\theta,\xi_i] \cap [\theta,\xi_j] 
\subsetneq
[\theta,\xi_k] \cap [\theta,\xi_\ell]$.
The object $(\equiv_{\#}, \langle \cdot, \cdot \rangle_{\#}, <_{\#})$ 
has the same
distribution as $(\equiv, \langle \cdot, \cdot \rangle, <)$.
\end{proposition}

\section{Distinguishing between left and right}\label{S:leftright}

Throughout this section, let $(\TT,d)$ be the $\bR$-tree
constructed in Section~\ref{S:R_tree} from 
an ergodic exchangeable random didendritic system
$\DD = (\equiv, \, \langle \cdot, \cdot \rangle, \, <_L, \, <_R)$
(equivalently, from the labeled version $(\tilde T_n^\infty)_{n \in \bN}$ 
of an extremal infinite R\'emy bridge $(T_n^\infty)_{n \in \bN}$).

Let $\SSS$, $\theta$ and $\mu$ be the objects described in
Proposition~\ref{P:up_to_left_right}.  Thus,
$\SSS$ is a complete separable  $\bR$-tree,
$\theta$ is an element of $\SSS$, and $\mu$ is
a diffuse probability measure on $\SSS$.  Further, let
$\xi_1, \xi_2, \ldots$ be i.i.d. random elements of $\SSS$
with common distribution $\mu$.  We may suppose that 
$(i,i) \equiv (k,\ell)$ if and only if $(i,i) = (k,\ell)$
and that
$(i,j) \equiv (k,\ell)$ for $i \ne j$ and $k \ne \ell$
if and only 
if $[\theta,\xi_i] \cap [\theta,\xi_j] = [\theta,\xi_k] \cap [\theta,\xi_\ell]$, where we recall that
$[\theta,x]$ is the geodesic segment between $\theta$ and $x$
in $\SSS$.

Recall from Lemma~\ref{L:partial_order_plus_left-right} that 
if we know the equivalence relation $\equiv$ 
and the partial order $<$ of the didendritic system 
$\DD$, then
the partial orders $<_L$ and $<_R$ (and hence the didendritic system)
is uniquely determined by the specification for all distinct $i,j \in \bN$
whether
$\langle i,j \rangle <_L i$ and $\langle i,j \rangle <_R j$ 
  or
$\langle i,j \rangle <_R i$ and $\langle i,j \rangle <_L j$.

Put
\[
J_{ij} := \ind\{\langle i,j \rangle <_L i, 
\, \langle i,j \rangle <_R j\}
\]
for $(i,j) \in \bN \times \bN \setminus \delta$,
where $\delta := \{(k,k) : k \in \bN\}$.
Note for all $(i,j) \in \bN \times \bN \setminus \delta$
 that $J_{ij} = 1$ if and only if 
$J_{ji}=0$.  


It follows from the exchangeability and ergodicity
of $\DD$ that the random array $J$ is jointly exchangeable
and ergodic and, indeed, the random array
$(\xi_i, \xi_j, J_{ij})_{(i,j) \in \bN \times \bN \setminus \delta}$
is also jointly exchangeable and ergodic.  Therefore,
by the  Aldous--Hoover--Kallenberg theory of such random arrays, we may suppose that
on some extension of our underlying probability space
there exist i.i.d. random variables $(U_{i})_{i \in \bN}$,
and $(U_{ij})_{i,j \in \bN, \, i < j}$ that are uniform on $[0,1]$ 
and a function $F:(\mathbf S\times [0,1])^2 \times [0,1] \to \{0,1\}$ such that 
\[
J_{ij} 
= F(\xi_i,U_i,\xi_j,U_j,U_{ij}),
\] 
where $U_{ij} = U_{ji}$ for $i>j$ (here
$<$ is the usual order on $\bN$)
(see \cite[Theorem 7.22, Lemma 7.35]{MR2161313}).  
Because $J_{ij} = 1 -J_{ji}$,
the function $F$
has the property $F(y,v,x,u,w) = 1-F(x,u,y,v,w)$.

If $i,j,k \in \bN$ are distinct we have
$J_{ij} = J_{ik}$ on the event 
$\{\langle i,j \rangle = \langle i,k \rangle < \langle j,k \rangle\}
= \{[\theta, \xi_i] \cap [\theta, \xi_j] = [\theta, \xi_i] \cap [\theta, \xi_k]
\subsetneq [\theta, \xi_j] \cap [\theta, \xi_k]\}$.  That is,
$F(\xi_i, U_i, \xi_j, U_j, U_{ij}) 
= J_{ij} = J_{ik} 
= F(\xi_i, U_i, \xi_k, U_j, U_{ik})
\in \{0,1\}$ on the latter event.  Similarly,
$1-F(\xi_i, U_i, \xi_j, U_j, U_{ij}) 
= J_{ji} = J_{jk} 
= F(\xi_j, U_j, \xi_k, U_k, U_{jk})
\in \{0,1\}$ on the event
$\{\langle j,i \rangle = \langle j,k \rangle < \langle i,k \rangle\}
= \{[\theta, \xi_j] \cap [\theta, \xi_i] = [\theta, \xi_j] \cap [\theta, \xi_k]
\subsetneq [\theta, \xi_i] \cap [\theta, \xi_k]\}$.

By Lemma~\ref{L:genuine_metric},
\[
\bP\{\langle k,i \rangle = \langle k,j \rangle < \langle i,j \rangle 
\, \forall k \notin \{i,j\}\} = 0,
\]
and Lemma \ref{Lemma8_1} below then gives that 
\begin{equation}\label{noUij}
J_{ij} = W(\xi_i, U_i, \xi_j, U_j)
\end{equation}
almost surely for some Borel function $W: (\SSS \times [0,1])^2\to \{0,1\}$.

The intuition for \eqref{noUij} being true is firstly that if we had 
$\langle k,i \rangle = \langle k,j \rangle < \langle i,j \rangle$ 
for all $k \notin \{i,j\}$, so that $\{i,j\}$ is a
cherry, then there could be a need to use the randomization provided
by $U_{ij}$ to build $J_{ij}$; that is,
$F(\xi_i, U_i, \xi_j, U_j, U_{ij})$ could depend on $U_{ij}$.  
However, cherries do not occur with positive probability because 
of Lemma ~\ref{L:genuine_metric}.  Moreover, on the event where 
$\langle i,j \rangle = \langle i,k \rangle < \langle j,k \rangle$ for 
at least one, and hence infinitely many, $k \notin \{i,j\}$, 
it follows that 
$F(\xi_i, U_i, \xi_j, U_j, U_{ij}) = F(\xi_i, U_i, \xi_k, U_k, U_{ik})$
for all such $k$, which can't happen if 
$(x,u,y,v,w) \mapsto F(x,u,y,v,w)$ has a genuine functional dependence
on $(v,w)$, and hence only the extra randomization provided by $U_i$
(rather than that provided by $U_j$ and $U_{ij}$)
might be needed to build $J_{ij}$. Similarly, on the event where 
$\langle i,j \rangle = \langle j,k \rangle < \langle i,k \rangle$ for 
at least one, and hence infinitely many, $k \notin \{i,j\}$, 
it follows that 
$F(\xi_i, U_i, \xi_j, U_j, U_{ij}) = F(\xi_k, U_k, \xi_j, U_j, U_{kj})$
for all such $k$ and only the extra randomization provided by $U_j$
(rather than that provided by $U_i$ and $U_{ij}$)
might be needed to build $J_{ij}$.  Lastly, on the event where
$\langle i,k \rangle > \langle i,j \rangle = \langle k,\ell \rangle = \langle i, 
\ell \rangle = \langle k,j \rangle < \langle j,\ell \rangle$ for 
at least one, and hence infinitely many, pairs $k,\ell \notin \{i,j\}$,
it follows that
$F(\xi_i, U_i, \xi_j, U_j, U_{ij}) = F(\xi_k, U_k, \xi_\ell, U_\ell, U_{k,\ell})$
for all such pairs $k,\ell$ and there is no need for the extra randomization
provided by $U_i$, $U_j$ or $U_{ij}$
to build $J_{ij}$. 

We now give the promised formal argument for \eqref{noUij}. 
\begin{lemma}\label{Lemma8_1}
Consider on some probability space $(\Omega, \cF, \bP)$ 
independent random elements $X_1$, $X_2$, $X_3$  
of some Borel space $(D, \cD)$
and $Y_{12}$, $Y_{13}$, $Y_{23}$
of some Borel space $(E, \cE)$.   
Suppose that $X_1$, $X_2$, $X_3$ have the
same diffuse probability distribution $\alpha$
and that $Y_{12}$, $Y_{13}$, $Y_{23}$ have the same
diffuse probability distribution $\beta$.  Write $B$ for the subset of
$D^3$ that consists of triplets with distinct entries.
Given an ordered listing $i,j,k$ of $\{1,2,3\}$
and a set $C \subseteq B$, put 
$C_{ijk}:= \{(x_1, x_2, x_3) \in B : (x_i, x_j, x_k) \in C\}$.
Suppose that there is a set $A \in \cD^3$
such  
\begin{itemize}
\item
$A_{123} = A_{132}$,
\item
$A_{213} = A_{231}$,
\item
$A_{312} = A_{321}$,
\item
these $3$ sets are pairwise disjoint and their
union is $B$.
\end{itemize}
Suppose further that
\[
\begin{split}
& 
\alpha^{\otimes 2}\{(x_1,x_2) \in D^2 : 
\alpha\{x_3 \in D : (x_1,x_2,x_3) \in B \setminus A_{312}\}
= 0
\} \\
& \quad = 
\alpha^{\otimes 2}\{(x_1,x_2) \in D^2 : 
\alpha\{x_3 \in D : (x_1,x_2,x_3) \in B \setminus A_{321}\}
= 0
\} \\
& \quad = 0. \\
\end{split}
\]
Consider a Borel function $H: D^2 \times E \to \{0,1\}$ such that
\begin{itemize}
\item
$H(X_1,X_2,Y_{12}) = H(X_1,X_3,Y_{13})$ 
on the event $\{(X_1, X_2, X_3) \in A\} = \{(X_1, X_3, X_2) \in A\}$
\item
$H(X_1,X_2,Y_{12}) = 1-H(X_2,X_3,Y_{23})$
on the event $\{(X_2, X_1, X_3) \in A\} = \{(X_2, X_3, X_1) \in A\}$.
\end{itemize}
Then there exists a Borel function $K: D^2 \to \{0,1\}$
such that $H(X_1,X_2,Y_{12}) = K(X_1, X_2)$ almost surely.
\end{lemma}

\begin{proof}
For $(x_1, x_2, y_{12}) \in D^2 \times E$ with $x_1 \ne x_2$ we have
\[
\begin{split}
& \int_{D \times E} H(x_1, x_2, y_{12}) \ind_{A_{123}}(x_1,x_2,x_3) \, \alpha \otimes \beta(d(x_3, y_{13})) \\
& \quad = 
\int_{D \times E} H(x_1, x_3, y_{13}) \ind_{A_{123}}(x_1,x_2,x_3) \, \alpha \otimes \beta(d(x_3, y_{13})) \\
\end{split}
\]
and
\[
\begin{split}
& \int_{D \times E} H(x_1, x_2, y_{12}) \ind_{A_{213}}(x_1,x_2,x_3) \, \alpha \otimes \beta(d(x_3, y_{23})) \\
& \quad = 
\int_{D \times E} (1 - H(x_2, x_3, y_{23}) \ind_{A_{213}}(x_1,x_2,x_3) \, \alpha \otimes \beta(d(x_3, y_{23})). \\
\end{split}
\]
 Thus,
\[
\begin{split}
& H(x_1, x_2, y_{12}) \int_D \ind_{A_{123}}(x_1,x_2,x_3) \, \alpha(dx_3) \\
& \quad =
\int_{D \times E} H(x_1, x_3, y_{13}) \ind_{A_{123}}(x_1,x_2,x_3) \, \alpha \otimes \beta(d(x_3, y_{13})) \\
\end{split}
\]
and
\[
\begin{split}
& H(x_1, x_2, y_{12}) \int_D \ind_{A_{213}}(x_1,x_2,x_3) \, \alpha(dx_3) \\
& \quad =
\int_{D \times E} (1 - H(x_2, x_3, y_{23}) \ind_{A_{213}}(x_1,x_2,x_3) \, \alpha \otimes \beta(d(x_3, y_{23})).  \\
\end{split}
\]
The last two equations specify the value of
$H(x_1, x_2, y_{12})$ as a quantity depending on
$(x_1, x_2)$ alone
except for those pairs $(x_1, x_2)$ such that
\[
\int_D \ind_{A_{123}}(x_1,x_2,x_3) \, \alpha(dx_3)
=
\int_D \ind_{A_{213}}(x_1,x_2,x_3) \, \alpha(dx_3)
=
0,
\]
but the set of such pairs has zero $\alpha^{\otimes 2}$-measure by assumption.
\end{proof}

The function $W$ is not arbitrary; it must satisfy 
some obvious consistency conditions.
For example, for distinct $i,j,k \in \bN$ when 
$\langle i,j \rangle 
= \langle i,k \rangle 
< \langle j,k \rangle$,
it must be the case that
$\langle i,j \rangle <_L i$ if and only if $\langle i,k \rangle <_L i$,
and this translates into the requirement that when
$[\rho, \xi_i] \cap [\rho,\xi_j] 
= [\rho, \xi_i] \cap [\rho, \xi_k] 
\subsetneq 
[\rho, \xi_j] \cap [\rho, \xi_k]$ 
it must be the case that
$W(\xi_i,U_i,\xi_j,U_j) = 1$ if and only if $W(\xi_i,U_i,\xi_k,U_k) = 1$,
which in turn translates into the requirement that for
$(\mu \otimes \lambda)^{\otimes 3}$-a.e. 
$((x,u),(y,v),(z,w)) \in (\SSS \times [0,1])^3$ when
$[\theta, x] \cap [\theta, y] = [\theta, x] \cap [\theta, z]
\subsetneq [\theta, y] \cap [\theta, z]$
it must be the case that $W((x,u),(y,v)) = W((x,u),(z,w))$.

The next result specifies fully these consistency conditions 
and combines, without the need for significant further argument,
the development leading to
Proposition~\ref{P:up_to_left_right} 
with the considerations so far in this section
about resolving ``left--vs--right''
to give a complete characterization of
the family of ergodic exchangeable random didendritic systems and hence
a concrete description of the family of extremal infinite R\'emy bridges.
The result is thus an explicit determination of the Doob--Martin boundary of the R\'emy chain.
The only point that deserves some added explanation is the claim
of ergodicity in the statement of the converse; however, this
follows from the observation made in Remark~\ref{R:jointly_exchangeable}
that ergodicity of an exchangeable random didendritic system (on $\bN$) is equivalent
to the independence of the exchangeable random didendritic systems it induces on
disjoint subset of $\bN$.

\begin{theorem}
\label{T:main}
Consider a complete separable  $\bR$-tree $\SSS$,
a point $\theta \in \SSS$, a diffuse probability measure
$\mu$ on $\SSS$, and a Borel function
$W: (\SSS \times [0,1])^2 \to \{0,1\}$. Let $\lambda$
be Lebesgue measure on $[0,1]$.
Suppose that the following hold.
\begin{itemize}
\item
For $\mu^{\otimes 3}$-a.e. $(x,y,z) \in \SSS^3$,
two of the three geodesic segments
$[\theta, x] \cap [\theta, y]$, $[\theta, x] \cap [\theta, z]$,
$[\theta, y] \cap [\theta, z]$ are equal and these two are
strictly contained in the third.
\item
For $(\mu \otimes \lambda)^{\otimes 3}$-a.e. 
$((x,u),(y,v),(z,w)) \in (\SSS \times [0,1])^3$,
$[\theta, x] \cap [\theta, y] = [\theta, x] \cap [\theta, z]
\subsetneq [\theta, y] \cap [\theta, z]$
implies that $W((x,u),(y,v)) = W((x,u),(z,w))$.
\item
For $(\mu \otimes \lambda)^{\otimes 2}$-a.e. 
$((x,u),(y,v)) \in (\SSS \times [0,1])^2$, 
$W((x,u),(y,v)) = 1 - W((y,v),(x,u))$.
\end{itemize}
Let $(\xi_1, U_1), (\xi_2, U_2), \ldots$ be i.i.d. random elements of 
$\SSS \times [0,1]$ with common distribution $\mu \otimes \lambda$.  
There is an ergodic exchangeable random didendritic system 
$(\equiv, \, \langle \cdot, \cdot \rangle, \, <_L, \, <_R)$
defined as follows. The random equivalence relation
$\equiv$ on $\bN \times \bN$ is given by declaring that 
\[
(i,i) \equiv (k,\ell) \quad \text{if and only if} \quad (i,i) = (k,\ell),
\]
and 
\[
(i,j) \equiv (k,\ell), i \ne j, k \ne \ell, \quad \text{if and only if} \quad
[\theta,\xi_i] \cap [\theta,\xi_j] = [\theta,\xi_k] \cap [\theta,\xi_\ell].
\]
The random partial orders $<_L$ and $<_R$
on the corresponding set of equivalence classes
$\{\langle i,j \rangle : i,j \in \bN\}$ are specified by declaring
for $i,j \in \bN$, $i \ne j$, that
\[
\langle i,j \rangle <_L  i
\; \text{\&} \langle i,j \rangle <_R  j
\quad \text{if and only if} \quad
W(\xi_i, U_i, \xi_j, U_j) = 1.
\]

Conversely, any ergodic exchangeable random didendritic system
has the same probability distribution as one constructed in this manner
for $\SSS, \theta, \mu, W$ satisfying the assumptions above.
\end{theorem}

\begin{example}
Recall the infinite R\'emy bridge of Example~\ref{E:coin_tossing_bridge}.
We know from Example~\ref{E:coin_tossing_tree} that we may take
\begin{itemize}
\item
$\SSS$ to be $[0,1]$ equipped with the usual metric,
\item
$\theta$ to be the point $0 \in [0,1]$,
\item
$\mu$ to be Lebesgue measure on $[0,1]$.
\end{itemize}
We may then take
\[
W(x,u,y,v) 
= 
\begin{cases}
1,& \quad \text{if $x < y$ and $u < \frac{1}{2}$}, \\
0,& \quad \text{if $x < y$ and $u > \frac{1}{2}$}, \\
1,& \quad \text{if $y < x$ and $v < \frac{1}{2}$}, \\
0,& \quad \text{if $y < x$ and $v > \frac{1}{2}$}, \\
0,& \quad \text{otherwise}.
\end{cases}
\]
\end{example}

\begin{example} 
Now consider the infinite R\'emy
bridge of Example~\ref{E:MBfulltree}. Here $\SSS$
is the $\bR$-tree $\TT$ of Example~\ref{E:MBfulltree_real}.
The leaves of $\SSS$ are in a bijective correspondence with
$\{0,1\}^\infty$ and $\mu$ may be identified with the fair
coin-tossing measure $\kappa$ on $\{0,1\}^\infty$
introduced in Example~\ref{E:complete_binary_embedding}.
There is no need for genuine randomization in this case.  Indeed, 
for $\mu^{\otimes 2}$-a.e. $(\xi_1, \xi_2)$ we have 
either
$W(\xi_1,u_1,\xi_2,u_2) = 0$ for
$\lambda^{\otimes 2}$-a.e. $(u_1,u_2)$
or
$W(\xi_1,u_1,\xi_2,u_2) = 1$ for
$\lambda^{\otimes 2}$-a.e. $(u_1,u_2)$.
That is, we can just take the $\bR$--tree $\SSS$ and augment it with
{\em deterministic} left--right choices because in this case for any $i,j$ we have
$\langle i,j \rangle = \langle k, \ell \rangle$
for infinitely many other $k,\ell$.
The resulting representation of the infinite R\'emy bridge
coincides with the one given in Example~\ref{E:MBfulltree}.
\end{example}

\begin{remark}
As we remarked in the Introduction, the distribution of the limit 
in the Doob--Martin topology of the R\'emy chain (that is,
the probability measure that appears in the Poisson boundary
of the R\'emy chain) is concentrated on points that
can be represented in terms of ensembles
$\SSS, \theta, \mu, W$ such that $W$ takes values in $\{0,1\}$.
\end{remark}

\begin{remark}
Theorem~\ref{T:main} gives a concrete characterization of
the family of ergodic exchangeable random didendritic systems
or, equivalently, the family of extremal infinite R\'emy bridges.  
Consequently, it gives an explicit description of
the points in the  Doob--Martin boundary of the R\'emy chain.
Of course, the ingredients
appearing in the representation afforded by the result are 
not unique.  Also, the Doob--Martin boundary is not just
a set: it carries a metrizable topological structure.  
However, a sequence of representations 
corresponds to a convergent sequence of boundary points if and only if 
the restrictions of the associated exchangeable random didendritic systems
to finite subsets of $\bN$ converge in distribution.
That is, a sequence of representations 
corresponds to a convergent sequence of boundary points if and only if
for all $m$ the sequence of random binary trees built by sampling
$m+1$ points according to the associated sampling measure and determining left--versus--right
orderings using extra randomness as necessary converges in distribution.  
\end{remark}

\section{Corrigenda}
\label{S:corrigenda}
\subsection{Remark~\ref{R:extremal_vs_ergodic} and proof of Lemma~\ref{L:tail_relation}}
\label{S9_1}
Remark~\ref{R:extremal_vs_ergodic}  says that 
Proposition~\ref{P:extremal_vs_ergodic}  can also be proved along the lines of the proof of 
Lemma~\ref{L:tail_relation}. As pointed out to us by Julian Gerstenberg, 
the proof of Lemma~\ref{L:tail_relation} is incorrect. The proof of Proposition~\ref{P:extremal_vs_ergodic} which follows the statement of that result does not depend on Lemma~\ref{L:tail_relation}. Also,  notwithstanding the incorrect proof of Lemma~\ref{L:tail_relation}, its assertion {\em is} correct, see 
the remark following Corollary 2.8 in \cite{Ger18}. Nonetheless, we want
to make it clear where the incorrectness in our proof of Lemma~\ref{L:tail_relation} lies.

Our argument for proving Lemma~\ref{L:tail_relation} proceeded as follows.  Let
$\mathcal F = \sigma\{L_p : p \in \mathbb N\}$, $\mathcal G_m = \sigma\{T_n^\infty :n \ge m\}$, $m \in \mathbb N$, and
$\mathcal G_\infty = \bigcap_{m \in \mathbb N} \mathcal G_m$.  We claimed that
\begin{equation} \label{interchange}
\bigcap_{m \in \mathbb N} (\mathcal F \vee \mathcal G_m) = \mathcal F \vee \mathcal G_\infty \, \quad \mathbb P\mbox{-a.s.}
\end{equation}
as a consequence of the implication (d) $\Longrightarrow$ (a) in \cite{MR699981}.  Condition (d) in \cite{MR699981}
says that there is a countably generated $\sigma$-field $\mathcal H$ such that if $\mathbb P^{\mathcal F}$ is a regular conditional
probability given $\mathcal F$, then 
\[
\mathcal G_\infty = \mathcal H, \quad \mathbb P^{\mathcal F}(\omega, \cdot)\text{-a.s.} \; \text{for all} \; \omega \in \Omega.
\]
In our attempt to establish this last conclusion, 
we first argued that since $\mathcal F$ and $\mathcal G_\infty$ are independent 
one can conclude that $\mathbb P^{\mathcal F}(\omega, \cdot)$ restricted to $\mathcal G_\infty$ coincides with $\mathbb P$
restricted to $\mathcal G_\infty$ for all $\omega \in \Omega$.

To see that this reasoning is incorrect, consider the simpler situation where $\mathbb P$ is fair coin-tossing measure on
$\Omega = \{(\omega_n)_{n \in \mathbb N} \in \{0,1\}^\infty\}$, $\mathcal F$ is the usual product $\sigma$-field, and
$\mathcal G_m = \sigma\{\omega \mapsto \omega_n : n \ge m\}$.  In this case, $\mathcal F$ and $\mathcal G_\infty$ are independent
because $\mathcal G_\infty$ is $\mathbb P$-trivial.  However, for the  $\mathcal G_\infty$-measurable sets
\[ [\omega] := \bigcup_{m\ge1}  \{(\alpha_n)\in \{0,1\}^\infty: \alpha_j=\omega_j \mbox{ for all } j\ge m\}\]
we have
 $\mathbb P^{\mathcal F}(\omega, [\omega]) = 1$ for all $\omega \in \Omega$,  whereas $\mathbb P([\omega]) =0$ for all $\omega \in \Omega$. 
 (While in this example the relation \eqref{interchange} is clearly true, 
\cite[p. 93]{MR699981}) provides an example  of a $\sigma$-field~$\mathcal F$ and a sequence of $\sigma$-fields $\mathcal G_1 \supseteq \mathcal G_2 \supseteq \ldots$ with a.s. trivial tail-$\sigma$-field $\mathcal G_\infty$ in which \eqref{interchange} does not hold. See also \cite{MR3630297} for related  recent developments.


\subsection{Amendment of Definition~\ref{D:didendritic_properties}}\label{S9_2} In Definition~\ref{D:didendritic_properties} we introduced didendritic system as an ensemble consisting of an equivalence relation on $\mathbb{N} \times \mathbb{N}$, a set of equivalence classes $\{\langle i,j \rangle : i,j \in \mathbb{N}\}$ where
$\langle i,j \rangle$ is the equivalence class containing the pair $(i,j)$, and
three partial orders $<_L$, $<_R$, and $<$ on $\{\langle i,j \rangle : i,j \in \mathbb{N}\}$.  The objects in this ensemble are required to satisfy a certain set of axioms.  If one restricts the equivalence
relation to $\mathcal N \times \mathcal N$ for some finite subset $\mathcal N$ of $\mathbb{N}$ and restricts the equivalence classes and partial orders similarly, then one obtains an object that satisfies
the same axioms  and could be called a didendritic system with the finite label set $\mathcal N$.

The further development in the paper makes essential use of the fact that  
 any didendritic system with the finite label set $\mathcal N$ has a faithful representation as a finite full binary tree (that is, a finite rooted ordered tree in which
each vertex has out-degree $2$ or $0$) with its leaves labeled by $\mathcal N$ and that all leaf-labeled finite full binary trees arise this way: for example, the relation
$\langle i,j \rangle < \langle k, \ell \rangle$ translates to the requirement that the path in the tree from the root to the most recent common ancestor of the leaves labeled $k$ and $\ell$
passes through the most recent common ancestor of the leaves labeled $i$ and $j$.

In the first sentence of Remark~\ref{R:didendritic_to_bridge} it is claimed that such a one-to-one correspondence between didendritic systems with finite label sets and  finite leaf-labeled full binary trees holds with the axioms from Definition~\ref{D:didendritic_properties}. The following example shows that this claim does not hold.
\begin{example}
Consider the equivalence relation on $\{1,2,3\} \times \{1,2,3\}$ defined by 
\[(h,i)\equiv (j,k) \mbox{ if an only if } h=j \mbox{ and } i=k, \mbox{ or } h=k \mbox{ and } i=j.\]
For $(i, j) \in \{(1,2),(1,3), (2,3)\} $ say that $\langle i,j\rangle <_L \langle i,i\rangle$ and $\langle i,j\rangle <_R \langle j,j\rangle$. Moreover, say that $\langle h,i \rangle < \langle j,k \rangle$ if $\langle h,i \rangle <_L \langle j,k \rangle$ or $\langle h,i \rangle <_R \langle j,k \rangle$. Then $(\{1,2,3\}, \equiv, <_L, <_R,<)$ meets the axioms of Definition ~\ref{D:didendritic_properties}. However,  this system  does not correspond to a binary tree with $3$~leaves, since this would require exactly $5$ equivalence classes (each one corresponding to a vertex of the tree), whereas this system has $3+3 = 6$ equivalence classes. In particular, none of the $6$ equivalence classes can be distinguished as a root in any meaningful way.
\end{example}
In \cite{EW18}
an amended definition of a didendritic system is given which consists of the axioms equivalent to those in the Remark~\ref{R:didendritic_properties} plus the {\em ``Triplet property''} 

\medskip
\noindent
(C) {\em For distinct $i, j, k$, exactly one of
\[
\langle i,j\rangle = \langle i,k\rangle < \langle j,k\rangle
\]
\[
\langle j,k\rangle = \langle j,i\rangle < \langle k,i \rangle
\]
\[
\langle k,i\rangle = \langle k,j\rangle < \langle i,j \rangle
\]
holds.}

\medskip

Property (C) allows one to distinguish in a unique manner one of the equivalence classes of the didendritic system which is then mapped into  the root of the corresponding leaf-labeled full binary tree. Indeed, it is shown in \cite[Proposition~5.8]{EW18} that, with this new definition, any didendritic system with a finite label set may be thought of as a leaf-labeled finite full binary tree.
Since the later development in current paper depends solely
on the conclusion of Remark~\ref{R:didendritic_to_bridge} therein rather than the specific axioms that one uses to characterize a didendritic system, the arguments in the paper become correct once one adds axiom (C) to the definition there of a didendritic system.

\textbf{Acknowledgment}. 
Rudolf Gr\"ubel thanks Julian Gerstenberg for many stimulating discussions.
We thank Stephan Gufler for pointing out a gap in a previous version
of Section~\ref{S:probmeas} and showing us how to repair it. 

\newcommand{\etalchar}[1]{$^{#1}$}
\def\cprime{$'$}
\providecommand{\bysame}{\leavevmode\hbox to3em{\hrulefill}\thinspace}
\providecommand{\MR}{\relax\ifhmode\unskip\space\fi MR }
\providecommand{\MRhref}[2]{%
  \href{http://www.ams.org/mathscinet-getitem?mr=#1}{#2}
}
\providecommand{\href}[2]{#2}

\end{document}